\documentclass[12pt,a4paper]{amsart}
\usepackage{amsmath,amsfonts,amsthm,amssymb,nameref,mathtools}
\usepackage[top=3cm, bottom=2.5cm, left=2.5cm, right=2.5cm]{geometry} 
\usepackage{tikz} 
\usetikzlibrary{calc,automata} 
\usepackage{color} 
\usepackage{verbatim}
\usepackage{xfrac}
\usepackage{enumerate}
\usepackage{pinlabel}
\usepackage{graphicx}
\usepackage{caption}
\usepackage{subcaption}
\usepackage[pdftex]{hyperref}
\usepackage[all]{hypcap}
\usepackage{microtype} 

\title[Negative curvature in graphical small cancellation groups]{Negative curvature in graphical small cancellation groups}

\author[Arzhantseva]{Goulnara N. Arzhantseva}
 \email{goulnara.arzhantseva@univie.ac.at}
\address{Fakult\"at f\"ur Mathematik, Universit\"at Wien,
  Oskar-Morgenstern-Platz 1, 1090 Wien, \"{O}sterreich}
\author[Cashen]{Christopher H. Cashen}
 \email{christopher.cashen@univie.ac.at}
\address{Fakult\"at f\"ur Mathematik, Universit\"at Wien,
  Oskar-Morgenstern-Platz 1, 1090 Wien, \"{O}sterreich}
\author[Gruber]{Dominik Gruber}
 \email{dominik.gruber@math.ethz.ch}
\address{Department of Mathematics, ETH Z\"urich, R\"amistrasse 101, 8092 Z\"urich, Switzerland}
\author[Hume]{David Hume}
 \email{david.hume@maths.ox.ac.uk}
\address{Mathematical Institute, University of Oxford, Woodstock Road, OX2 6GG Oxford, UK}
\thanks{This work is supported by the ERC grant of Goulnara
  Arzhantseva ``ANALYTIC" no. 259527. The second author is supported by the
Austrian Science Fund (FWF):M1717-N25. The third author is supported
by the Swiss National Science Foundation Professorship FN
PP00P2-144681/1 of Laura Ciobanu. The fourth author is supported
by the ERC grant no. 278469 of Pierre--Emmanuel Caprace and the grant ANR-14-CE25-0004 ``GAMME'' of Damien Gaboriau and Romain Tessera.}

\hypersetup{
    pdftitle={Negative curvature in 
  graphical small cancellation groups},    
    pdfauthor={Arzhantseva, Cashen, Gruber, Hume},     
    pdfkeywords={graphical small cancellation, Morse quasi-geodesic,
  contracting projection, HEC property }, 
    colorlinks=true,       
    linkcolor=black,          
    citecolor=black,        
    filecolor=black,      
    urlcolor=black           
}

\newcommand{\N}	{\mathbb N}

\newcommand{\R}	{\mathbb R}

\newcommand{\Cay}	{\operatorname{Cay}}
\newcommand{\diam}	{\operatorname{diam}} 

\newcommand{\fgen}[1]{\left\langle #1 \right\rangle}

\newcommand{\fpres}[2]{\left\langle #1 \left| #2 \right.\right\rangle}

\newcommand{\from}{\colon\thinspace}
\newcommand{\bdry}{\partial}

\renewcommand{\setminus}{\smallsetminus}

\newlength{\mytriwidth}
\newlength{\myquadwidth}
\newlength{\myquadheight}
\newcommand{\rels}{\mathcal{R}}

\newcommand{\gens}{\mathcal{S}}

\newcommand{\asympleq}{\preceq}

\mathtoolsset{centercolon} 

\newcommand{\onto}{\twoheadrightarrow}
\renewcommand\thesubfigure{\Alph{subfigure}}
\newcommand\ifrac[2]{\sfrac{#1}{#2}} 

\newtheorem{thm}{Theorem}[section] 
\newtheorem{theorem}{Theorem}[section] 
\newtheorem{prop}{Proposition}[section] 
\newtheorem{proposition}{Proposition}[section] 
\newtheorem{lem}{Lemma} [section] 
\newtheorem{lemma}{Lemma}[section] 
\newtheorem{cor}{Corollary}[section] 
\newtheorem{corollary}{Corollary}[section] 
\newtheorem*{thm*} {Theorem} 
\newtheorem*{prop*}{Proposition}
\newtheorem*{lem*} {Lemma} 
\newtheorem*{cor*} {Corollary}
\theoremstyle{definition}

\newtheorem{defi}{Definition}[section] 
\newtheorem{definition}{Definition}[section] 
\newtheorem{example}{Example}[section] 
\newtheorem{question}{Question}[section] 
\newtheorem{remark}{Remark}[section] 
\newtheorem*{defi*}{Definition}
\newtheorem*{example*}{Example}
\newtheorem*{remark*}{Remark}
\newtheorem*{problem*}{Problem}
\newtheorem*{convention*}{Convention}
\newtheorem*{defi1*}{Definition I of the WPD element}
\newtheorem*{defi2*}{Definition II of the WPD element}

\makeatletter 
\let\c@question=\c@thm 
\makeatother
\makeatletter 
\let\c@theorem=\c@thm 
\makeatother
\makeatletter 
\let\c@proposition=\c@thm 
\makeatother
\makeatletter 
\let\c@corollary=\c@thm 
\makeatother
\makeatletter 
\let\c@definition=\c@thm 
\makeatother
\makeatletter 
\let\c@lemma=\c@thm 
\makeatother
\makeatletter 
\let\c@lem=\c@thm 
\makeatother
\makeatletter 
\let\c@prop=\c@thm 
\makeatother
\makeatletter 
\let\c@cor=\c@thm 
\makeatother
\makeatletter 
\let\c@defi=\c@thm 
\makeatother
\makeatletter 
\let\c@example=\c@thm 
\makeatother
\makeatletter 
\let\c@remark=\c@thm 
\makeatother
\def\makeautorefname#1#2{\expandafter\def\csname#1autorefname\endcsname{#2}}
\let\fullref\autoref

\makeautorefname{thm}{Theorem} 
\makeautorefname{lem}{Lemma} 
\makeautorefname{prop}{Proposition} 
\makeautorefname{cor}{Corollary} 
\makeautorefname{defi}{Definition}
\makeautorefname{theorem}{Theorem} 
\makeautorefname{lemma}{Lemma} 
\makeautorefname{proposition}{Proposition} 
\makeautorefname{corollary}{Corollary} 
\makeautorefname{definition}{Definition}
\makeautorefname{remark}{Remark}
\makeautorefname{example}{Example}
\makeautorefname{section}{Section}
\makeautorefname{subsection}{Section}
\makeautorefname{subsubsection}{Section}
\makeautorefname{question}{Question}

\begin{document}
\begin{abstract}
We use the interplay between combinatorial and coarse geometric
versions of negative curvature to investigate the geometry of infinitely presented
graphical $Gr'(\ifrac{1}{6})$ small cancellation groups.
In particular, we 
characterize their `contracting geodesics', which should be thought of
as the geodesics
that behave hyperbolically.

 We show that every degree of contraction can be achieved by a geodesic in a finitely
generated group. We construct the first example of a finitely generated group $G$
containing an element $g$ that is strongly
contracting with respect to one finite generating set of $G$ and not
strongly contracting with respect to another. In the case of classical $C'(\ifrac{1}{6})$ small cancellation groups we give
complete characterizations of geodesics that are Morse and
that are strongly contracting. 

We show that many graphical $Gr'(\ifrac{1}{6})$ small cancellation groups contain strongly
contracting elements and, in particular, are growth tight. We construct uncountably many quasi-isometry classes of finitely
generated, torsion-free groups in which every maximal cyclic subgroup
is hyperbolically embedded. These are the first examples of this kind
that are not subgroups of hyperbolic groups. 

In the course of our analysis we show that if the defining graph of a graphical $Gr'(\ifrac{1}{6})$ small cancellation group has finite components, then the elements of the group have translation lengths
that are rational and bounded away from zero.
 

\end{abstract}

\maketitle

\renewcommand{\thesubfigure}{(\arabic{subfigure})}

\section{Introduction}\label{sec:introduction}
Graphical small cancellation theory was introduced by Gromov as a powerful tool for constructing finitely generated groups with desired geometric and analytic properties \cite{Gro03}. Its key feature is that it produces infinite groups with prescribed subgraphs in their Cayley graphs. The group properties are thus derived from the combinatorial or asymptotic properties of the embedded subgraphs. Over the last two decades, graphical small cancellation theory has become an increasingly prominent and versatile tool of geometric group theory with a wide range of striking examples and applications. 

In this paper, we provide a thorough investigation of the hyperbolic-like geometry of graphical small cancellation constructions. Our theorems show that the constructed groups behave strongly like groups hyperbolic relative to their defining graphs. This contrasts the fact that in general they need not be Gromov hyperbolic or even relatively hyperbolic.
With this geometric analogy in mind, we produce a variety of concrete examples and determine the spectrum of negative curvature possible in the realm of finitely generated groups.

Graphical small cancellation theory was first used by Gromov \cite{Gro03} in the description of the groups now known as `Gromov monsters', which are finitely generated groups that contain sequences of expander graphs in their Cayley graphs. These monster groups do not coarsely embed into a Hilbert space, whence they are not coarsely amenable (i.e.\ do not have Yu's property A), and they give counterexamples to the Baum-Connes conjecture with coefficients \cite{HigLafSka02}. Graphical small cancellation theory is currently the only means of proving the existence of finitely generated groups with any of these three properties.

Since Gromov's initial impetus, the theory has gained in significance and variety of applications.
Indeed, the construction is very versatile: every countable group embeds into a 2-generated graphical small cancellation group~\cite[Example 1.13]{GruPhD}, and
graphical small cancellation theory has been
used, for instance, to give many new groups without the unique product property \cite{ArzSte14,Ste15,GruMarSte15}, the first examples of non-coarsely amenable groups with the Haagerup property \cite{ArzOsa14,Osa}, and new hyperbolic groups with Kazhdan's Property (T)~\cite{Gro03,Sil03,OllWis07}, 
as well as
to 
build a continuum of Gromov monsters~\cite{Hume14}, the first examples of finitely generated groups that do not coarsely embed into Hilbert space and yet do not contain a weakly embedded expander~\cite{ArzTes16},  and to analyze the Wirtinger presentations of prime alternating link groups~\cite{CuSh12}. 
Moreover, since the class of graphical small cancellation groups contains all classical small cancellation groups, it contains, for example, groups having no finite quotients \cite{Pri89}, groups with prescribed asymptotic cones \cite{ThoVel00,DruSap05}, and groups with exceptional divergence functions \cite{GruSis14}.

Our paper has two purposes.
The first purpose is to study geometric aspects of graphical small cancellation groups. An important property we focus on is the existence of subspaces that `behave like' subspaces of a negatively curved space. We quantify this phenomenon by considering \emph{contraction properties} of a subspace. Intuitively, this measures the asymptotic growth of closest point projections of metric balls to the subspace.
In a recent work \cite{ArzCasGrub}, we defined a general quantitative spectrum of contraction in arbitrary geodesic metric spaces and used it to produce new results on the interplay between contraction, divergence, and the property of being Morse. Our definitions generalize prior notions of contraction that have been instrumental in the study of numerous examples of finitely generated groups of current interest, such as mapping class groups \cite{MasMin99,Beh06,DucRaf09}, outer automorphism groups of free groups \cite{Min96,Alg11}, and, more generally, acylindrically hyperbolic groups \cite{DGO11}. In the present paper, we completely determine the contraction properties of geodesics in graphical small cancellation groups through their defining graphs. En route, we describe geodesic polygons and translation lengths in these groups. 

The second purpose of this paper is to detect the range of possible contracting behaviors in finitely generated groups. To this end, we use graphical small cancellation theory to show that every degree of contraction can be achieved by a geodesic in a suitable group. Moreover, we give the first examples of strongly contracting geodesics that are not preserved under quasi-isometries of groups. These results further establish the graphical small cancellation technique as a fundamental source of novel examples of finitely generated groups.

Our main technical result is a \emph{local-to-global} theorem for the contraction properties of geodesics in graphical small cancellation groups. It states that the contraction function of a geodesic is measured by its intersections with the defining graph. This confirms the analogy with relatively hyperbolic spaces and their peripheral subspaces.
Beyond the applications alluded to above, the theorem also enables us to prove the general result that many infinitely presented graphical small cancellation groups contain strongly contracting elements and, in particular, are growth tight, and to provide a characterization of Morse geodesics in classical $C'(\ifrac{1}{6})$ small cancellation groups. Furthermore, using the fact that strongly contracting elements give rise to hyperbolically embedded virtually cyclic subgroups, we produce the first examples of torsion-free groups in which every element is contained in a maximal virtually cyclic hyperbolically embedded subgroup but that are not subgroups of hyperbolic groups.

The proof of our local-to-global theorem rests on a meticulous analysis of the geometry of the Cayley graphs of graphical $Gr'(\ifrac{1}{6})$ small cancellation groups. 
In particular, we provide a complete classification of the geodesic quadrangles in the Cayley graphs of these groups, which is of independent interest and 
is new even for classical $C'(\ifrac{1}{6})$--groups.

The general tools that we establish have additional applications.
For instance, we show that in many infinitely presented graphical small cancellation groups, the translation lengths of infinite order elements are rational and bounded away from zero.
These tools will undoubtedly be useful towards further applications of this very interesting class of groups.

In the remainder of this introduction we explain the key concepts and main results of this paper in more detail, and give a brief overview of the proof of our local-to-global theorem.

\subsection*{Acknowledgements} 
A part of this work was developed during the program \textit{Measured group theory} held at the Erwin Schr\"{o}dinger Institute for Mathematics and Physics in Vienna in 2016. We thank its organizers and the institute for its hospitality.
We also thank the Isaac Newton Institute for Mathematical Sciences for support and hospitality during the program
\emph{Non-positive curvature: group actions and cohomology}. This work is partially supported by EPSRC Grant Number EP/K032208/1.

Finally, we thank the referees for their careful reading of the paper and their helpful comments.

\subsection{Contracting subspaces} In the following we shall assume that $X$ is a geodesic metric space such that for every closed $Y\subset X$ and $x\in X$, the set $\pi(x):=\{y\in Y\mid d(x,y)=d(x,Y)\}$ is non-empty. This is true for a proper space $X$, but also for a connected graph (i.e. a connected 1-dimensional CW-complex) $X$.
We call $\pi$ \emph{closest point projection to $Y$}. 
We do \emph{not} assume the sets $\pi(x)$ have uniformly bounded diameter.

\begin{definition}[Contracting]\label{def:contractingprojection}
Let $Y$ be a closed subspace of $X$, and denote by $\pi$ the closest point projection to $Y$. Let $\rho_1$ and $\rho_2$ be non-decreasing, eventually non-negative functions, with $\rho_1(r)\leqslant r$ and $\rho_1$ unbounded. We say that $Y$ is \emph{$(\rho_1,\rho_2)$--contracting} if the following conditions are satisfied for all $x,\,x'\in X$:
\begin{itemize}
\item $d(x,x')\leqslant\rho_1(d(x,Y))\implies\diam\pi(x)\cup\pi(x')\leqslant\rho_2(d(x,Y))$
\item $\lim_{r\to\infty}\frac{\rho_2(r)}{\rho_1(r)}=0$
\end{itemize}
If $\rho_1(r)=r$, then we say $Y$ is \emph{sublinearly contracting}, if $\rho_1(r)=r$ and $\rho_2(r)=C$ for some constant $C$ we say it is \emph{strongly contracting}, and if $\rho_1(r)=r/2$ and $\rho_2(r)=C$ for some constant $C$ we say it is \emph{semi-strongly contracting}
\end{definition}

We say a function $f$ is \emph{sublinear} if it is non-decreasing, eventually non-negative, and $\lim_{r\to\infty}\frac{f(r)}{r}=0$.

The most basic example of a strongly contracting subspace is  a geodesic in a tree or, more generally,  a geodesic in a $\delta$--hyperbolic space. The opposite extreme occurs in a Euclidean space where there are no contracting geodesics (for any choice of $\rho_1$ and $\rho_2$). 
The contrast between hyperbolic and Euclidean type behavior is evident in the following  well-known examples of contraction:
\begin{itemize}
\item A geodesic in a CAT(0) space is strongly contracting if and only
  if it is Morse \cite{BesFuj09, Sul14}.
\item A geodesic in a relatively hyperbolic space is strongly
  contracting if for every $C\geqslant 0$ there exists a $B\geqslant 0$ such that the
  geodesic spends at most time $B$ in the $C$--neighborhood of a
  peripheral subset \cite{Sis13projection}.
\end{itemize}

The common idea in these situations is that the given space has certain regions that are
not hyperbolic, but geodesics that avoid these non-hyperbolic regions
behave very much like geodesics in a hyperbolic space. Similar phenomena occur for pseudo-Anosov axes in the Teichm\"uller space of a hyperbolic surface and iwip axes in the Outer Space of the outer
automorphism group of a free group. Such axes avoid the `thin parts' of their respective spaces and therefore are strongly contracting \cite{Min96,Alg11}.

A version of semi-strong contraction, where the projection is not necessarily closest point projection,  occurs for pseudo-Anosov axes in the mapping class group of a hyperbolic surface \cite{MasMin99,Beh06,DucRaf09}.

\subsection{Local-to-global theorem} Given a directed graph $\Gamma$ whose edges are labelled by the elements of a set $\gens$, the group defined by $\Gamma$, denoted $G(\Gamma)$ is given by the presentation $\langle \gens \mid \text{labels of embedded cycles in }\Gamma\rangle$. The graphical $Gr'(\ifrac{1}{6})$ small cancellation condition, see \fullref{sec:smallcancellation}, is a combinatorial requirement on the labelling of $\Gamma$, whose key consequence is that the connected components of $\Gamma$ isometrically embed into $\Cay(G(\Gamma),\gens)$. In the case that $\Gamma$ is a disjoint union of cycle graphs labelled by a set of words $\rels$, the graphical $Gr'(\ifrac{1}{6})$--condition for $\Gamma$ corresponds to the classical $C'(\ifrac{1}{6})$--condition for $\rels$.

We show that, similar to the situations described above, geodesics in Cayley graphs of graphical $Gr'(\ifrac{1}{6})$ small cancellation groups behave like hyperbolic geodesics as long as they avoid the embedded components of the defining graph. In fact, a geodesic is as hyperbolic as its intersections with the
embedded components of $\Gamma$: 

\begin{thm*}[\fullref{theorem:strongcontraction}]
Let $\Gamma$ be a $Gr'(\ifrac{1}{6})$--labelled graph.
There exist $\rho_1'$ and $\rho_2'$ such that a geodesic $\alpha$ in
$X:=\Cay(G(\Gamma),\gens)$ is $(\rho_1',\rho_2')$--contracting if and
only if there exist $\rho_1$ and $\rho_2$ such that for every
embedded component $\Gamma_0$ of $\Gamma$ in $X$ such that $\Gamma_0\cap\alpha\neq\emptyset$, we have that $\Gamma_0\cap\alpha$ is
$(\rho_1,\rho_2)$--contracting as a subspace of $\Gamma_0$.

Moreover, $\rho_1'$ and $\rho_2'$ can be bounded in terms of $\rho_1$
and $\rho_2$, and when $\rho_1(r)\geqslant \ifrac{r}{2}$ we can take
$\rho_1'=\rho_1$ and $\rho_2'\asymp\rho_2$.
\end{thm*}

Here $\asymp$ denotes a standard notion of asymptotic equivalence, see \fullref{sec:preliminaries}.
Our theorem gives the following explicit application to classical $C'(\ifrac{1}{6})$--groups. We denote by $|\cdot|$ the number of edges of a path graph or cycle graph.

\begin{thm*}[\fullref{corollary:classicalmorse}]
Let $\Gamma$ be a $Gr'(\ifrac{1}{6})$--labelled graph whose components are
cycle graphs.
Let $\alpha$ be a geodesic in  $X:=\Cay(G(\Gamma),\gens)$.
Define $\rho(r):=\max_{|\Gamma_i|\leqslant
  r}|\Gamma_i\cap\alpha|$, where the $\Gamma_i$ range over embedded
components of $\Gamma$ in $X$.
Then $\alpha$ is sublinearly contracting if and only if $\rho$ is sublinear, in which case $\alpha$ is $(r,\rho)$--contracting. 
In particular, $\alpha$ is strongly contracting if and only if $\rho$ is bounded.
\end{thm*}

\subsection{Morse geodesics} 
A classically more well-studied notion of what it means to behave like a subspace of a hyperbolic space is the property of being \emph{Morse}.

\begin{definition}[Morse]\label{def:Morse}
  A subspace $Y$ of a geodesic metric space $X$ is $\mu$--\emph{Morse}  if every $(L,A)$--quasi-geodesic in $X$ with endpoints on
  $Y$ is contained in the $\mu(L,A)$--neighborhood of $Y$.
A subspace is \emph{Morse} if there exists some $\mu$ such that it is $\mu$--Morse.
\end{definition}

The property of being Morse is invariant under quasi-isometries, and the fact that quasi-geodesics in a Gromov hyperbolic space are Morse is known as the `Morse Lemma'. 
These two results are main ingredients in the proof that hyperbolicity is preserved by quasi-isometries. 
Morse geodesics are of further interest due to their close connection with the geometry of asymptotic cones and relations with other important geometric concepts such as divergence, see for example \cite{DruMozSap10,BehDru14,ArzCasGrub} and references therein.
In \cite{ArzCasGrub}, we prove that being contracting is, in fact, equivalent to being Morse.

\begin{theorem}[{\cite[Theorem~1.4]{ArzCasGrub}}]\label{morseequalssublinearlycontracting}
  If $Y$ is a subspace of a geodesic metric space such that the empty set is not in the image of closest point projection to $Y$, then $Y$ is Morse if and only if $Y$ is $(\rho_1,\rho_2)$-contracting for some $\rho_1$ and $\rho_2$.
\end{theorem}

Thus, in the case of classical $C'(\ifrac{1}{6})$--groups, \fullref{theorem:strongcontraction} and \fullref{morseequalssublinearlycontracting} enable us to provide a complete characterization of Morse geodesics in the Cayley graph.

\begin{thm*}[\fullref{corollary:classicalmorse}]
Let $\Gamma$ be a $Gr'(\ifrac{1}{6})$--labelled graph whose components are cycle graphs.
Let $\alpha$ be a geodesic in $\Cay(G(\Gamma),\gens)$.
Define $\rho(r):=\max_{|\Gamma_i|\leqslant r}|\Gamma_i\cap\alpha|$, where the $\Gamma_i$ range over embedded components of $\Gamma$.
Then $\alpha$ is Morse if and only if $\rho$ is sublinear.
\end{thm*}

\subsection{Range of contracting behaviors} 
As mentioned, in a CAT(0)-space, a geodesic is Morse if and only if it is strongly contracting.
Thus,  \fullref{morseequalssublinearlycontracting} says that in a CAT(0) space a geodesic is either strongly contracting or not contracting at all. 
We show that in finitely generated groups, the spectrum of contraction is, in fact, much richer: \emph{every} degree of contraction can be attained.

\begin{thm*}[\fullref{thm:allcontractionrates}]
  Let $\rho$ be a sublinear function. There exists a group $G$ with finite generating set $\gens$ and a sublinear function $\rho'\asymp\rho$ such that there exists an $(r,\rho')$--contracting geodesic $\alpha$ in $\Cay(G,\gens)$, and $\rho'$ is optimal, in the sense that if $\alpha$ is  $(r,\rho'')$--contracting for some other function $\rho''$ then $\limsup_{r\to\infty}\frac{\rho''(2r)}{\rho(r)}\geqslant 1$.

Furthermore $\alpha$ can be chosen to be within finite Hausdorff distance of a cyclic subgroup of $G$.
\end{thm*}

Our current examples are not finitely presentable.
The contraction spectrum for finitely presented groups remains largely
unexplored.
Indeed, if one restricts to geodesics within finite Hausdorff distance
of a cyclic subgroup, finitely presented groups can only display countably many degrees of contraction.
\begin{question}
  For which functions $\rho$ do there exist finitely presented groups
  $G$ containing a geodesic in some Cayley graph that is $(r,\rho)$--contracting?
\end{question}

\subsection{Non-stability of strong contraction} While the property of being Morse is stable under quasi-isometries, it has remained unknown whether the property of being strongly contracting is. We provide a negative answer by providing the first examples of spaces $X$ and $\tilde X$ and geodesics $\gamma$ and $\tilde \gamma$ such that there exists a quasi-isometry $X\to \tilde X$ mapping $\gamma$ to $\tilde \gamma$ and such that $\gamma$ is not strongly contracting, but $\tilde\gamma$ is strongly contracting.

\begin{thm*}[\fullref{thm:nonstability}]
There exists a group $G$ with finite generating sets $\gens\subset \tilde\gens$ and an infinite geodesic $\gamma$ in $X:=\Cay(G,\gens)$ labelled by the powers of a generator such that $\gamma$ is not strongly contracting, but its image $\tilde\gamma$ in $\tilde X:=\Cay(G,\tilde\gens)$ obtained from the inclusion $\gens\subset\tilde\gens$ is an infinite strongly contracting geodesic.
\end{thm*}

Indeed, in many familiar settings, such as hyperbolic groups, CAT(0) groups, or toral relatively hyperbolic groups, such examples could not be obtained, since in those contexts, strong contraction is equivalent to the Morse property. 

\subsection{Strongly contracting elements and growth tightness} Another of our main results is the existence of strongly contracting elements in many graphical small cancellation groups:

\begin{thm*}[\fullref{thm:existenceofstronglycontractingelement}]
Let $\Gamma$ be a $Gr'(\ifrac{1}{6})$--labelled graph whose components
are finite, labelled by a finite set $\gens$. Assume that $G(\Gamma)$ is infinite. Then there exists an infinite order element $g\in G(\Gamma)$ such that $\fgen{g}$ is strongly contracting in  $\Cay(G(\Gamma),\gens)$.
\end{thm*}

The element $g$ is, in fact, the WPD element for the action on the hyperbolic coned-off space in Gruber and Sisto's proof of acylindrical hyperbolicity of these groups \cite{GruSis14}. \fullref{thm:existenceofstronglycontractingelement} has the following consequence (which does not follow from acylindrical hyperbolicity):

Arzhantseva, Cashen, and Tao \cite{ArzCasTao15} have shown that the action of a finitely generated group $G$ on a Cayley graph $X$ is \emph{growth tight} if the action has a strongly contracting element, that is, an element $g$ such that $\langle g\rangle$ is strongly contracting in $X$.  Growth tightness means that the exponential growth rate of an orbit of $G$ in $X$ is strictly greater than the growth rate of an orbit of $G/N$ in $N\backslash X$, for every infinite normal subgroup $N$. \fullref{thm:existenceofstronglycontractingelement} therefore implies:

\begin{thm*}[\fullref{thm:growthtight}]
Let $\Gamma$ be a $Gr'(\ifrac{1}{6})$--labelled graph whose components are finite, labelled by a finite set $\gens$. 
Then the action of $G(\Gamma)$ on $\Cay(G(\Gamma),\gens)$ is growth tight.
\end{thm*}

This has raised our interest in the following question, first asked by Grigorchuk and de la Harpe for hyperbolic fundamental groups of closed orientable surfaces \cite{GriHar97}:

\begin{question} Let $\Gamma$ be a $Gr'(\ifrac{1}{6})$--labelled graph whose components are finite, labelled by a finite set $\gens$. Does $G(\Gamma)$ attain its infimal growth rate with respect to the generating set $\gens$?
\end{question}

Together with \fullref{thm:growthtight}, a positive answer, even for the subclass of classical $C'(\ifrac{1}{6})$--groups, would establish small cancellation theory as an abundant source of \emph{Hopfian} groups.

\subsection{Contraction and hyperbolically embedded subgroups} An important application of the notion of strong contraction is the fact that an infinite order element whose orbit in the Cayley graph is strongly contracting is contained in a virtually cyclic hyperbolically embedded subgroup \cite{DGO11}, see \fullref{defi:hyperbolicallyembedded}. In particular, our proof of \fullref{thm:existenceofstronglycontractingelement} gives a new argument that the WPD elements of \cite{GruSis14} produce hyperbolically embedded subgroups.
Furthermore, our methods detect strongly contracting elements that need not be hyperbolic elements for the action on the coned-off space of \cite{GruSis14}, see \fullref{rem:unboundedproj} and \fullref{newstrongcontractingelements}.
Admitting a proper infinite hyperbolically embedded subgroup is equivalent to being acylindrially hyperbolic \cite{Osi13}, which implies a number of strong group theoretic  properties.

Not all hyperbolically embedded virtually cyclic subgroups are strongly contracting; see \cite{ArzCasTao15} for an example. However, every hyperbolically embedded subgroup of a finitely generated group is Morse \cite{SistoMorse}. In light of \fullref{morseequalssublinearlycontracting}, a natural question is whether there exists some critical rate of contraction that guarantees a subgroup is hyperbolically embedded. That is, does there exist an unbounded sublinear function $\rho_2$ such that every element $g$ with a $(r,\rho_2)$--contracting orbit in some Cayley graph has a hyperbolically embedded virtually cyclic elementary closure? The elementary closure of $g$ is the subgroup generated by all virtually cyclic subgroups containing $g$. We prove no such $\rho_2$ exists.

\begin{thm*}[\fullref{theorem:unbddrhocontnotHE}]
Let $\rho_2$ be an unbounded sublinear function. 
There exists a $Gr'(\ifrac{1}{6})$--labelled graph $\Gamma$ with set of labels $\gens:=\{a,b\}$ whose components are all cycles such that $G(\Gamma)$ has the following properties: Any virtually cyclic subgroup $E$ of $G(\Gamma)$ containing $\langle a\rangle$ is $(r,\rho'_2)$--contracting in the Cayley graph $\Cay(G(\Gamma),\gens)$ for some $\rho'_2\asymp \rho_2$, but $E$ is not hyperbolically embedded in $G(\Gamma)$.
\end{thm*}

\subsection{Hyperbolically embedded cycles} In a subgroup of a hyperbolic group, every infinite order element is contained in a maximal virtually cyclic, hyperbolically embedded subgroup, whence we define the following:
\begin{definition}[HEC property]\label{defi:hec}
A group has the \emph{hyperbolically embedded cycles property (HEC property)} if the elementary closure $E(g)$ of every infinite order element $g$ is virtually cyclic and hyperbolically embedded. 
\end{definition}
It is natural to ask whether this property characterizes subgroups of hyperbolic groups.
While torsion presents an obvious complication, see \fullref{sec:hypembcycle}, we also present a negative answer to our question in the torsion-free case.

\begin{thm*}[\fullref{thm:allmaxsubgrpsHE}] There exist $2^{\aleph_0}$
  pairwise non-quasi-isometric finitely generated torsion-free groups
  in which every non-trivial cyclic subgroup is strongly contracting and which, therefore, have the HEC property.
\end{thm*}

These are the first examples of groups of this kind that do not arise as
subgroups of hyperbolic groups. Our examples include
exotic specimens such as Gromov monsters.

\subsection{Translation lengths}
Let $|\cdot|$ be the word length in $G(\Gamma)$ with respect to $\gens$.
The \emph{translation length} of an element $g\in G(\Gamma)$ is:
\[\tau(g):=\lim_{n\to\infty}\frac{|g^n|}{n}\]

Conner \cite{Con97} calls a group whose non-torsion elements have 
translation length bounded away from zero \emph{translation discrete}. 
Hyperbolic groups \cite{Swe95}, CAT(0) groups \cite{Con00}, and 
finitely presented groups satisfying various classical small cancellation conditions
\cite{Kap97} are translation discrete.

We show that many (possibly infinitely presented) graphical small cancellation groups are also translation discrete:
\begin{thm*}[{\fullref{thm:translationlengths}}]
  Let $\Gamma$ be a $Gr'(\ifrac{1}{6})$--labelled graph whose components
are finite, labelled by a finite set $\gens$. 
Then every infinite order element of $G(\Gamma)$ has rational translation length, and translation lengths are bounded away from zero.
\end{thm*}


\subsection{The idea of the proof of the local-to-global theorem}

In a tree, geodesic quadrangles are degenerate, as seen in
\fullref{fig:treequadrangle}.
\begin{figure}[h]
  \centering
\hfill
\begin{subfigure}[h]{.49\textwidth}
\centering
\labellist
\tiny
\endlabellist
\includegraphics[scale=.35]{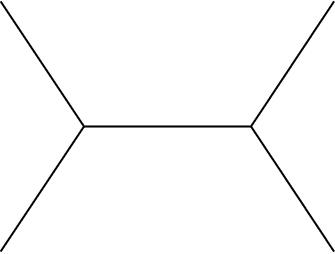}  
\end{subfigure}
\hfill
\begin{subfigure}[h]{.49\textwidth}
\centering
\labellist
\tiny
\endlabellist
\includegraphics[scale=.35]{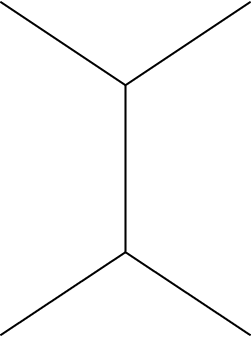}  
\end{subfigure}
\hfill
  \caption{Tree quadrangles}\label{fig:treequadrangle}
\end{figure}
In a hyperbolic space, geodesic quadrangles can be approximated by
geodesic quadrangles in a tree. 
If the base is a fixed geodesic $\alpha$, the top is some given
geodesic $\gamma$, and
the sides are given by closest point projection from the endpoints of
the top to the bottom, then the resulting geodesic quadrangle is
either `short' or `thin', as in
\fullref{fig:hyperbolicquadrangle}.
\begin{figure}[h]
  \centering
  \hfill
\begin{subfigure}[h]{.49\textwidth}
\centering
\labellist
\tiny
\pinlabel $\alpha$ [t] at 20 1
\pinlabel $\gamma$ [bl] at 11 90
\endlabellist
\includegraphics[scale=.35]{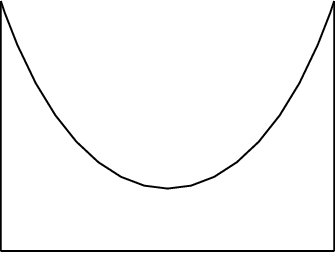}  
\end{subfigure}
\hfill
\begin{subfigure}[h]{.49\textwidth}
\centering
\labellist
\tiny
\pinlabel $\alpha$ [t] at 20 1
\pinlabel $\gamma$ [bl] at 31 139
\endlabellist
\includegraphics[scale=.35]{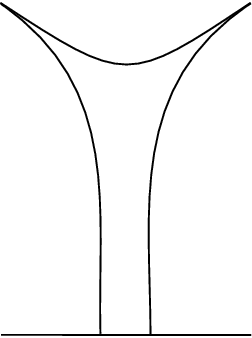}  
\end{subfigure}
\hfill
  \caption{Hyperbolic quadrangles from closest point projection}
  \label{fig:hyperbolicquadrangle}
\end{figure}

In \fullref{prop:combinatorialngons}, we show a combinatorial version
of this dichotomy through an analysis of van Kampen diagrams in graphical
$Gr'(\ifrac{1}{6})$ small cancellation groups. 
Specifically, if $X:=\Cay(G(\Gamma),\gens)$ and $\alpha\subset X$ is a $(\rho_1,\rho_2)$--contracting
geodesic, $\gamma\subset X$ is another geodesic, and
each endpoint of $\gamma$ is connected via a geodesic to a closest
point of $\alpha$, then the boundary word of the resulting geodesic
quadrangle admits a van Kampen diagram that is either `short' or `thin'
in terms of number of faces, as depicted in \fullref{fig:introquaddichotomy}.

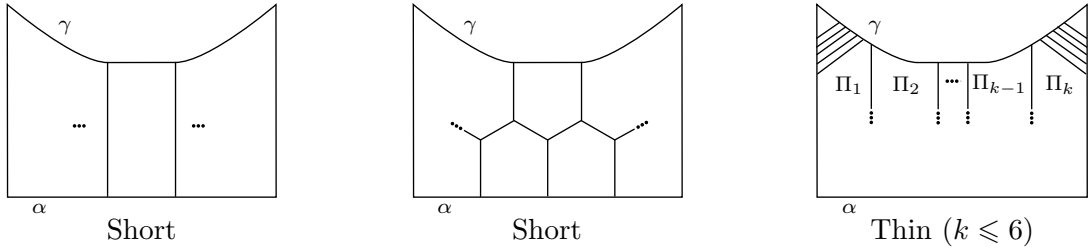
\begin{figure}[h]
\captionsetup[subfigure]{labelformat=empty}
  \centering
\hfill
  \begin{subfigure}[h]{.33\textwidth}
\centering      
\labellist
\tiny
\pinlabel $\gamma$ [bl] at 28 94
\pinlabel $\alpha$ [t] at 20 1
\endlabellist
  \includegraphics[height=1.3\myquadheight]{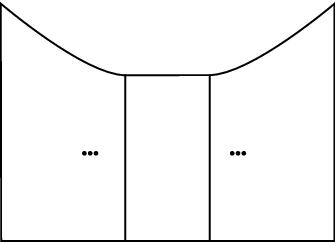}
\caption{Short}
  \end{subfigure}\hfill
\begin{subfigure}[h]{.33\textwidth}
\centering      
\labellist
\tiny
\pinlabel $\gamma$ [bl] at 28 94
\pinlabel $\alpha$ [t] at 20 1
\endlabellist
  \includegraphics[height=1.3\myquadheight]{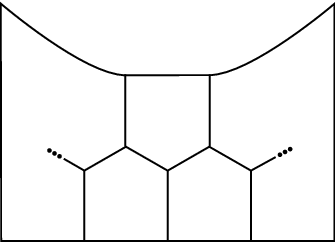}
\caption{Short}
  \end{subfigure}\hfill
\begin{subfigure}[h]{.33\textwidth}
\centering      
\labellist
\tiny
\pinlabel $\Pi_1$ at 20 67
\pinlabel $\Pi_2$ at 54 67
\pinlabel $\Pi_{k-1}$ at 109 67
\pinlabel $\Pi_k$ at 145 67
\pinlabel $\gamma$ [bl] at 28 94
\pinlabel $\alpha$ [t] at 20 1
\endlabellist
  \includegraphics[height=1.3\myquadheight]{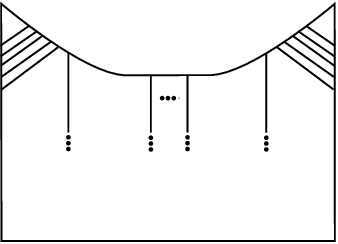}
\caption{Thin ($k\leqslant 6)$}
  \end{subfigure}\hfill
  \caption{Combinatorially short and thin quadrangles}\label{fig:introquaddichotomy}
\end{figure}

The main ingredient in establishing this dichotomy is a classification
of `special combinatorial geodesic quadrangles', see
  \fullref{thm:classificationofspecialquads}, that extends Strebel's
  classification of geodesic bigons and triangles in small
  cancellation groups, see \fullref{thm:strebel}.
This classification is of independent interest, and is novel even
within the class of classical small cancellation groups.

These combinatorial versions of short and thin quadrangles do
not immediately imply their metric counterparts, because the faces in
the van Kampen diagrams may have boundary words that are arbitrarily
long relators. 
In \fullref{sec:gsccontraction} we use the fact that $\alpha$ is
$(\rho_1,\rho_2)$--contracting to show that even if the faces have
long boundaries, their projections to $\alpha$ are small with respect
to their distance from $\alpha$.

The essential trick that is used repeatedly is to play off the small cancellation condition against the contraction condition. 
Specifically, if $\Pi$ is a face of the van Kampen diagram with few
sides, one of which sits on $\alpha$, we use the small cancellation
condition to show that $|\alpha\cap\bdry\Pi|$ is bounded below by a
linear function of $|\bdry\Pi|$. Then we use the contraction condition
to say that $|\alpha\cap\bdry\Pi|$ is bounded above by a sublinear
function of $|\bdry\Pi|$.
Thus, we have a sublinear function of $|\bdry\Pi|$ that gives an upper
bound to a linear function of $|\bdry\Pi|$.
This is only possible
if $|\bdry\Pi|$ is smaller than some bound depending on the two functions.


\section{Preliminaries}\label{sec:preliminaries}
We set notation. Let $\gens$ be a set.
\begin{itemize}
\item $\fgen{\gens}$ is the group generated by $\gens$. If $\gens$ is
  a subset of a group $G$ then $\fgen{\gens}$ is the subgroup of $G$
  generated by $\gens$. If $\gens$ is a set of formal symbols then
  $\fgen{\gens}$ is the free group freely generated by $\gens$.
\item If $\gens$ is a set of formal symbols then 
  $\gens^-:=\{s^{-1}\mid s\in\gens\}$ is the set of formal inverses,
  and $\gens^{\pm}:=\gens\cup\gens^-$. 
\item $\gens^*$ is the free monoid over $\gens$.
\item $2^\gens$ is the 
set of subsets of $\gens$.
\item $\gens^\N$ is the set of infinite 
sequences with terms in $\gens$.
\end{itemize}

We write $f\asympleq g$ if there exist $C_1>0$, $C_2>0$, $C_3\geqslant 0$, and $C_4\geqslant 0$ such that $f(x)\leqslant C_1g(C_2x+C_3)+C_4$ for all $x$.
If $f\asympleq g$ and $g\asympleq f$ then we write $f\asymp g$.

Note that if $f\asympleq g$ and $g$ is bounded then $f$ is bounded, and if $f$ is non-decreasing and eventually non-negative and $f\asympleq
g$ for a function $g$ such that $\lim_{r\to\infty}\frac{g(r)}{r}=0$, then $f$ is sublinear.

The \emph{girth} of a graph is the length of its shortest non-trivial cycle.

\subsection{Graphical small cancellation}\label{sec:smallcancellation}
Graphical small cancellation theory is a generalization of classical small
cancellation theory.
The main application is an embedding of a desired sequence of graphs into the Cayley graph of a group.
It was introduced by Gromov \cite{Gro03}, and was later clarified and expanded by Ollivier \cite{Oll06},
Arzhantseva and Delzant \cite{ArzDel08}, and, in a systematic way, by Gruber \cite{Gru15}.
\subsubsection{Basic facts}
Let $\Gamma$ be a directed graph with edges labelled by a set
$\gens$.
We allow paths to traverse edges against their given
direction, with the convention that the label of an oppositely
traversed edge is the formal inverse of the given label.
Thus, given a finite path in $\Gamma$ we can \emph{read} a word in
$(\gens^\pm)^*$ by concatenating the labels of the edges along the path. 

We require that the labelling is \emph{reduced}, in the sense that no vertex
has two incident outgoing edges with the same label, and no vertex has
two incident incoming edges with the same label. 
This implies that the word read on an immersed path is freely reduced and, hence, an element of $\fgen{\gens}$. Also, the word read
on an immersed cycle is cyclically reduced.

Let $\rels$ be the set of words in $\left< \gens\right>$ read on embedded cycles in $\Gamma$.
Note that this definition implies that elements of $\rels$ are
cyclically reduced and that $\rels$ is closed
under inversion and cyclic permutation of its elements. 

\begin{definition}[Group defined by a labelled graph]
  The group $G$ defined by a reduced $\gens$--labelled graph $\Gamma$ is the
  group $G(\Gamma):=\langle \gens\mid\rels\rangle$.
\end{definition}
The notion of a group defined by a labelled graph first appeared in Rips and Segev's construction of torsion-free groups without the unique-product property \cite{RipSeg87}.

\begin{definition}[Piece]
  A \emph{piece} is a labelled path graph $p$ that admits two distinct
  label-preserving maps
$\phi_1,\,\phi_2\from p\to\Gamma$ such that there is no
label-preserving automorphism
$\psi$ of $\Gamma$ with $\phi_2=\psi\circ\phi_1$.
\end{definition}

\begin{definition}[$Gr'(\lambda)$ and $C'(\lambda)$ conditions]
Let $\Gamma$ be a reduced labelled graph, and let $\lambda>0$.

  $\Gamma$ is \emph{$Gr'(\lambda)$--labelled} if
  whenever $p$ is a piece contained in a simple cycle $c$ of $\Gamma$ then
  $|p|<\lambda|c|$.

$\Gamma$ is \emph{$C'(\lambda)$--labelled} if it is
$Gr'(\lambda)$--labelled and, in addition, every label-preserving automorphism of $\Gamma$ restricts to the identity on every connected component with non-trivial fundamental group. 

A presentation $\langle\gens\mid\rels\rangle$ satisfies the \emph{classical $C'(\lambda)$-condition} if  the disjoint union of cycle graphs labelled by the elements of $\rels$ is a $Gr'(\lambda)$-labelled graph.
\end{definition}

Actually, every group
is defined by a $Gr'(\ifrac{1}{6})$--labelled graph: simply take $\Gamma$ to be its
Cayley graph with respect to any generating set of the group \cite[Example~2.2]{Gru15}.
Therefore, general statements about groups defined by $Gr'(\ifrac{1}{6})$--labelled graphs either require that some
additional condition be imposed on $\Gamma$ or are tautologically true when $\Gamma=\Cay(G(\Gamma),\gens)$.

A subspace $Y$ of a geodesic metric space is \emph{convex} if every geodesic segment
between points of $Y$ is contained in $Y$.

\begin{lemma}[{\cite[Lemma~2.15]{GruSis14}}]
  Let $\Gamma$ be a $Gr'(\ifrac{1}{6})$--labelled graph.
Let $\Gamma_i$ be a component of $\Gamma$.
For any choice of a vertex $x\in X:=\Cay(G(\Gamma),\gens)$ and any vertex $y\in\Gamma_i$ there is a unique label-preserving map
$\Gamma_i\to X$ that takes $y$ to $x$, and this map is an isometric embedding with convex image.
\end{lemma}

\begin{definition}[Embedded component]
  An \emph{embedded component} $\Gamma_0$ of $\Gamma$ refers to the
  image of an isometric embedding of some $\Gamma_i$ into $X:=\Cay(G(\Gamma),\gens)$ via a
  label-preserving map. Equivalently, it is a $G(\Gamma)$--translate in $X$ of
  the image of $\Gamma_i$ under the unique label-preserving map determined by an arbitrary choice of basepoints in
  $X$ and $\Gamma_i$.
\end{definition}

We consider a graph $\Gamma$ as a sequence $(\Gamma_i)_i$
of its connected components.

\subsection{Contraction terminology}
Recall \fullref{def:contractingprojection}. In \cite{ArzCasGrub} we considered, more generally, almost closest point projections $x\mapsto \{y\in
Y\mid d(x,y)\leqslant d(x,Y)+\epsilon\}$ to ensure the empty set is
not in the image of the projection. 
That is unnecessary in this paper as we are in the case that $Y$ is a
subgraph of a graph $X$, which guarantees
$\emptyset\notin\mathrm{Im}\thinspace\pi$. 
Here, and from now on, $\pi\from X\to 2^Y$ denotes closest point projection to $Y$.

We say a geodesic $\alpha$ in $\Cay(G(\Gamma),\gens)$ is \emph{locally $(\rho_1,\rho_2)$--contracting} if, for each embedded component $\Gamma_0$ of $\Gamma$ such that $\Gamma_0\cap\alpha$ is non-empty, closest point projection in $\Gamma_0$ of $\Gamma_0$ to $\Gamma_0\cap\alpha$ is $(\rho_1,\rho_2)$--contracting. 

We say a geodesic is \emph{uniformly locally contracting} if there exist $\rho_1$ and $\rho_2$ such that it is locally
$(\rho_1,\,\rho_2)$--contracting. We add `uniform' here to stress that the intersection of the geodesic with each embedded component of $\Gamma$ is contracting with respect
to the same contraction functions. Similarly, a geodesic is \emph{uniformly locally sublinearly contracting} if it is locally $(r,\rho_2)$--contracting, and is \emph{uniformly locally strongly contracting} is it locally $(r,\rho_2)$--contracting for $\rho_2$ bounded.


\section{Classification of quadrangles}\label{sec:polygons}
In this section, we establish geometric results that will let us prove our theorems about contraction in graphical small cancellation groups. In particular, we provide a complete classification of the geodesic quadrangles in the Cayley graph of a group defined by a $Gr'(\ifrac{1}{6})$--labelled graph. The main technical result to be used in our subsequent investigation will be recorded in \fullref{prop:combinatorialngons}.

\subsection{Combinatorial geodesic polygons} One of the main tools of small cancellation theory are so-called `van Kampen diagrams'.
\begin{definition}[Diagram]
  A (disc) \emph{diagram} is a finite, simply-connected,
  2--dimensional CW complex with an embedding into the plane,
  considered up to orientation-preserving homeomorphisms of the plane.
It is \emph{$\gens$--labelled} if its directed edges are labelled by
  elements in $\gens$.
It is a diagram \emph{over $\rels$} if it is $\gens$--labelled and the word read on the boundary of each 2--cell belongs to $\rels$. A diagram is \emph{simple} if it is homeomorphic to a disc.
\end{definition}

If $D$ is a diagram over $\rels$, $b$ is a basepoint in $D$, and
$g$ is an element of $G=\langle\gens\mid\rels\rangle$, then there exists a unique 
label-preserving map from the 1--skeleton of $D$ into $\Cay(G,\gens)$ taking $b$ to $g$. 
In general, this map need not be an immersion.

An \emph{arc} in a diagram $D$ is a maximal path of length at least 1 all of whose
interior vertices have valence 2 in $D$.
An \emph{interior arc} is an arc whose interior is contained in the
interior of $D$. An \emph{exterior arc} is an arc contained in the
boundary of $D$. 
A \emph{face} is the image of a closed 2-cell of $D$.
If $\Pi$ is a face, its \emph{interior degree} $i(\Pi)$ is the number
of interior arcs in its boundary. 
Likewise, its \emph{exterior degree} $e(\Pi)$ is the number of
exterior arcs. An \emph{interior face} is one with exterior degree 0; an \emph{exterior face} is one with positive exterior degree.

If $D$ is a finite, simply connected, planar, 2--dimensional CW-complex
whose boundary is written as a concatenation of immersed subpaths
$\gamma_1,\dots,\gamma_k$, which we call \emph{sides} of $D$, then
there is a
unique, up to orientation-preserving homeomorphism of $\mathbb{R}^2$, embedding 
$\phi\from D\to \mathbb{R}^2$ such that the concatenation of the $\phi(\gamma_i)$
is the positively oriented boundary $\bdry\phi(D)$.
This claim follows easily from the Schoenflies Theorem.
Thus, $(D,(\gamma_i)_i)$ uniquely determines a (not necessarily
simple) diagram $\phi(D)$.
We call $\phi$ the \emph{canonical embedding} of $(D,(\gamma_i)_i)$.
Having said this once, we omit $\phi$ from the notation and conflate
$D$ and the $\gamma_i$ with their $\phi$--images. 

\begin{defi}[$(3,7)$--diagram]
  A \emph{$(3,7)$--diagram} is a diagram such that every interior vertex
  has valence at least three and every interior face has interior
  degree at least seven.
\end{defi}

\begin{defi}[{Combinatorial geodesic polygon \cite[Definition~2.11]{GruSis14}}]\label{def:combinatorialgeodesicpolygon}
  A \emph{combinatorial geodesic $n$--gon} $(D,(\gamma_i)_i)$ is a
$(3,7)$--diagram $D$ whose boundary is a concatenation of immersed subpaths
$\gamma_0,\dots,\gamma_{n-1}$ such that each boundary face whose
  exterior part is a single arc that is contained in one of the sides
  $\gamma_i$ has interior degree at least 4.
A valence 2 vertex that belongs to more than one side is called a
\emph{distinguished vertex}. 
A face whose exterior part contains an arc not contained in one of the
sides is a \emph{distinguished face}. 
\end{defi}

The ordering of the sides of a combinatorial geodesic $n$--gon is considered up to cyclic permutation, with subscripts
modulo $n$. 
We also refer to `the combinatorial geodesic $n$--gon
$D$' when the sides are clear from context or irrelevant.
We can also say `combinatorial geodesic polygon' when the number of
sides is irrelevant.
Following common usage, 2--gons, 3--gons, and 4--gons will
respectively be denominated \emph{bigons}, \emph{triangles}, and \emph{quadrangles}.

If $D$ is a simple combinatorial geodesic $n$--gon then every distinguished face contains a distinguished vertex,
so there are at most $n$ distinguished faces.

We record the following crucial fact about diagrams over graphical small cancellation presentations. In the following, $\rels$ is the set of labels of simple cycles on a $Gr'(\ifrac{1}{6})$--graph $\Gamma$ labelled by the set $\gens$, and $X:=\Cay(G(\Gamma),\gens)$.
\begin{lemma}[{\cite[Lemma~2.13]{Gru15}}]\label{lem:graphical_diagrams}
Let $\Gamma$ be a $Gr'(\ifrac{1}{6})$--labelled graph, and let $w\in\langle\gens\rangle$ represent the identity in $G(\Gamma)$. Then, there exists an $\gens$-labelled diagram over $\rels$ with boundary word $w$ in which every interior arc is a piece.
\end{lemma}

The sides of a combinatorial geodesic polygon are \emph{not} assumed
to be geodesic.
The definition and choice of terminology are motivated by the
following proposition. 
An $n$--gon $P$ in $X$ is a closed edge path that
decomposes into immersed simplicial subpaths
  $\gamma'_0,\dots,\gamma'_{n-1}$, which are called \emph{sides of $P$}.

\begin{proposition}\label{prop:sctocombinatorial}
If $P$ is an $n$--gon in $X$ with sides $\gamma'_0,\dots,\gamma'_{n-1}$
that are
geodesics then there is an
$\gens$--labelled diagram $D$ over $\rels$ with sides $\gamma_0,\dots,\gamma_{n-1}$
such that for each $0\leqslant i< n$ the word of $\left<\gens\right>$ read on $\gamma_i$ is the same as
the word read on $\gamma'_i$.
Furthermore, we can choose $D$ in such a way that
after forgetting
interior vertices of valence 2, we obtain
a combinatorial geodesic $n$--gon $(D,(\gamma_i)_i)$ 
  
\end{proposition}
Here, forgetting one interior vertex of valence 2 means replacing its two incident edges by a single one. Note that, when performing this operation, we consider $D$ merely as (unlabelled) diagram, i.e.\ we ignore the orientations and labels of edges.
Forgetting interior vertices of valence 2 means iterating this operation, such that we end up with a diagram without interior vertices of valence 2.

\begin{proof}
  The existence of an $\gens$--labelled diagram over $\rels$ whose boundary label matches the label of
  $P$ is the well-known van Kampen Lemma. 
\fullref{lem:graphical_diagrams} guarantees that the diagram can be chosen such that all interior arcs are pieces.
The small cancellation condition then implies
interior faces have interior degree at least 7. 
If there is a face $\Pi$ with $e(\Pi)=1$ whose exterior part is
contained in a single $\gamma_i$, then the exterior part is a
geodesic. 
Thus, the length of the interior part is at least half of
the length of $\bdry\Pi$. 
Since interior arcs are pieces, the small cancellation condition
implies there must be at
least four of them to account for half the length of $\bdry\Pi$.
Now if we forget interior vertices of valence 2, we have the
desired combinatorial geodesic polygon.
\end{proof}

\begin{remark}\label{remark:combinatorialvssc}
  The word $w\in \left<\gens\right>$ read on a cycle in $X$
   represents the trivial element in $G(\Gamma)$.
The combinatorial geodesic $n$--gon of
\fullref{prop:sctocombinatorial} is a special type of van Kampen
diagram witnessing the triviality of the word $w$ labelling an
$n$--gon in $X$ whose sides are geodesics.
In the remainder of \fullref{sec:polygons} we make combinatorial
arguments about arbitrary $(3,7)$--diagrams, not necessarily
$\gens$--labelled diagrams
over $\rels$. 
In \fullref{sec:gsccontraction} we use \fullref{prop:sctocombinatorial} to
apply results of this section to graphical small cancellation groups.
\end{remark}

We record an equivalent formulation of the Euler characteristic
formula for certain diagrams. Recall that for a face $\Pi$ of a
diagram, $e(\Pi)$ is the exterior
degree of  $\Pi$, which is the number of exterior
arcs in its boundary. Similarly, $i(\Pi)$ is the interior degree of $\Pi$.
\begin{lemma}[Strebel's curvature formula, {\cite[p.253]{Str90}}]\label{lem:curvature}
 Let $D$ be a simple diagram without vertices of degree 2. Then:
\begin{align*}
6&=2\sum_v(3-d(v)) \\
 &+\sum_{e(\Pi)=0}(6-i(\Pi))+\sum_{e(\Pi)=1}(4-i(\Pi))+\sum_{e(\Pi)\geqslant 2}(6-2e(\Pi)-i(\Pi)). 
\end{align*}
Here $d(v)$ denotes the degree of a vertex $v$.
 \end{lemma}

It readily follows from \fullref{lem:curvature} that any $(3,7)$--diagram with more than one face has at least 2 faces with exterior degree 1 and interior degree at most 3. (This is usually known as Greendlinger's lemma.) Therefore:

\begin{lemma}
The sides of a combinatorial geodesic polygon are embedded, and every combinatorial geodesic polygon has at least two sides.
\end{lemma}

The same argument gives the following well-known fact, which greatly simplifies many considerations:

\begin{lemma}
Let $D$ be a $(3,7)$--diagram. Then any face is simply connected.
\end{lemma}

We also state an immediate consequence of \cite[Lemma 4.14]{Gru15SQ}:

\begin{lemma}\label{corollary:faceseqsegment}
  If $\Pi$ is a face of a combinatorial geodesic polygon $D$ and 
  $\alpha$ is a side of $D$ then $\Pi\cap\alpha$ is empty or connected.  
  If $\Pi_1,\dots,\Pi_k$ is a sequence of faces of a combinatorial
  geodesic polygon $D$ such that $\Pi_i\cap\Pi_{i+1}\neq\emptyset$ for
  all $1\leqslant i<k$ and $\alpha$ is a side of $D$ such that
  $\Pi_i\cap\alpha\neq\emptyset$ for all $i$ then $\cup_{1\leqslant
    i\leqslant k}\Pi_i\cap\alpha$ is connected. 
\end{lemma}

\begin{defi}[Degenerate]
  A combinatorial geodesic $n$--gon $(D,(\gamma_i)_i)$ is \emph{degenerate} if there
  exists an $i$ such that $D$,
  $\gamma_0,\dots,\gamma_i\gamma_{i+1},\dots,\gamma_{n-1}$ is a
  combinatorial geodesic $(n-1)$--gon.
In this case the terminal vertex of $\gamma_i$ is called a
\emph{degenerate vertex}.
\end{defi}
It will be useful to minimize the number of sides of a
diagram $D$ by replacing a degenerate combinatorial geodesic $n$--gon
$(D,(\gamma_i)_i)$ with a non-degenerate combinatorial geodesic $k$--gon
$(D,(\gamma'_i)_i)$ for some $k<n$.

\subsection{Reducibility}
In this section we define operations for combining and reducing
combinatorial geodesic $n$--gons.
We stress that the setting is only combinatorial --- these operations need
not preserve the property of being a diagram over $\rels$.

First note that if $(D,(\gamma_i)_i)$ is a combinatorial geodesic
$n$--gon, and if $D'$ is obtained from $D$ by subdividing an edge,
then $(D',(\gamma_i)_i)$ is still a combinatorial geodesic $n$--gon. 
The new vertex produced by subdivision is a non-distinguished vertex of
valence 2.
Conversely, if $v$ is a non-distinguished vertex of valence 2 then we
can `forget' it by replacing the two incident edges with a single
edge.

If $D$ is simple and non-degenerate then by forgetting all non-distinguished vertices of
valence 2 we can arrange that the distinguished vertices are exactly
the vertices of valence 2 and all other vertices have valence at least 3.

\begin{defi}[Reducible]
  A combinatorial geodesic $l$--gon $P$ is \emph{reducible} if it
  admits a vertex, edge, or face reduction, as defined below, see \fullref{fig:combinationreduction}. It is \emph{irreducible} otherwise.
\end{defi}

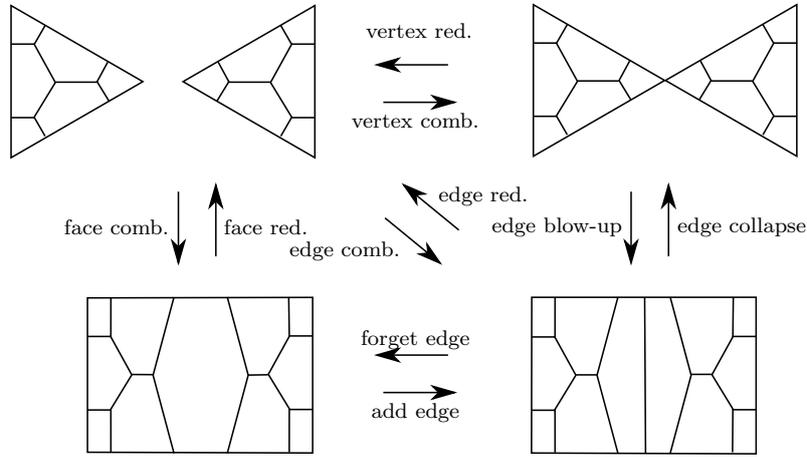
\begin{figure}[h]
  \centering
  \labellist
  \tiny
\pinlabel {vertex red.} [b] at 210 210
\pinlabel {vertex comb.} [t] at 208 180
\pinlabel {edge red.} [bl] at 217 125
\pinlabel {edge comb.} [tr] at 203 112
\pinlabel {face red.} [l] at 106 117 
\pinlabel {face comb.} [r] at 86 117
\pinlabel {edge collapse} [l] at 338 117
\pinlabel {edge blow-up} [r] at 318 117
\pinlabel {add edge} [t] at 208 30
\pinlabel {forget edge} [b] at 208 50
  \endlabellist
\includegraphics[height=3\myquadheight]{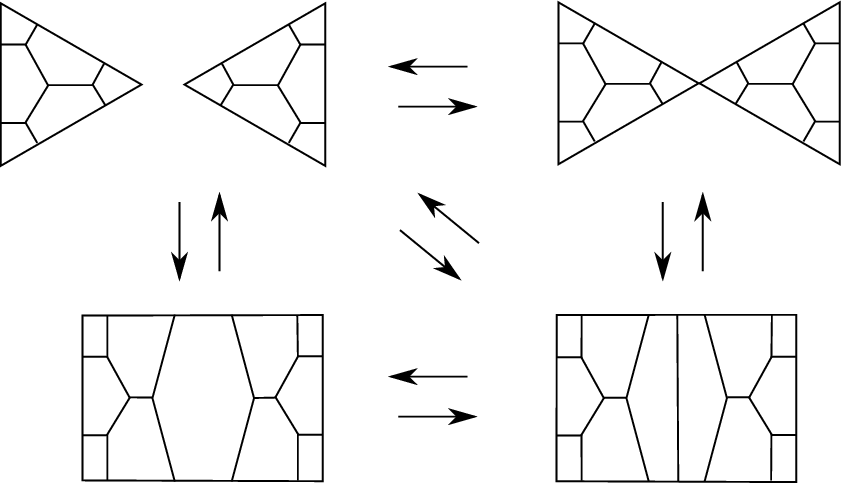}
  \caption{Combination and reduction}
  \label{fig:combinationreduction}
\end{figure}

In all of the following cases, let $(D,(\gamma_i)_i)$ be a combinatorial geodesic
$n$--gon, and let $(D',(\gamma'_i)_i)$ be a combinatorial geodesic $n'$--gon.

\subsubsection{Vertex reduction}
Suppose $v\in D$ is a separating vertex that is in the boundary
of exactly two faces and that these two faces are the only maximal
cells containing $v$. 
Suppose that there are exactly two
sides, $\gamma_i$ and $\gamma_j$, containing $v$, which is necessarily
true if $(D,(\gamma_i)_i)$ is non-degenerate.
Let $\gamma^{-v}_i$ denote the initial path of $\gamma_i$ ending at
$v$, and let $\gamma^{v+}_i$ denote the terminal path of $\gamma_i$
beginning at $v$, and similarly for $\gamma_j$.
Define the \emph{vertex reduction of $(D,(\gamma_i)_i)$ at $v$} to be the
two combinatorial geodesic polygons whose underlying CW-complexes are each $D$ minus
one of the complementary components of $D\setminus v$, respectively, and whose sides
are, respectively,
$\gamma^{v+}_i,\gamma_{i+1},\dots,\gamma_{j-1},\gamma^{-v}_j$ and
$\gamma^{v+}_j,\gamma_{j+1},\dots\gamma_{i-1},\gamma^{-v}_i$. 

Note that each of the resulting combinatorial polygons contains a
distinguished vertex corresponding to $v$.

The inverse operation of vertex reduction we denominate a \emph{vertex combination}.
Suppose that $v\in D$ is a distinguished vertex that is contained in
the boundary of a single face and is not contained in any other
maximal cell. Then there is some $i$ such that  $v=\gamma_i\cap\gamma_{i+1}$.
Make corresponding assumptions for $v'=\gamma'_{i'}\cap\gamma'_{i'+1}\subset
D'$.
The \emph{vertex combination of $(D,(\gamma_i)_i)$ and $(D',(\gamma_i')_i)$
  at $v$ and $v'$}  is the combinatorial geodesic $(n+n'-2)$--gon
whose underlying CW-complex is the wedge sum of $D$ and $D'$ at $v$ and $v'$, and whose
sides
are: \[\gamma_0,\dots,\gamma_{i-1},\gamma_i\gamma'_{i'+1},\gamma'_{i'+2},\dots,\gamma'_{i'-1},\gamma'_{i'}\gamma_{i+1},\gamma_{i+1},\dots,\gamma_{n-1}\]
This is the unique way to define the sides so that the canonical
embeddings of $(D,(\gamma_i)_i)$ and $(D',(\gamma'_i)_i)$ factor through
inclusion into the wedge sum and the canonical embedding of the
resulting combinatorial geodesic $(n+n'-2)$--gon.

\subsubsection{Edge reduction}
Suppose $e\subset D$ is an interior edge of $D$ such that the boundary of $e$ is contained in the
boundary of $D$.
Suppose further that the only maximal cells that intersect $e$ are the two
faces that intersect its interior. 
The hypotheses imply that $e$ separates $D$ into two components, and the $(3,7)$--condition implies that $e$ has two distinct boundary
vertices. 
Suppose that each of these boundary vertices belongs to exactly one
side, which is necessarily true if $(D,(\gamma_i)_i)$ is
non-degenerate. 
Define the \emph{edge reduction of $(D,(\gamma_i)_i)$ at $e$} to be the
two combinatorial geodesic polygons obtained by collapsing $e$ to a
vertex and then performing vertex reduction at the resulting vertex.

The inverse operation to edge reduction we denominate
\emph{edge combination}.
Suppose each of $v\in D$ and $v'\in D'$ is a distinguished vertex that
is contained in a single face and in no other maximal cell.
First perform a vertex combination at
$v$ and $v'$ and then blow up the wedge point to an interior
edge, while keeping the same sides.
As before, we require that $v$ and $v'$ each belong to a single face
and no other maximal cell, which implies that the resulting
combinatorial geodesic polygon is uniquely determined.

\subsubsection{Face reduction}
Suppose $\Pi\subset D$ is a face with $e(\Pi)\geqslant 2$. 
Suppose that there are boundary edges $e$ and $e'$ of $\Pi$ that are
boundary edges of $D$ such that removing the union of the interiors of $\Pi$,
$e$, and $e'$ separates $D$ into two components, $D_1$ and $D_2$.
Suppose that $e$ and $e'$ each intersect only one side of $D$, which is
necessarily true if $(D,(\gamma_i)_i)$ is non-degenerate.
Finally, suppose that $D_1$ and $D_2$ each contain a distinguished vertex.
Define the \emph{face reduction of $(D,(\gamma_i)_i)$ at $(\Pi,e,e')$},
or at $\Pi$, when $e$ and $e'$ are clear, to be the two combinatorial
geodesic polygons obtained by subdividing $e$ and $e'$, subdividing
$\Pi$ by adding a new edge connecting the subdivision points of $e$
and $e'$, and then performing an edge reduction on this new edge.

The inverse operation to face reduction, which we denominate
\emph{face combination}, is to first perform edge combination, which
results in a new
interior edge in the boundary of exactly two faces, and then forget
this new edge, replacing the two incident faces by a single face, and replacing the four resulting edges by two.

\subsection{Combinatorial geodesic bigons and triangles} 

We state Strebel's classification of combinatorial geodesic
bigons and triangles. 
Let us stress again that we are working in the combinatorial setting,
cf \fullref{remark:combinatorialvssc}. 
Strebel's original statement includes that the diagram $D$ comes from
a small cancellation presentation, but what is actually used in the
proof are the properties of the diagram that we have encapsulated in
the definition of `combinatorial geodesic polygon',
\fullref{def:combinatorialgeodesicpolygon}.
This observation was first made in \cite[Section~2.5 and Remark~4.7]{GruSis14} and \cite[Lemma~4.7]{Gru15SQ}.
\begin{thm}[Strebel's classification\footnote{Strebel also considers a
    second definition of combinatorial geodesic polygon that yields one
    additional shape $\mathrm{III}_2$. This is not relevant for us,
    but we retain the subscript for shape $\mathrm{III}_1$ for
    consistency with Strebel's notation.}, {\cite[Theorem 43]{Str90}}]\label{thm:strebel} Let $D$ be a simple diagram that is not a single face.
\begin{itemize}
 \item If $D$ is a combinatorial geodesic bigon, then $D$ has shape $\mathrm{I_1}$ in \fullref{fig:strebel}.
 \item If $D$ is a combinatorial geodesic triangle, then $D$ has one
   of the shapes $\mathrm{I}_2$, $\mathrm{I}_3$, $\mathrm{II}$,
   $\mathrm{III}_1$, $\mathrm{IV}$, or $\mathrm{V}$ in \fullref{fig:strebel}.
\end{itemize}
\end{thm}

\begin{figure}[h]
\captionsetup[subfigure]{labelformat=empty}
  \centering\hfill%
  \begin{subfigure}[h]{.33\textwidth}
\centering
\includegraphics[width=\mytriwidth]{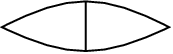}
    \caption{$\mathrm{I}_1$}
  \end{subfigure}\hfill%
\begin{subfigure}[h]{.33\textwidth}\centering
\includegraphics[width=\mytriwidth]{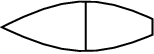}
    \caption{$\mathrm{I}_2$}
  \end{subfigure}\hfill%
\begin{subfigure}[h]{.33\textwidth}\centering
\includegraphics[width=\mytriwidth]{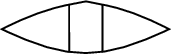}
    \caption{$\mathrm{I}_3$}
  \end{subfigure}\hfill

\begin{subfigure}[h]{.25\textwidth}\centering
\includegraphics[width=\mytriwidth]{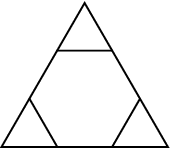}
    \caption{$\mathrm{II}$}
  \end{subfigure}\hfill%
\begin{subfigure}[h]{.25\textwidth}\centering
\includegraphics[width=\mytriwidth]{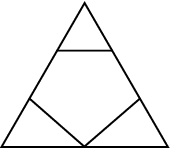}
    \caption{$\mathrm{III}_1$}
  \end{subfigure}\hfill%
\begin{subfigure}[h]{.25\textwidth}\centering
\includegraphics[width=\mytriwidth]{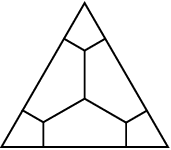}
    \caption{$\mathrm{IV}$}
  \end{subfigure}\hfill%
\begin{subfigure}[h]{.25\textwidth}\centering
\includegraphics[width=\mytriwidth]{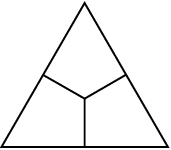}
    \caption{$\mathrm{V}$}
  \end{subfigure}\hfill%

  \caption{Strebel's classification of combinatorial geodesic bigons
    and triangles.}
  \label{fig:strebel}
\end{figure}

Note that shapes $\mathrm{I}_2$ and $\mathrm{I}_3$ degenerate to
combinatorial geodesic bigons.

Each of these shapes represents an infinite family of combinatorial
geodesic bigons or triangles obtained by performing face combination at a
non-degenerate, distinguished vertex with a
shape $\mathrm{I}_1$ bigon arbitrarily many times. 
\fullref{fig:altV} shows alternate examples of each shape.

\begin{figure}[h]
\captionsetup[subfigure]{labelformat=empty}
  \centering\hfill%
  \begin{subfigure}[h]{.33\textwidth}
\centering
\includegraphics[width=\mytriwidth]{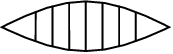}
    \caption{$\mathrm{I}_1$}
  \end{subfigure}\hfill%
\begin{subfigure}[h]{.33\textwidth}\centering
\includegraphics[width=\mytriwidth]{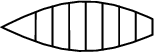}
    \caption{$\mathrm{I}_2$}
  \end{subfigure}\hfill%
\begin{subfigure}[h]{.33\textwidth}\centering
\includegraphics[width=\mytriwidth]{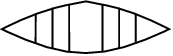}
    \caption{$\mathrm{I}_3$}
  \end{subfigure}\hfill

\begin{subfigure}[h]{.25\textwidth}\centering
\includegraphics[width=\mytriwidth]{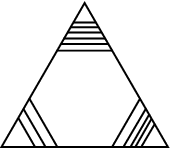}
    \caption{$\mathrm{II}$}
  \end{subfigure}\hfill%
\begin{subfigure}[h]{.25\textwidth}\centering
\includegraphics[width=\mytriwidth]{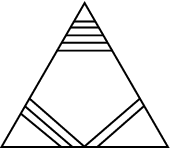}
    \caption{$\mathrm{III}_1$}
  \end{subfigure}\hfill%
\begin{subfigure}[h]{.25\textwidth}\centering
\includegraphics[width=\mytriwidth]{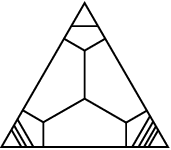}
    \caption{$\mathrm{IV}$}
  \end{subfigure}\hfill%
\begin{subfigure}[h]{.25\textwidth}\centering
\includegraphics[width=\mytriwidth]{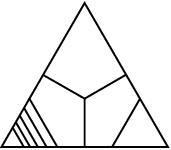}
    \caption{$\mathrm{V}$}
  \end{subfigure}\hfill%

  \caption{Alternate examples of each shape.}
  \label{fig:altV}
\end{figure}

\subsection{Special combinatorial geodesic quadrangles}
In this section diagram faces are labelled with their contribution to
the curvature sum (Lemma~\ref{lem:curvature}) if this contribution is non-zero.

\begin{defi}[Special]
A combinatorial geodesic $n$--gon, for $n>2$, is \emph{special} if it is
simple, non-degenerate, irreducible, and every non-distinguished vertex has valence 3.
\end{defi}

The only special combinatorial geodesic triangles are the
representatives of shapes $\mathrm{IV}$ and $\mathrm{V}$ pictured in 
\fullref{fig:strebel}.
In this section we classify special combinatorial
geodesic quadrangles.

Let $D$ be a special combinatorial geodesic polygon.
Simplicity and trivalence imply that $\Pi\cap \bdry D$ is a disjoint union
of arcs, for each face $\Pi$. 
Irreducibility implies $\Pi\cap \bdry D$ consists of at most one arc.
Non-degeneracy implies that every distinguished face has
exactly one distinguished vertex and either 2 or 3 interior arcs.

The curvature formula of \fullref{lem:curvature} can be simplified as follows.
\begin{lem}[Special curvature formula]
 Let $D$ be a special combinatorial geodesic polygon that is not a single face. Then:
\[6=\sum_{e(\Pi)=0}(6-i(\Pi))+\sum_{e(\Pi)=1}(4-i(\Pi))\]
 \end{lem}
 
 \begin{proof}
 By definition, every non-distinguished vertex has valence 3. To apply \fullref{lem:curvature}, 
   we must address the possible existence of degree 2 vertices in the
   boundary: we may iteratively remove such vertices, always replacing
   the two adjacent edges by a single edge. Since $D$ is not a single
   face, this makes sense for every degree 2 vertex, and we thus
   remove all degree 2 vertices. Since the degree of a face counts
   \emph{arcs}, not edges, the operation does not alter the sum. The
   formula follows by applying
   \fullref{lem:curvature}, and noting, as a consequence of irreducibility, that every face has exterior degree at most 1.
 \end{proof}

For the remainder of this section, let $D$ be a special combinatorial
geodesic quadrangle.
Then $D$ has exactly four distinguished faces, each of which
contributes either 1 or 2 to the curvature sum, and every other face
makes a non-positive contribution.
Let $D_k$ refer to the set of distinguished faces of $D$ that contribute $k$ to the
curvature sum, ie, with $e(\Pi)=1$ and $i(\Pi)=4-k$.

An \emph{ordinary} face will refer to a non-distinguished face $\Pi$ with $e(\Pi)=1$ and $i(\Pi)=4$, which contributes 0 to the curvature sum.
An \emph{extraordinary} face will refer to a non-distinguished face
$\Pi$ with $e(\Pi)=0$ or with $e(\Pi)=1$ and $i(\Pi)>4$. Note that if $i(\Pi)>6$, we must have $e(\Pi)=0$.

\subsubsection{Zippers}
Ordinary faces can fit together to make arbitrarily long sequences of
subsequent faces we call \emph{zippers}, as in \fullref{fig:zipper}.
\begin{figure}[h]
  \centering
  \includegraphics[height=\myquadheight]{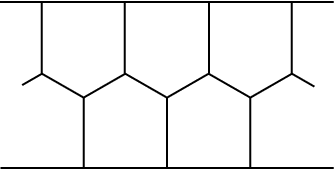}
  \caption{A zipper.}
  \label{fig:zipper}
\end{figure}

The ordinary faces in a zipper are called \emph{teeth}.
We define a \emph{zipper $Z$ with zero teeth} to be three consecutive
interior edges that separate the diagram into two parts, each of which
contains two distinguished faces.
For example, the bold edges of \fullref{fig:zippernoteeth} form a zipper
with zero teeth.
Since $D$ is special, the two interior
 edges incident to the interior vertices of $Z$ and not
 belonging to $Z$ must be contained in opposite complementary components
 of $Z$, otherwise $D$ would admit a face reduction.

 \begin{figure}[h]
   \centering
\labellist
\small
\pinlabel 1 at 29 57
\pinlabel 2 at 20 18
\pinlabel 1 at 74 22
\pinlabel 2 at 82 62
\endlabellist
   \includegraphics[height=\myquadheight]{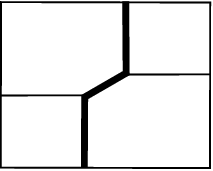}
   \caption{Zipper with zero teeth (in bold).}
   \label{fig:zippernoteeth}
 \end{figure}

Using a symmetry argument, we show that the portions of the diagram on opposite sides of a zipper each contribute 3 to the curvature sum: first, consider a zipper $Z$ of length $0$. Then the two interior edges incident at $Z$ are not on the same side of $Z$, for otherwise we would have a face with exterior degree at least 2 (and hence face-reducibility) or only one distinguished face on that side. Now assume that one of the two sides  $S$ contributes $k\neq 3$ to the curvature sum. Then we may rotate a copy of $S$ by 180 degrees and attach it to $S$ by identifying the respective copies of $Z$, thus obtaining a special combinatorial quadrangle for which the curvature formula amounts to $2k\neq 6$; a contradiction. The case of an arbitrary zipper $Z$ now follows similarly by attaching a rotated (or in the case that $Z$ has an odd number of faces reflected) copy of $S$ to $S\cup Z$.

We need to see how to terminate a zipper.
Let $\Pi$ be the face with two edges on the zipper. 
If $\Pi$ is ordinary then the zipper just gets longer, so assume not. 
One possibility is that $\Pi$ is distinguished, in which case it is a $D_1$ and there is only one other face, which is a $D_2$.
Otherwise, since the end of the zipper containing $\Pi$ must contribute 3 to the curvature sum, we must have $i(\Pi)=5$ and $e(\Pi)=1$, both distinguished faces must be $D_2$'s, and the other two faces sharing edges with $\Pi$ must be ordinary.

The two possibilities are shown in \fullref{fig:zipperends}.
In conclusion:

\begin{lem}\label{lem:enumeratezippered}
  There are six infinite families of configurations of special
  combinatorial geodesic quadrangles containing zippers, determined
  by the choice of two zipper ends from \fullref{fig:zipperends} and
  the parity of the number of teeth.
\end{lem}

\begin{figure}[h]
  \centering
\hfill
  \begin{subfigure}[h]{.5\textwidth}
  \centering  
\labellist
\tiny
\pinlabel $\Pi$ at 40 50
\small
\pinlabel 1 at 30 57
\pinlabel 2 at 22 18
\endlabellist
\includegraphics[height=\myquadheight]{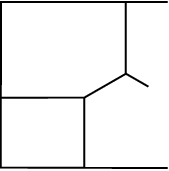}
    \caption{Zipper end 1.}
\label{fig:zipend1}
  \end{subfigure}\hfill%
\begin{subfigure}[h]{.5\textwidth}
  \centering  
\labellist
\tiny
\pinlabel $\Pi$ at 80 50
\small
\pinlabel $-1$ at 58 57
\pinlabel 2 at 14 12
\pinlabel 2 at 14 62
\endlabellist
\includegraphics[height=\myquadheight]{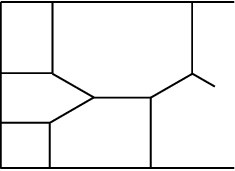}
    \caption{Zipper end 2.}
\label{fig:zipend2}
  \end{subfigure}\hfill
  \caption{Terminating a zipper}
  \label{fig:zipperends}
\end{figure}
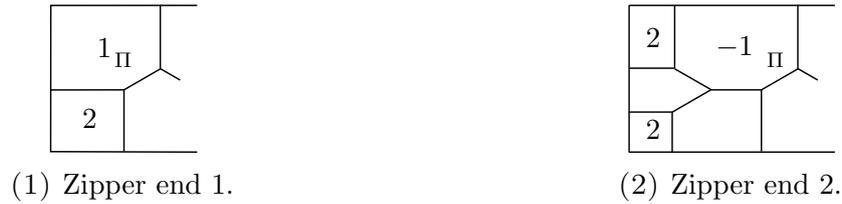

\subsubsection{Extraordinary configurations}
\begin{lem}\label{lem:enumerateextraordinary}
  The six configurations shown in \fullref{fig:extraordinary} are the
  only special combinatorial geodesic quadrangles containing an
  extraordinary face and no zipper.
\end{lem}

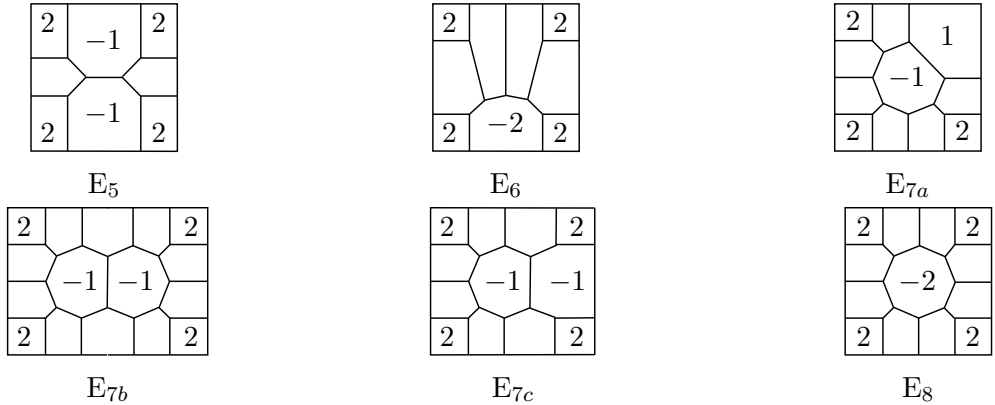
\begin{figure}[h]
\captionsetup[subfigure]{labelformat=empty}
  \centering
\hfill
\begin{subfigure}[h]{.33\textwidth}
  \centering  
  \labellist
  \small
\pinlabel 2 at 10 10
\pinlabel 2 at 10 70
\pinlabel 2 at 70 70
\pinlabel 2 at 70 10
\pinlabel $-1$ at 40 20
\pinlabel $-1$ at 40 60
  \endlabellist
  \includegraphics[height=\myquadheight]{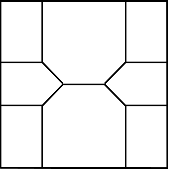} 
  \caption{$\mathrm{E}_{5}$}
\end{subfigure}\hfill%
\begin{subfigure}[h]{.33\textwidth}
  \centering  
  \labellist
  \small
\pinlabel 2 at 10 10
\pinlabel 2 at 10 70
\pinlabel 2 at 70 70
\pinlabel 2 at 70 10
\pinlabel $-2$ at 40 15
  \endlabellist
  \includegraphics[height=\myquadheight]{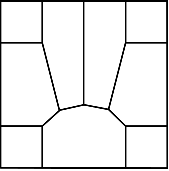} 
  \caption{$\mathrm{E}_{6}$}
\end{subfigure}\hfill%
\begin{subfigure}[h]{.33\textwidth}
  \centering  
  \labellist
  \small
  \pinlabel $-1$ at 40 40
  \pinlabel 2 at 11 71
  \pinlabel 2 at 11 11
\pinlabel 2 at 70 11
\pinlabel 1 at 62 63
  \endlabellist
  \includegraphics[height=\myquadheight]{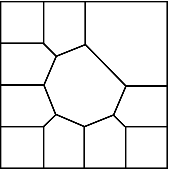} 
  \caption{$\mathrm{E}_{7a}$}
\end{subfigure}
\hfill

\hfill
\begin{subfigure}[h]{.33\textwidth}
  \centering  
  \labellist
  \small
\pinlabel $-1$ at 40 40
\pinlabel $-1$ at 70 40
\pinlabel $2$ at 10 70
\pinlabel $2$ at 10 10
\pinlabel $2$ at 100 10
\pinlabel $2$ at 100 70
  \endlabellist
  \includegraphics[height=\myquadheight]{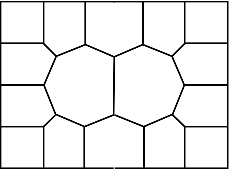} 
  \caption{$\mathrm{E}_{7b}$}
\end{subfigure}\hfill%
\begin{subfigure}[h]{.33\textwidth}
  \centering  
  \labellist
  \small
\pinlabel $-1$ at 40 40
\pinlabel $-1$ at 75 40
\pinlabel $2$ at 10 70
\pinlabel $2$ at 10 10
\pinlabel $2$ at 80 10
\pinlabel $2$ at 80 70
  \endlabellist
  \includegraphics[height=\myquadheight]{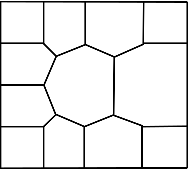} 
  \caption{$\mathrm{E}_{7c}$}
\end{subfigure}\hfill%
\begin{subfigure}[h]{.33\textwidth}
  \centering
\labellist
\small
\pinlabel 2 at 11 10
\pinlabel 2 at 11 70
\pinlabel 2 at 70 10
\pinlabel 2 at 70 70
\pinlabel $-2$ at 40 40
\endlabellist
  \includegraphics[height=\myquadheight]{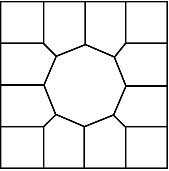}
  \caption{$\mathrm{E}_8$}
\end{subfigure}
  \caption{Extraordinary special combinatorial geodesic quadrangles.}
  \label{fig:extraordinary}
\end{figure}

\begin{proof}
Let $D$ be a special combinatorial geodesic quadrangle without
zippers.
First, suppose that $D$ contains no interior faces.

\paragraph{Case $\mathrm{E}_5$: $D$ contains an extraordinary face
  $\Pi$ with $i(\Pi)=5$.} In this case $\Pi$ contributes $-1$ to the curvature sum.
Consider the third interior edge $e$ of $\Pi$. 
Let $\Pi'$ be the face on the opposite side of $e$.
If $e$ is the second interior edge of $\Pi'$ then we get a zipper, so it must be at least the third.
(It cannot be the first, for this would give a vertex of degree at least 4.)

By symmetry, we see that $i(\Pi')\geqslant 5$. As $\Pi'$ contributes at least $-1$ to the curvature sum and is not interior, we deduce $i(\Pi')=5$ and $e(\Pi')=1$, 
with $e$ the third interior edge of $\Pi'$.
This then implies that the distinguished faces are all $D_2$'s, and every other face is ordinary. There are two ordinary faces bordering both $\Pi$ and $\Pi'$. 
Since no face has more than one exterior edge, we must then fill in the four $D_2$'s on the corners.

\smallskip
\paragraph{Case $\mathrm{E}_6$: $D$ contains an extraordinary face
  $\Pi$ with $i(\Pi)=6$.} In this case $\Pi$ contributes $-2$ to the
curvature sum, so the distinguished faces are all $D_2$'s, and every
other face is ordinary.
Consider the third interior vertex of $\Pi$.
Let $e$ be the edge incident to this vertex that does not belong to
$\Pi$. 
If $e$ does not have a vertex on the boundary then there is a zipper
contained in the boundary of $\Pi\cup\Pi'$, where $\Pi'$ is either one
of the ordinary faces with side $e$.
Since we have assumed no zippers, $e$ does have a boundary
vertex, and there is a unique way to fill in the rest of $D$ with ordinary faces and $D_2$'s.

\smallskip
Now we move on to the interior face cases.
According to the curvature formula of \fullref{lem:curvature},
interior faces have either 7 or 8 sides.

\smallskip
\paragraph{Case $\mathrm{E}_8$: $D$ contains an extraordinary face
  $\Pi$ with $i(\Pi)=8$.}
In this case, $e(\Pi)=0$, and $\Pi$ contributes $-2$ to the curvature sum, so all four distinguished faces are $D_2$ and every other face is ordinary.
Since the interior sides of a $D_2$ have a vertex on the boundary,
they cannot share an edge with $\Pi$, so every face sharing an edge
with $\Pi$ is ordinary.
There is only one way to pack 8 ordinary faces around $\Pi$, up to
symmetry, and this determines the placement of the four $D_2$'s.

\smallskip
\paragraph{Case $\mathrm{E}_7$: $D$ contains an extraordinary face
  $\Pi$ with $i(\Pi)=7$.}
In this case, $e(\Pi)=0$, and $\Pi$ contributes $-1$ to the curvature sum. 
The distinguished faces are therefore either three $D_2$'s and one $D_1$ or four $D_2$'s.
If there are three $D_2$'s and one $D_1$ then every other face is ordinary. 
A $D_2$ cannot share an edge with an interior face, so $\Pi$ has at least 6 edges that are shared by ordinary faces.

If there are four $D_2$'s then none of them share a face with $\Pi$, and there is exactly one other face that is not ordinary. 

In either case, $\Pi$ has at least 6 edges that are shared by ordinary faces.
Let $\Pi_1,\dots,\Pi_7$ be the consecutive faces sharing an edge with
$\Pi$, and assume all except possibly $\Pi_7$ are ordinary. 
Let $e$ be the edge shared by $\Pi_1$ and $\Pi_7$. 
If $e$ has a vertex on the boundary then there is only one way to fit
6 ordinary faces around $\Pi$.
In this case, $\Pi_7$ is a $D_1$, and the configuration is shown in \fullref{fig:E7a1}.

If $e$ does not have a vertex on the boundary then, again, there is only one way to fit six ordinary faces around $\Pi$, shown in \fullref{fig:E7b1}.
We see that $i(\Pi_7)\geqslant 5$. Since $\Pi_7$ contributes at least $-1$ to the curvature sum, the two possibilities are $i(\Pi_7)=7$ and $e(\Pi_7)=0$ or $i(\Pi_7)=5$ and $e(\Pi_7)=1$. 
In both cases, the remaining distinguished faces are $D_2$'s, all
other faces are ordinary, and there is a unique way to complete the
4--gon.
These are types $\mathrm{E}_{7b}$ and $\mathrm{E}_{7c}$, respectively, of \fullref{fig:extraordinary}.\qedhere

\begin{figure}[h]
  \centering  
\hfill
\begin{subfigure}[h]{.45\textwidth}
  \centering  
  \labellist
  \small
  \pinlabel $-1$ at 40 40
  \pinlabel 2 at 11 71
  \pinlabel 2 at 11 11
\pinlabel 2 at 70 11
\pinlabel 1 at 62 63
\tiny
\pinlabel $\Pi_1$ at 33 65
\pinlabel $\Pi_2$ at 16 50
\pinlabel $\Pi_3$ at 16 32
\pinlabel $\Pi_4$ at 32 16
\pinlabel $\Pi_5$ at 52 16
\pinlabel $\Pi_6$ at 66 31
\pinlabel $\Pi_7$ at 57 53
\pinlabel $e$ at 46 70 
  \endlabellist
  \includegraphics[height=\myquadheight]{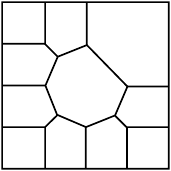} 
  \caption{Case $e$ does have boundary vertex.}
  \label{fig:E7a1}
\end{subfigure}
\hfill
\begin{subfigure}[h]{.45\textwidth}
  \centering  
  \labellist
  \small
  \pinlabel $-1$ at 40 40
  \pinlabel 2 at 11 71
  \pinlabel 2 at 11 11
\tiny
\pinlabel $\Pi_2$ at 33 65
\pinlabel $\Pi_3$ at 16 50
\pinlabel $\Pi_4$ at 16 32
\pinlabel $\Pi_5$ at 32 16
\pinlabel $\Pi_6$ at 52 16
\pinlabel $\Pi_7$ at 66 41
\pinlabel $\Pi_1$ at 54 65 
\pinlabel $e$ at 67 55
  \endlabellist
  \includegraphics[height=\myquadheight]{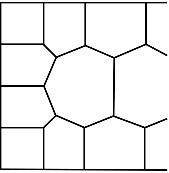} 
  \caption{Case $e$ does not have boundary vertex.}
  \label{fig:E7b1}
\end{subfigure}\hfill%
  \caption{Interior 7--gon.}
  \label{fig:E7}
\end{figure}
\end{proof}

\subsubsection{Classification of special combinatorial geodesic
  quadrangles}
\begin{thm}\label{thm:classificationofspecialquads}
  Every special combinatorial geodesic quadrangle is either one of
  the six extraordinary configurations of
  \fullref{lem:enumerateextraordinary} or belongs to one of the six
zippered families of \fullref{lem:enumeratezippered}.
\end{thm}
The theorem is proven by \fullref{lem:enumerateextraordinary},
\fullref{lem:enumeratezippered}, and the following:
\begin{lem}\label{lem:specialhaszipperorextraordinary}
  Every special combinatorial geodesic quadrangle contains an
  extraordinary face or a zipper.
\end{lem}
\begin{proof}
Suppose $D$ is a special combinatorial geodesic quadrangle that
does not contain an extraordinary face.
Then no face makes a negative contribution to the curvature sum, so
 $D$ is composed of two $D_2$'s, two $D_1$'s, and some number of ordinary faces.
 
Pick a side of the quadrangle. 
Let $S$ be the union of faces along the side. 
Let $A$ be the union of interior edges separating $S$ from $D\setminus S$.
Since every vertex has valence 3, both $S$ and $D\setminus S$ are connected, and each contains two distinguished faces.

Consider the edges incident to interior vertices of $A$. 
Each one is contained either in $S$ or in $D\setminus S$. 
At least one is contained in $S$ and one in $D\setminus S$, since each side contains two distinguished vertices and no face separates $D$.
Two consecutive edges cannot point into $D\setminus S$, because the
face $\Pi\subset S$ containing the edge between them would either be
extraordinary, contradicting the hypothesis, or a distinguished face 
with interior degree at least four, contradicting non-degeneracy.

Two consecutive edges cannot point into $S$, because the face
$\Pi\subset S$ between
them would be non-distinguished with $i(\Pi)=3$.
Therefore the edges along $A$ alternate, and $A$ has length at least 3.

If $A$ consists of 3 edges then it is a zipper, and we are done, so suppose it consists
of at least 4 edges.

If the first face $\Pi$ adjacent to $A$ on the $D\setminus S$ side is
distinguished then we get a zipper, so suppose it is not.
Let $\Pi'$ be the next face along $A$ on the $D\setminus S$ side.
The two possibilities are shown in \fullref{fig:sidestrip1} and \fullref{fig:sidestrip2}.

\smallskip
\begin{figure}[h]
  \centering
\hfill
  \begin{subfigure}[h]{.33\textwidth}
\centering      
\labellist
\small
\pinlabel $S$ [b] at 12 80
\pinlabel $D\setminus S$ [b] at 50 79
\tiny
\pinlabel $\Pi$ at 30 72.5
\pinlabel $\Pi'$ at 30 50
\endlabellist
  \includegraphics[height=\myquadheight]{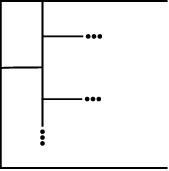}
\caption{Case 1}
\label{fig:sidestrip1}
  \end{subfigure}\hfill
  \begin{subfigure}[h]{.33\textwidth}
\centering
\labellist
\small
\pinlabel $S$ [b] at 12 80
\pinlabel $D\setminus S$ [b] at 50 79
\tiny
\pinlabel $\Pi$ at 30 64
\pinlabel $\Pi'$ at 30 33
\endlabellist
    \includegraphics[height=\myquadheight]{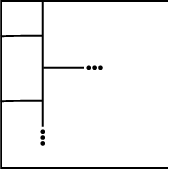}
\caption{Case 2}
\label{fig:sidestrip2}
  \end{subfigure}
\hfill
  \begin{subfigure}[h]{.33\textwidth}
\centering
\labellist
\small
\pinlabel $S$ [b] at 12 80
\pinlabel $D\setminus S$ [b] at 50 79
\tiny
\pinlabel $\Pi$ at 30 64
\pinlabel $\Pi'$ at 30 33
\pinlabel $\Pi''$ at 60 50
\endlabellist
    \includegraphics[height=\myquadheight]{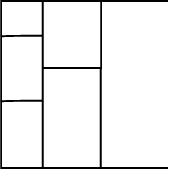}
\caption{Case 2 - next level}
\label{fig:sidestrip2f}
  \end{subfigure}
\hfill
  \caption{Quadrangle with no zipper or extraordinary face.}
  \label{fig:lastcase}
\end{figure}
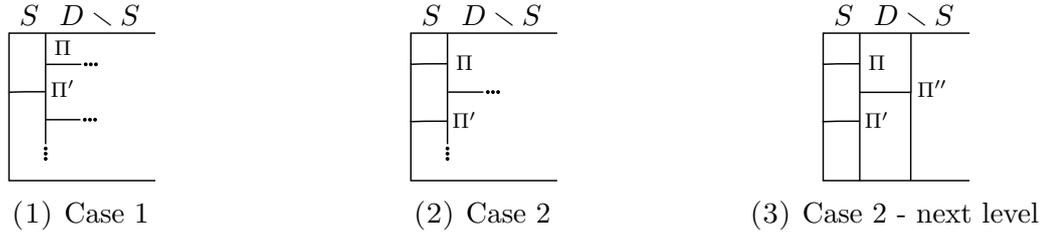

Since $\Pi$ is not distinguished, the edge shared by $\Pi$ and $\Pi'$
does not contain a boundary vertex.
In the first case $i(\Pi')\geqslant 5$, using that $A$ has length at least 4, contrary to hypothesis.
In the second case, $i(\Pi)\geqslant 4$ and $i(\Pi')\geqslant 4$, so both are
ordinary, and we have the situation in \fullref{fig:sidestrip2f}.
Let $\Pi''$ be the face of $D\setminus S$ adjacent to $\Pi$ and
$\Pi'$.
Either $\Pi''$ contains more than one boundary arc, or $\Pi''$ is
distinguished and contains two distinguished vertices, but both of these are
contrary to hypothesis.
\end{proof}

\subsection{Quadrangle dichotomy} The following proposition says that a combinatorial geodesic
quadrangle must be either \emph{short}, conditions (\ref{item:separatingquad}) or
(\ref{item:zipper}), or \emph{thin}, (\ref{item:topsequence}).
\begin{prop}\label{prop:combinatorialngons}
Let $D$ be a simple combinatorial geodesic quadrangle with boundary path $\gamma_1\delta_1\gamma_2^{-1}\delta_2^{-1}$. Then one of the following holds:

\begin{enumerate}
 \item There exists a face $\Pi$ that intersects both $\gamma_1$ and $\gamma_2$ in edges 
with $2=e(\Pi)=i(\Pi)$.\label{item:separatingquad}
 \item There exist faces $\Pi$ intersecting $\gamma_1$ in edges, and $\Pi'$ and $\Pi''$, both intersecting $\gamma_2$ in edges, such that $1=e(\Pi)=e(\Pi')=e(\Pi'')$, $4=i(\Pi)=i(\Pi')=i(\Pi'')$, and any two of $\Pi$, $\Pi'$ and $\Pi''$ pairwise intersect in edges. Moreover, $\Pi'\cap \Pi''$ is an arc connecting $\gamma_2$ to $\Pi$, and $\Pi\cap (\Pi'\cup\Pi'')$ is connected.\label{item:zipper}
 \item  There exist $k\leqslant 6$ and faces $\Pi_1,\Pi_2,\dots,\Pi_k$, each intersecting $\gamma_1$ in edges, such that $\Pi_1\cap\delta_1\neq \emptyset$ and $\Pi_k\cap \delta_2\neq \emptyset$ and, for each $1\leqslant i<k$, we have $\Pi_i\cap \Pi_{i+1}\neq\emptyset$.\label{item:topsequence}
\end{enumerate}
\end{prop}
The three cases are pictured in Figures \ref{fig:tricquad},
\ref{fig:triczip}, and \ref{fig:tricthin}, respectively.

\fullref{corollary:faceseqsegment} implies that in case (\ref{item:topsequence}) of
\fullref{prop:combinatorialngons} the set $\bigcup_{i=1}^k\Pi_i\cap\gamma_1$
is a path subgraph of $\gamma_1$.

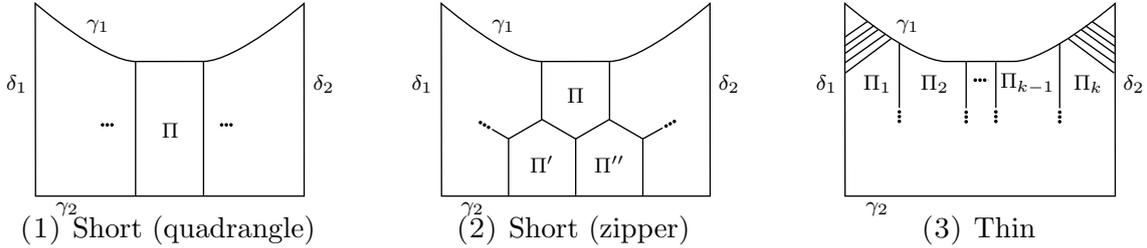
\begin{figure}[h]
  \centering
\hfill
  \begin{subfigure}[h]{.33\textwidth}
\centering      
\labellist
\tiny
\pinlabel $\Pi$ at 81 39
\pinlabel $\gamma_1$ [bl] at 28 94
\pinlabel $\gamma_2$ [t] at 20 1
\pinlabel $\delta_1$ [r] at 1 67
\pinlabel $\delta_2$ [l] at 161 67
\endlabellist
  \includegraphics[height=1.3\myquadheight]{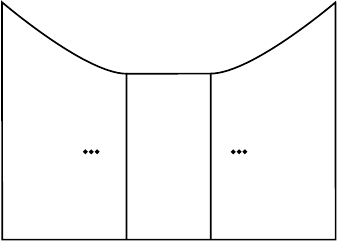}
\caption{Short (quadrangle)}
\label{fig:tricquad}
  \end{subfigure}\hfill
\begin{subfigure}[h]{.33\textwidth}
\centering      
\labellist
\tiny
\pinlabel $\Pi$ at 81 60
\pinlabel $\Pi'$ at 61 19
\pinlabel $\Pi''$ at 101 19
\pinlabel $\gamma_1$ [bl] at 28 94
\pinlabel $\gamma_2$ [t] at 20 1
\pinlabel $\delta_1$ [r] at 1 67
\pinlabel $\delta_2$ [l] at 161 67
\endlabellist
  \includegraphics[height=1.3\myquadheight]{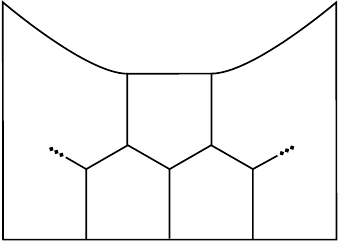}
\caption{Short (zipper)}
\label{fig:triczip}
  \end{subfigure}\hfill
\begin{subfigure}[h]{.33\textwidth}
\centering      
\labellist
\tiny
\pinlabel $\Pi_1$ at 20 67
\pinlabel $\Pi_2$ at 54 67
\pinlabel $\Pi_{k-1}$ at 109 67
\pinlabel $\Pi_k$ at 145 67
\pinlabel $\gamma_1$ [bl] at 28 94
\pinlabel $\gamma_2$ [t] at 20 1
\pinlabel $\delta_1$ [r] at 1 67
\pinlabel $\delta_2$ [l] at 161 67
\endlabellist
  \includegraphics[height=1.3\myquadheight]{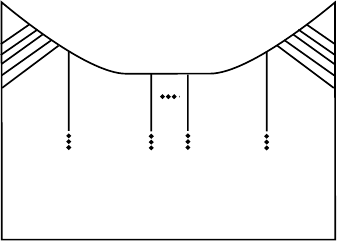}
\caption{Thin}
\label{fig:tricthin}
  \end{subfigure}\hfill
  \caption{Quadrangle dichotomy}
\end{figure}

\begin{proof}[Proof of \fullref{prop:combinatorialngons}]
If $D$ is a single face we are done, so suppose not.

  Suppose that the first two conditions do not hold. 
Let $k>0$ be minimal such that $\Pi_1,\dots,\Pi_k$ is a sequence of faces
satisfying condition (\ref{item:topsequence}).
Minimality is equivalent to requiring $\Pi_i\cap\delta_1=\emptyset$
and $\Pi_i\cap\delta_2=\emptyset$ for all $1<i<k$.

We may also assume $D$ is non-degenerate, since for triangles we have $k\leqslant 3$, by \fullref{thm:strebel}.

Since $D$ is a simple $(3,7)$--diagram it has no valence 1 vertices. 
If it has non-distinguished valence 2 vertices we can forget them.
We then arrange that non-distinguished vertices of $D$ have valence 3
as follows. 

Suppose $v$ is a vertex of valence greater than 3.
Since $D$ is simple, for any two faces $\Pi$ and $\Pi'$ that contain
$v$ in their boundaries and do not share an edge incident to $v$,
there is a unique way to blow up $v$ to an interior edge $e_v$ in such
a way that the images of $\Pi$ and $\Pi'$ are the faces with $e_v$ in their
boundaries.
Let $\beta$ denote the blow-up map.
Since $\Pi$ and $\Pi'$ do not share an edge incident to $v$, the two
vertices of $e_v$ each have valence at least 3 and strictly less than
that of $v$. 

The only effect of this blow-up on faces is to increase the interior
degrees of $\Pi$ and $\Pi'$ by 1 each, so $\beta(D)$ is
still a simple combinatorial geodesic quadrangle.
The faces $\beta(\Pi_1),\dots,\beta(\Pi_k)$ still satisfy condition
(\ref{item:topsequence}), since $\beta$ only introduces an interior
edge. 
We now check that $\beta$ does not produce a diagram
satisfying conditions (\ref{item:separatingquad}) or
(\ref{item:zipper}) and not condition (\ref{item:topsequence}).

We argue for the face $\Pi$. The same arguments apply for $\Pi'$.
Since $D$ is simple $\Pi$ cannot have $e(\Pi)=2$ and $i(\Pi)=1$, so 
$\beta(\Pi)$ does not satisfy condition (\ref{item:separatingquad}).

If $e(\Pi)>1$ then $e(\beta(\Pi))=e(\Pi)>1$, so $\beta(\Pi)$ cannot be one of the faces
satisfying condition (\ref{item:zipper}).
Suppose now that $e(\Pi)=1$.
The blow-up does not change the boundary of $D$, so $\Pi$ is
distinguished if and only if $\beta(\Pi)$ is distinguished.
If $\Pi$ is distinguished with $i(\Pi)\geqslant 3$ then
$\beta(\Pi)$ is distinguished with $i(\beta(\Pi))\geqslant 4$, so
$\beta(D)$ degenerates to a triangle, which implies $k\leqslant 3$, and we are done.
 Otherwise, if $\Pi$ is distinguished then $i(\beta(\Pi))<4$, so $\beta(\Pi)$ cannot be one of the faces
satisfying condition (\ref{item:zipper}).
Finally, if $\Pi$ is non-distinguished then the combinatorial
geodesic polygon condition requires $i(\Pi)\geqslant 4$, so
$i(\beta(\Pi))\geqslant 5$, and $\beta(\Pi)$ cannot be one of the
faces satisfying
condition (\ref{item:zipper}).

We conclude that $k\leqslant 3$, in which case we are done, or
$\beta(D)$ is a simple non-degenerate combinatorial geodesic
quadrangle not satisfying conditions (\ref{item:separatingquad}) or
(\ref{item:zipper}) and containing a sequence $\beta(\Pi_1),\dots,\beta(\Pi_k)$ of
faces satisfying the requirements of condition
(\ref{item:topsequence}).
Moreover, since $\beta$ only introduced an interior edge,
$\beta(\Pi_i)\cap\beta(\delta_j)=\emptyset$ for all $1<i<k$ and
$j\in\{1,2\}$, so  $\beta(\Pi_1),\dots,\beta(\Pi_k)$ is a minimal
length sequence satisfying condition
(\ref{item:topsequence}) in $\beta(D)$.

\medskip
We now repeat blowing up higher valence vertices until either the
quadrangle becomes degenerate, and we are done, or there are no higher
valence vertices left.
Thus, we may assume non-distinguished vertices of $D$ have valence 3.

If $D$ is irreducible then it is special.
If $D$ is special and has no zipper then $k\leqslant 5$, by considering the possibilities
given by \fullref{thm:classificationofspecialquads}.

If $D$ is special and has a zipper then, by considering the possible zippered
configurations of \fullref{lem:enumeratezippered}, $k$ is at most four plus the number of teeth of
the zipper adjacent to $\gamma_1$. 
Since $D$ does not satisfy condition (\ref{item:zipper}), there can be at most two teeth of
the zipper adjacent to $\gamma_1$, so $k\leqslant 6$. 
In fact, $k=6$ can
occur, so 6 is the best possible bound for $k$.

The remaining possibility is that $D$ is reducible.
It is not vertex reducible, since it is simple. 
If an edge reduction is possible then so is a face reduction, so
suppose $\Pi$ is a face with edges $e,\,e'\subset\bdry\Pi$ such that $D$ admits a face reduction at $(\Pi,e,e')$.
Since $D$ is non-degenerate there
are four distinguished faces.

Suppose the face reduction separates one distinguished
face from the other three, which occurs when one of $e$ or $e'$ is an
edge of one of the $\gamma_i$'s and the other is an edge of one of the $\delta_i$'s. 
Then $D$ is a union of a bigon and a quadrangle with fewer faces than
  $D$. 
By minimality of $k$, either $\Pi=\Pi_1$ or $\Pi=\Pi_k$ or
$\Pi\neq\Pi_i$.
Therefore $\Pi_1,\dots,\Pi_k$ corresponds to a sequence of faces in
the new quadrangle still satisfying condition
\ref{item:topsequence}.

Repeat the argument for the new quadrangle.
Since the number of faces in the quadrangle decreases, this process
stops after finitely many steps, so we may assume that $D$ is not
reducible into a bigon and a quadrangle.
Since $D$ is non-degenerate, this implies that every
distinguished face has one exterior arc and either  2 or 3 interior
arcs, and the distinguished faces are the only ones that intersect
both one of the $\delta_i$'s and one of the $\gamma_i$'s.

Suppose there is a face $\Pi$ containing edges $e\subset
\bdry\Pi\cap\delta_1$ and $e'\subset\bdry\Pi\cap\delta_2$.
Face reduction at $(\Pi,e,e')$ sends the $\Pi_i$'s to a minimal length sequence of
faces along one side of a combinatorial geodesic triangle connecting
the other two sides.
Thus, $k\leqslant 3$.

Next, suppose there is an edge $e$ that meets both $\gamma_1$ and
$\gamma_2$. 
Then $e$ separates $\delta_1$ from $\delta_2$, so there is
some $i<k$ such that $\Pi_1,\dots,\Pi_i$ are on one side of $e$ and
$\Pi_{i+1},\dots,\Pi_k$ are on the other. 
Edge reduction of $D$ at $e$ results in two combinatorial geodesic
triangles.
In one of these triangles there is a sequence of faces corresponding
to $\Pi_1,\dots,\Pi_i$ that run along the side corresponding to
$\gamma_1$ and connect the side corresponding to $\delta_1$ to the
opposite distinguished vertex. 
Since we were not in case (\ref{item:separatingquad}), $i(\Pi_i)>2$,
so the face corresponding to $\Pi_i$ in the triangle has more than one
interior arc.
It follows that $i\leqslant 3$. 
The same argument applied to $\Pi_{i+1}$ shows $k-i\leqslant 3$, so
$k\leqslant 6$. (The lower right diagram in
\fullref{fig:combinationreduction} shows $k=6$ can be achieved in this
case.)

Finally, suppose $D$ has a separating face $\Pi_i$ and no separating
edge.
 Face reduction of $D$ at $\Pi_i$ yields two combinatorial geodesic
 triangles.
Since $D$ has no separating edge, each of these triangles has a
distinguished face corresponding to $\Pi_i$ with interior degree
greater than one.
The same argument as in the previous case then tells us $i\leqslant 3$
and $k-i+1\leqslant 3$, so $k\leqslant 5$.\end{proof}


\section{Contraction in $Gr'(\ifrac{1}{6})$-groups}\label{sec:gsccontraction}
In this section, we study contraction in groups defined by $Gr'(\ifrac{1}{6})$--labelled graphs. 
We give a characterization of strongly contracting geodesics in the
Cayley graph in terms of their intersections with embedded components
of the defining graph. 
Throughout this section, $\Gamma$ is an arbitrary $Gr'(\ifrac{1}{6})$--labelled graph with a set of labels $\gens$, and $X:=\Cay(G(\Gamma),\gens)$.

  Recall that a geodesic $\alpha$ in $X$ is \emph{locally $(\rho_1,\rho_2)$--contracting} if for every embedded component $\Gamma_0$ of $\Gamma$ such that $\Gamma_0\cap\alpha\neq\emptyset$, closest point projection in $\Gamma_0$ of $\Gamma_0$ to $\Gamma_0\cap\alpha$ is $(\rho_1,\rho_2)$--contracting.

We will prove the following theorem:
\begin{thm}\label{theorem:strongcontraction}
A geodesic $\alpha$ in $X$ is locally $(\rho_1,\rho_2)$--contracting if and
only if there exist $\rho_1'$ and $\rho_2'$ such that $\alpha$ is
$(\rho_1',\rho_2')$--contracting. 

Moreover, $\rho_1'$ and $\rho_2'$ can be bounded in terms of $\rho_1$
and $\rho_2$, and when $\rho_1(r)\geqslant \ifrac{r}{2}$ we can take
$\rho_1'=\rho_1$ and $\rho_2'\asymp\rho_2$.

Consequently, $\alpha$ is strongly contracting if and only if it is
uniformly locally strongly contracting, $\alpha$ is semi-strongly
contracting if and only if it is locally uniformly semi-strongly
contracting, and $\alpha$ is sublinearly
contracting if and only if it is uniformly locally sublinearly contracting.
\end{thm}

\begin{remark} 
\fullref{theorem:strongcontraction} provides an analogy with the
geometry present in many of the preeminent examples of spaces with a mixture
of hyperbolic and non-hyperbolic behavior, such as Teichm\"uller
space, the Culler-Vogtmann Outer Space, or relatively hyperbolic
spaces.
The simplest situation is that geodesics that avoid spending a long
time in a non-hyperbolic region behave like hyperbolic geodesics. 
The analogy here is strongest with relatively hyperbolic spaces, with
the embedded components of the defining graph corresponding to the
peripheral regions of a relatively hyperbolic space. 
The peripheral regions are not necessarily
non-hyperbolic, so hyperbolic geodesics do not necessarily have to
avoid them completely. 
Rather, a geodesic is roughly as hyperbolic as its
intersections with peripheral regions.

Our original motivation for the present paper was to make this analogy
precise, and, in particular, to determine whether graphical small
cancellation groups contained strongly contracting elements, and
therefore fit into the scheme of groups with growth tight actions
introduced by Arzhantseva, Cashen, and Tao \cite{ArzCasTao15}, see \fullref{sec:scelements}.

In concrete terms, this analogy was first suggested by results of Gruber and Sisto, who
proved that the Cayley graph of a graphical small cancellation group
is weakly hyperbolic relative to the embedded components of the
defining graph \cite{GruSis14} in the sense of \cite{DGO11}, see also
\fullref{sec:conedoff}. 

Note, however, that in general the groups we consider in the present paper need not be non-trivially
relatively hyperbolic~\cite{BehDruMos09,GruSis14}.

\end{remark}

\subsection{Contraction and Morse quasi-geodesics}\label{sec:contraction}
We quote some of our technical results from \cite{ArzCasGrub}
that let us simplify and reformulate the contraction conditions.

\begin{lem}[{\cite[Lemma~6.3]{ArzCasGrub}}]\label{lem:strcontbddHdist}
Let $Y$ and $Y'$ be closed subspaces of $X$ at bounded Hausdorff distance from
one another. 
  Suppose $Y$ is $(r,\rho_2)$--contracting.
Then $Y'$ is $(r,\rho_2')$--contracting for some
$\rho_2'\asymp\rho_2$.
In particular, if $Y$ is strongly contracting then so is $Y'$.
\end{lem}

\begin{theorem}[{\cite[Theorem~1.4, Theorem~7.1]{ArzCasGrub}}]\label{thm:morseequalscontracting} 
Let $Y$ be a closed subspace of $X$.
Then the following are equivalent:
\begin{enumerate}
\item $Y$ is Morse.\label{item:Morse}
\item $Y$ is sublinearly contracting.\label{item:sublinearcontraction}
\item There exist $\rho_1$ and $\rho_2$ such that $Y$ is
  $(\rho_1,\rho_2)$--contracting.\label{item:contraction}
\item There exist a constant $C\geqslant 0$ and a sublinear function $\rho$ such that if $\gamma$ is a
  geodesic segment then
  $d(\gamma,Y)\geqslant C$ implies
$\diam\pi(\gamma)\leqslant\rho(\max_{z\in \gamma}d(z,Y))$.\label{item:git}
\end{enumerate}
Moreover, for each implication the function of the conclusion depends
only on the function of the hypothesis, not on $Y$.
\end{theorem}

We remark that in the case $\rho_1(r):=r$ the proof of
`(\ref{item:contraction}) implies (\ref{item:git})' yields $\rho\asymp\rho_2$.
When, in addition,
$\rho_2$ is bounded, which is the strongly contracting case, this recovers the well-known `Bounded Geodesic Image
Property', cf \cite{MasMin00, BesFuj09}.
\begin{corollary}\label{corollary:bgi}
  If $Y$ is strongly contracting then every geodesic segment that
  stays sufficiently far from $Y$ has uniformly bounded projection diameter.
\end{corollary}

\subsection{Proof of \fullref{theorem:strongcontraction}} 

\begin{lem}\label{lem:piece_projection}
 Let $\alpha$ be a geodesic in $X$, and let $x$ be a vertex not in $\alpha$. Let $\gamma$ be a path from $x$ to a vertex of $\alpha$ such that $|\gamma|=d(x,\alpha)$. Then, if $p$ is a path from $x$ to a vertex of $\alpha$ such that $p$ is a piece, then $p=\gamma$. In particular, the closest point projection of $p$ to $\alpha$ is $p\cap\alpha$.
\end{lem}
\begin{proof}
 If $p$ has the same terminal vertex as $\gamma$ and if $p\neq
 \gamma$, then there exist subpaths $\gamma'$ of $\gamma$ and $p'$
 of $p$, each of length at least 1, such that $c:=\gamma'p'^{-1}$ is a
 simple cycle. 
Since $\gamma'$ is a geodesic, we have $|\gamma'|\leqslant \ifrac{|c|}{2}$. Since $p'$ is a piece we have $|p'|<\ifrac{|c|}{6}$. This is a contradiction.
 
 If $p$ has a different terminal vertex, then there exist terminal
 subpaths $p'$ of $p$ and $\gamma'$ of $\gamma$, respectively, each
 having length at least 1, and a path $\alpha'$ contained in $\alpha$
 such that $c:=\gamma'\alpha'p'^{-1}$ is a simple cycle. 
By assumption on $\gamma$, we have $|\gamma'|\leqslant |p'|$ and $|\alpha'|\leqslant \ifrac{|c|}{2}$ , whence we conclude $|p'|\geqslant \ifrac{|c|}{4}$, contradicting the fact that $p'$ is a piece.  

The final claim follows from the fact that a subpath of a piece is a piece.
\end{proof}

\begin{lem}\label{lem:intersection_projection}
 Let $\Gamma_0$ be an embedded component of $\Gamma$ in $X$, and let $\alpha$ be a geodesic in $X$ such that $\Gamma_0$ intersects $\alpha$. 
Closest point projection of $\Gamma_0$ to $\alpha$ in $X$ agrees with closest point projection to $\Gamma_0\cap\alpha$ in $\Gamma_0$.
\end{lem}
\begin{proof}
Consider a vertex $x\in\Gamma_0\setminus\alpha$ such that there exists
a point $v\in\alpha$ with $d(x,\alpha)=d(x,v)$ and $v\notin
\Gamma_0$. Let $\gamma$ be a geodesic from $x$ to $v$, let $p$ be a
path in $\Gamma_0$ from $x$ to a vertex $v'$ of $\alpha$.
Let $D$ be a diagram over $\rels$ as in \fullref{lem:graphical_diagrams} filling the triangle
$\gamma[v,v']p^{-1}$, where $[v,v']$ is the reduced path in $\alpha$
from $v$ to $v'$.

Among all possible choices of $x$, $v$, $v'$, $\gamma$, $p$, and $D$
as above, make the choice for which $D$ has the minimal possible number of edges. Note that, by minimality, $D$ is a simple disc diagram. 
 
 Let $\Pi$ be a face of $D$ that intersects the side of $D$ corresponding to $p$ in an arc $a$. Then
 $a$ has a lift to $\Gamma$ via being a subpath of (a copy of) $p$, and one via
 being a subpath of $\partial \Pi$. If the lifts coincide (up to a
 label-preserving
 automorphism of $\Gamma$), then we can remove the edges of $a$ from $D$, thus obtaining a path $p'$ in $\Gamma_0$ as above, contradicting minimality. Hence, $a$ is a piece. Therefore, \fullref{lem:piece_projection} implies that the side of $D$ corresponding to $p$ is not contained in a single face of $D$.

We make a diagram $D'$ by 
attaching a new face $\Pi'$ to $D$ by identifying a proper subpath of
the boundary of $\Pi'$ with the side of $D$ corresponding to $p$. 
This operation is purely combinatorial, the boundary of $\Pi'$ is not labelled.
Note that if $\Pi$ is a face of $D$ with $e(\Pi)=1$ whose exterior arc
is contained in $p$ then, by the previous paragraph, that exterior arc
is a piece.
Since interior arcs of $D$ are pieces, $i(\Pi)\geqslant 6$ in $D$.
Thus, $\Pi$ becomes an interior face of $D'$ with $i(\Pi)\geqslant 7$.
It follows that $D'$  is a combinatorial geodesic bigon.

Apply \fullref{thm:strebel} to $D'$: it has at most two distinguished
faces, one is $\Pi'$ and the other, if it exists, is at the vertex
corresponding to $v$. 
Any other face of $D'$ came from $D$, and, in particular, the side of $D$ corresponding to $p$ is contained in a single face of $D$. 
This is a contradiction.
\end{proof}

\fullref{lem:intersection_projection} immediately implies the global-to-local direction of \fullref{theorem:strongcontraction}.

\begin{lem}\label{lem:2_pieces_projection}
 Let $\alpha$ be a geodesic in $X$. Let $p_1p_2$ be a simple path
 starting at a vertex $y$ in $\alpha$ and terminating at a vertex
 $x\in X$ such that $p_1p_2$ is contained in some embedded component
 $\Gamma_0$ and such that each $p_i$ is a piece that is not a single vertex. 
Then $d(x,\alpha)> \max\{|p_1|,|p_2|\}$. 
\end{lem}
\begin{proof}
 Let $q$ be a path starting at $x$ and terminating at a vertex $y'$ in $\alpha$ such that $|q|=d(x,\alpha)$. If $y'=y$, then the claim follows from the convexity of $\Gamma_0$, noting that any simple path that is a concatenation of two pieces must be a geodesic.
Similarly, if $q\cap p_1\neq\emptyset$ then \fullref{lem:piece_projection} implies $q=p_1p_2$.
Hence, assume that $q\cap p_1=\emptyset$.

Without loss of generality, we may assume $q\cap p_2=x$.
Consider a diagram $D$ over $\rels$ as in \fullref{lem:graphical_diagrams} filling the embedded geodesic quadrangle $p_1p_2q[y',y]$. 
We may stick onto the 2-complex $D$ two 2-cells $\Pi_1$ and $\Pi_2$ by identifying proper subpaths of their boundaries with $p_1$ and $p_2$, respectively. By construction, this yields a combinatorial geodesic triangle, and by \fullref{thm:strebel}, it has shape $\mathrm{III}_1$, where $\Pi_1$ and $\Pi_2$
are the distinguished faces intersecting each other only in a valence 4 vertex.
Using \fullref{lem:piece_projection}, there are no faces with exterior degree 2, so $D$ is a single face.
 
 This shows that there exists a simple cycle of the form $c:=p_1p_2q[y',y]$ in some embedded component $\Gamma_1$. 
Since $|p_1|,\,|p_2|<\ifrac{|c|}{6}$ and $|[y',y]|\leqslant \ifrac{|c|}{2}$, we have $|q|>\ifrac{|c|}{6}>\max\{|p_1|,|p_2|\}$.
\end{proof}

\begin{lem}\label{lem:corner_degree_2_projection} 
Let $\alpha$ be a geodesic in $X$.
Let $\Gamma_0$ be an embedded component of $\Gamma$ intersecting $\alpha$.
Suppose $\alpha\cap\Gamma_0$ is $(\rho_1,\rho_2)$--contracting in
$\Gamma_0$.
Let $c$ be a simple cycle in $\Gamma_0$ such that $c=p_1p_2qa$,
where each $p_i$ is a piece, $q$ realizes the distance of the terminal
vertex $x$ of $p_2$ to $\alpha$, and $a$ is a subpath of
$\alpha$. 
Then $|c|$ is bounded, with bound depending only on $\rho_1$ and $\rho_2$.
\end{lem}
\begin{proof}
Let $x:=p_2\cap q$, and let $y:=p_2\cap p_1$.
Since any subpath of a piece is a piece itself, we may assume without loss of generality that $x$ is the point of $p_2$ maximizing $\diam\pi(x)\cup\pi(y)$, where $\pi$ denotes closest point projection to $\alpha$ in $X$.
By \fullref{lem:piece_projection}, we have $|a|\leqslant \diam \pi(x)\cup\pi(y)$.

Since $p_1p_2$ is a path from $\alpha$ to $x$ and $q$ minimizes distance, $|q|\leqslant |p_1|+|p_2|<\ifrac{|c|}{3}$, which implies $|a|=|c|-|p_1|-|p_2|-|q|>\ifrac{|c|}{3}$.

Let $C\geqslant 0$ be as \fullref{thm:morseequalscontracting}~(\ref{item:git}). If $|p_1|<C$, then $|q|\leqslant |p_1|+|p_2|<C+\ifrac{|c|}{6}$, which implies $|c|<12C$.

Otherwise, \fullref{lem:2_pieces_projection} implies that for every point $z\in p_2$ we have $d(z,\alpha)\geqslant C$, so $p_2$ is a geodesic that stays outside the $C$--neighborhood of $\alpha$.
\fullref{thm:morseequalscontracting}~(\ref{item:git}) says $\diam\pi(p_2)$ is bounded by a sublinear function of $\max_{z\in p_2}d(z,\alpha)<\ifrac{|c|}{3}$.
Thus, $\ifrac{|c|}{3}<|a|\leqslant\diam\pi(p_2)$ is bounded above by a sublinear function of $|c|$, which implies $|c|$ is bounded.
\end{proof}

\begin{lemma}\label{lem:interior_degree_4_projection} 
Let $\alpha$ be a geodesic in $X$.
Let $\Gamma_0$ be an embedded component of $\Gamma$ intersecting $\alpha$.
Suppose $\alpha\cap\Gamma_0$ is $(\rho_1,\rho_2)$--contracting in
$\Gamma_0$.
Let $c$ be a simple cycle in $\Gamma_0$ of the form $c=p_1p_2p_3p_4a$, where each $p_i$ is a piece and $a$ is a subpath of $\alpha$. 
Then $|c|$ is bounded, with bound depending only on $\rho_1$ and $\rho_2$.
\end{lemma}
\begin{proof}
  Let $x:=p_2\cap p_3$. Let $y_1:=p_1\cap\alpha$. Let $y_2:=p_4\cap\alpha$. 
By \fullref{lem:intersection_projection}, there is a path $q\subset\Gamma_0$ such that $|q|=d(x,\alpha)$. Let $y':=q\cap\alpha$.

By symmetry, we may suppose $q\cap p_3=x$. Apply \fullref{lem:corner_degree_2_projection} to see that $|q|+|p_3|+|p_4|+d(y_2,y')$ is uniformly bounded.
If $q$ coincides with $p_1p_2$ we are done. 
Otherwise it must be that $q\cap p_1=\emptyset$, for otherwise we
would have a simple cycle composed of a geodesic and one or two
pieces, which is impossible.
Let $p_2':=\overline{p_2\setminus q}$, and let $q':=\overline{q\setminus p_2}$.
Apply \fullref{lem:corner_degree_2_projection} to see $|q'|+|p_2'|+|p_1|+d(y_1,y')$ is uniformly bounded. Thus, $|c|$ is uniformly bounded.
\end{proof}

\begin{lem}\label{lem:interior_degree_1_projection} 
Let $\alpha$ be a geodesic in $X$.
Let $\Gamma_0$ be an embedded component of $\Gamma$ intersecting $\alpha$.
Suppose $\alpha\cap\Gamma_0$ is $(\rho_1,\rho_2)$--contracting in
$\Gamma_0$.
Let $c$ be a simple cycle in $\Gamma_0$ such that $c=q_1pq_3a$,
where $p$ is a piece, $q_1$ and $q_3$ are geodesics realizing the
closest point projections of the endpoints of $p$ to $\alpha$, and $a$
is a subpath of $\alpha$. 
Then there is a sublinear function $\rho_2'$ depending only on
$\rho_1$ and $\rho_2$ such that
$\diam\pi(c)\leqslant \rho_2'(|c|)$.
If $\rho_1(r)\geqslant \ifrac{r}{2}$ then we can take  $\rho_2':=2\rho_2$. 
\end{lem}
\begin{proof}
Among all path subgraphs of $p$, consider those with maximal projection diameter.
Among those, choose one, $p'$, with minimal length, and let $x'$ and
$y'$ be its endpoints.
Let $x''\in\pi(x')$ and
$y''\in\pi(y')$ be vertices such that
$d(x'',y'')=\diam\pi(x')\cup\pi(y')$. 
By \fullref{lem:intersection_projection} there are geodesics
$q_1'\subset \Gamma_0$ connecting $x'$ to $x''$ and
$q'_3\subset\Gamma_0$ connecting $y'$ to $y''$.
Let $a'$ be the subpath of $\alpha$ from $x''$ to $y''$.
Let $c':=q'_1p'q'_3a'$.
The maximality hypothesis on $p'$ implies $\pi(p')\subset a'$ and
$\diam\pi(p')=\diam\pi(p)$.
It is immediate from the definitions that $\diam \pi(c)=\diam \pi(p)$, so it suffices to bound $d(x'',y'')$.

Let $x$ and $y$ be the endpoints of $p$, and note that
 $|q_1'|+|q_3'|\leqslant |q_1|+d(x,x')+d(y,y')+|q_3|\leqslant |c|$.
If there exists a vertex $z\in q_1'\cap q_3'$ then both $x''$ and $y''$
are in $\pi(z)$, so $d(x'',y'')\leqslant \rho_2(d(z,\alpha))\leqslant
\rho_2(|q'_1|)\leqslant \rho_2(|c|)$, and we are done. 
Otherwise, minimality of $|p'|$ implies $c'$ is a simple cycle.

Since $p'$ is a piece and $a'$ is geodesic, we have
$|q'_1|+|q'_3|>\ifrac{|c'|}{3}>2|p'|$, so there exists a point $z'\in
p'$ such that $d(x',z')\leqslant \ifrac{|q_1'|}{2}$ and
$d(z',y')\leqslant \ifrac{|q_3'|}{2}$. 

If $\rho_1(r)\geqslant \ifrac{r}{2}$, then take $\rho'_1:=\rho_1$ and $\rho_2'':=\rho_2$.
Otherwise, by
\fullref{thm:morseequalscontracting}~(\ref{item:sublinearcontraction})
there exists a sublinear function $\rho_2''$ such that
$\alpha\cap\Gamma_0$ is $(\rho_1',\rho_2'')$--contracting in $\Gamma_0$
for $\rho_1'(r):=r$. Let $\rho_2':=2\rho_2''$.
Since $\pi(p')\subset a'$, we conclude:
\begin{align*}
d(x'',y'')&\leqslant \diam\pi(x')\cup\pi(z')+\diam\pi(y')\cup\pi(z')\\
&\leqslant \rho_2''(d(x',\alpha))+\rho_2''(d(y',\alpha))\\
&=\rho_2''(|q_1'|)+\rho_2''(|q_3'|)\leqslant 2\rho_2''(|c|)=\rho_2'(|c|)\qedhere
\end{align*}
\end{proof}

\begin{lem}\label{closestpointprojection} 
Let $\alpha$ be a locally $(\rho_1,\rho_2)$--contracting geodesic in $X$.
Let $Y\subset X$ be either an embedded component of $\Gamma$, a piece,
or a single vertex.
Then there is a sublinear function $\rho_2'$ depending only on
$\rho_1$ and $\rho_2$ such that if $Y$ is disjoint from $\alpha$ then 
 $\diam\pi(Y)\leqslant\rho_2'(d(Y,\alpha))$.

If $\rho_1(r)\geqslant\ifrac{r}{2}$ we can take  $\rho_2'\asymp\rho_2$.
\end{lem}
\begin{proof}
Suppose $Y$ is disjoint from $\alpha$ and choose a vertex $y\in Y$
such that $d(y,\alpha)=d(Y,\alpha)$.
Let $y'$ be a point in $\pi(y)$.
It suffices to show that there exists a sublinear function $\rho_2''$
such that for every $x'\in\pi(Y)$ we have $d(x',y')\leqslant \rho_2''(d(Y,\alpha))$. 
Given such a $\rho_2''$,  set $\rho_2':=2\rho_2''$, and the lemma
follows from the triangle inequality. 

If $\pi(Y)=\{y'\}$ we are done. Otherwise,
let $x'$ be an arbitrary point in
$\pi(Y)\setminus\{y'\}$.
Let $\alpha'$ be the subpath of $\alpha$ from $x'$ to $y'$.

Choose $x\in Y$ such that $x'\in\pi(x)$.
Choose a path $p$ from $x$ to $y$ in $Y$, and geodesics $\beta_1$ and
$\beta_2$ connecting $x$ to $x'$ and $y$ to $y'$, respectively.
Choose a diagram $D$ over $\rels$ as in
\fullref{lem:graphical_diagrams} filling
$\alpha'\beta_2^{-1}p^{-1}\beta_1$.
Assume that we have chosen  $x$, $p$, $\beta_1$, $\beta_2$, and $D$ so
that $D$ has the minimal number of edges among all possible choices. 

In the case that $\beta_1$ and $\beta_2$ intersect, let $D'$ be the disc component
of $D$ intersecting the side corresponding to $\alpha$.
Then $D'$ is a
combinatorial geodesic triangle.
Apply \fullref{thm:strebel} to $D'$. 
Since $\beta_1$ and $\beta_2$ are geodesics realizing
closest point projection, the only  possibilities are that  $D'$  is:
\begin{enumerate}
 \item a single face,\label{item:singleface}
 \item shape $\mathrm{I}_2$, where $\alpha'$ is the side joining two vertices in the same distinguished face,\label{item:i2}
 \item shape $\mathrm{IV}$ with exactly three faces incident to $\alpha'$, two corners with interior degree 2 each and one ordinary face, or\label{item:iv}
 \item shape $\mathrm{V}$ with exactly two faces incident to $\alpha'$, the two corners with interior degree 2 each.\label{item:v}
\end{enumerate}
In all cases, $d(x',y')$ is bounded by a sublinear function $\rho_2''$
of $d(Y,\alpha)$ that
depends only on $\rho_1$ and $\rho_2$: 
For case (\ref{item:singleface}) this follows from  the fact that
$\alpha$ is uniformly  locally contracting.
For case (\ref{item:iv}) this follows from \fullref{lem:corner_degree_2_projection} and
\fullref{lem:interior_degree_4_projection}, 
and for case (\ref{item:v}) this follows from \fullref{lem:corner_degree_2_projection}. 
In these two cases, the bounds are in fact constants depending only on
$\rho_1$ and $\rho_2$. 
Now consider case (\ref{item:i2}). 
Let $\Pi$ be the face of $D'$ containing $\alpha'$.
Let $c$ be the embedded  quadrangle in $X$ whose sides are $\alpha'$, a subpath $q_1$ of $\beta_1$, a piece $p'$, and a subpath $q_2$ of $\beta_2$.
Apply \fullref{lem:interior_degree_1_projection} to $c$, and 
observe that $||q_1|-|q_2||\leqslant
|p'|<\ifrac{|c|}{6}$, and $|q_1|+|q_2|> \ifrac{|c|}{3}$, whence
$d(Y,\alpha)=|\beta_2|\geqslant |q_2|>\ifrac{|c|}{6}$.

Now suppose that $\beta_1$ and $\beta_2$ do not intersect.
In this case minimality of $D$ and the fact that $y$ minimizes the
distance from $Y$  to $\alpha$ imply that $D$ is simple.
If $Y$ is an embedded component of
$\Gamma$, it follows as in the proof of
\fullref{lem:intersection_projection} that any arc in the side of $D$
corresponding to $p$ is a piece.
The same is true if $Y$ is a piece since
subpaths of pieces are pieces. 
Thus, we can stick  a new face onto $p$
to obtain a combinatorial geodesic triangle $D'$, and we make the same
argument as above, noting this time that case (\ref{item:singleface}) cannot hold, for it
would imply that $Y$ intersects $\alpha$.
\end{proof}

\begin{proof}[Proof of \fullref{theorem:strongcontraction}] 
Recall that the global-to-local direction of \fullref{theorem:strongcontraction} follows from \fullref{lem:intersection_projection}.

Suppose that $\alpha$ is locally $(\rho_1,\rho_2)$--contracting.

Let $x$ and $y$ be points of $X$ such that $d(x,y)\leqslant \rho_1(d(x,\alpha))$. 
Let $\gamma$ be a geodesic from $x$ to $y$.
Let $x'\in\pi(x)$ and $y'\in\pi(y)$ be points realizing $\diam\pi(x)\cup\pi(y)$.
Let $\delta_1$ be a geodesic from $x'$ to $x$, and let $\delta_2$ be a geodesic from $y'$ to $y$.
Let $\alpha'$ be the path subgraph of $\alpha$ from $x'$ to $y'$.

First, assume that $\alpha'$ does not enter the $C$-neighborhood of $\gamma$, where $C$ is the constant from \fullref{thm:morseequalscontracting}~(\ref{item:git}) associated to $(\rho_1,\rho_2)$.

If $\delta_1$ and $\delta_2$ intersect, \fullref{closestpointprojection} yields the claim, whence we will assume that they do not. 
Moreover, by removing initial and terminal subpaths of $\gamma$ that
do not increase the size of the closest point projection to $\alpha$,
we may assume that $\gamma$, $\delta_1$, $\delta_2$ and $\alpha'$ can
be concatenated to a simple closed path $c_0$. Let $D$ be a diagram as
in \fullref{lem:graphical_diagrams} for the label of $c_0$, and, by
identifying $\partial D$ with $c_0$, we can consider $\gamma$,
$\delta_1$, $\delta_2$ and $\alpha'$ as subpaths of $\partial D$. 
As the interior arcs of $D$ are pieces and its four sides are geodesics, we can apply \fullref{prop:combinatorialngons}.

The first possibility is that there is a face $\Pi$ with $e(\Pi)=2$ and $i(\Pi)=2$.
Its boundary is a cycle $c=p_1qp_3a$ in some embedded component such that the $p_i$ are pieces,
$a$ is a path subgraph of $\alpha$, and $q$ is a path subgraph of $\gamma$.
We have $\ifrac{|c|}{6}<|a|\leqslant \ifrac{|c|}{2}$. Therefore, $\max_{z\in c}d(z,\alpha)\leqslant \ifrac{5|c|}{12}$. 

We now choose $R_2\geqslant0$ so that for all $r\geqslant R_2$ we have $2\rho_2(r)<\rho_1(r)$. Suppose $\gamma$ does not enter the $R_2$--neighborhood of $\alpha$. 

By \fullref{thm:morseequalscontracting}~(\ref{item:git}) we have that $\ifrac{|c|}{6}<|a|$ is bounded by a sublinear function of $\max_{z\in c}d(z,\alpha)\leqslant \ifrac{5|c|}{12}$.
This implies $|c|$ is uniformly bounded, as are $|p_1|,\,|p_3|<\ifrac{|c|}{6}$.
Therefore, $\gamma$ enters a uniformly bounded neighborhood of $\alpha$.

The second possibility is that $D$ contains a zipper with two teeth on $\alpha$.
By \fullref{lem:interior_degree_4_projection}, the boundary lengths of the two teeth on $\alpha$ are uniformly bounded.
It follows that the boundary length of the upper tooth is also uniformly bounded, since its intersection with the bottom teeth accounts for more than $\ifrac{1}{6}$--th of its length.
 Therefore, $\gamma$ enters a uniformly bounded neighborhood of $\alpha$.

Let $C'$ be larger than $C$, $R_2$, and the bounds from the first two
cases.
Let $z$ be the first point of $\gamma$ such that $d(z,\alpha)\leqslant
C'$, if such a point exists. 
Otherwise let $z:=y$.
The projection diameter of $\gamma$ is at most the projection diameter
of the path subgraph of $\gamma$ from $x$ to $z$ plus the projection
diameter of the path subgraph from $z$ to $y$. 
The latter is at most $4C'$, since every point of this path is
within $2C'$ of $\alpha$, so it suffices to bound the former. 
Thus, we may assume that the geodesic from $x$ to $y$ does not enter
the $C'$ neighborhood of $\alpha$.

By our choice of $C'$, we conclude that $D$ must fall into the third case of \fullref{prop:combinatorialngons}. 
Therefore, there exist $k\leqslant 6$ and a path graph $p:=p_0p_1p_2\dots p_{k}p_{k+1}$ (recall \fullref{corollary:faceseqsegment}) from $\delta_1$ to $\delta_2$ such that:
\begin{itemize}
 \item If $1\leqslant i\leqslant k$, then $p_i$ is a path subgraph of  $\gamma\cap\Pi_i$, and
 \item $p_0$ is empty or a piece in $\Pi_1$, and $p_{k+1}$ is empty or a piece in $\Pi_{k}$.
\end{itemize}
The second claim follows since, in $D$, the corner that is separated
from the rest of $D$ by removing $\Pi_1$ is either empty, a face, or
has shape $\mathrm{I}_1$. 
The same observation holds for the corner at $\Pi_k$.

Notice that every point of $\gamma$ is within $d(x,\alpha)+\rho_1(d(x,\alpha))\leqslant2d(x,\alpha)$ of $\alpha$.

For $1\leqslant i\leqslant k$, the path graph $p_i$ is a geodesic subsegment of $\gamma$ that is outside the $R_2$--neighborhood of $\alpha$.
If $p_i$ is contained in an embedded component $\Gamma_0$ of $\Gamma$ disjoint from $\alpha$, then $\diam\pi(p_i)$ is bounded by a sublinear function of $d(\Gamma_0,\alpha)\leqslant d(p_i,\alpha)< 2d(x,\alpha)$, by \fullref{closestpointprojection}.
If $p_i$ is not contained in an embedded component disjoint from $\alpha$, then \fullref{thm:morseequalscontracting}~(\ref{item:git}) says $\diam\pi(p_i)$ is bounded by a sublinear function of $\max_{z\in p_i}d(z,\alpha)\leqslant 2d(x,\alpha)$.
When $i\in\{0,k+1\}$, the diameter of $\pi(p_i)$ is bounded by a sublinear function of $d(p_i,\alpha)$ by  \fullref{closestpointprojection}.

We have that
$\diam\pi(x)\cup\pi(y)\leqslant\sum_{i=0}^{k+1}\diam\pi(p_i)$, and each of the at most 8 terms is bounded by a sublinear function of $d(x,\alpha)$, so $\diam\pi(x)\cup\pi(y)$ is bounded by a sublinear function $\rho_2'$ of $d(x,\alpha)$.
Thus, $\alpha$ is $(r,\rho_2')$--contracting. 
Moreover, each of the constituent sublinear functions of $\rho_2'$ is
determined by $\rho_1$ and $\rho_2$ and, when $\rho_1(r)\geqslant \ifrac{r}{2}$, is either bounded or
asymptotic to $\rho_2$, so $\rho_2'$ depends only on $\rho_1$ and
$\rho_2$, and we can take $\rho_2'\asymp\rho_2$ when $\rho_1(r)\geqslant\ifrac{r}{2}$.
\end{proof}

\subsection{First applications of~\fullref{theorem:strongcontraction}}

\fullref{theorem:strongcontraction} and \fullref{thm:morseequalscontracting} show:
\begin{theorem}\label{corollary:locallysublinearimpliesmorse}
  A geodesic in $X$ is Morse if and only if it is uniformly
  locally contracting.
\end{theorem}

In the classical small cancellation case we have more explicit
criteria:
\begin{corollary}\label{corollary:classicalmorse}
Let $\Gamma$ be a $Gr'(\ifrac{1}{6})$--labelled graph whose components are cycle graphs.
Let $\alpha$ be a geodesic in $X$.
Define $\rho(r):=\max_{|\Gamma_i|\leqslant r}|\Gamma_i\cap\alpha|$, where the $\Gamma_i$ range over embedded components of $\Gamma$.
Then $\alpha$ is Morse if and only if $\rho$ is sublinear, and $\alpha$ is strongly contracting if and only if $\rho$ is bounded.
\end{corollary}

In hyperbolic spaces and in CAT(0) spaces Morse geodesics are known to
be strongly
contracting. 
In graphical small cancellation groups we build the first examples with a wide
range of degrees of contraction:
\begin{theorem}\label{thm:allcontractionrates}
  Let $\rho$ be a sublinear
  function.
There exists a group $G$ with finite generating set $\gens$ and a function
$\rho'\asymp\rho$ such that
there exists an $(r,\rho')$--contracting geodesic $\alpha$ in the Cayley graph $X$ of $G$ with
respect to $\gens$.

Furthermore, $\rho'$ is optimal, in the following sense: If $\alpha$ is 
$(r,\rho'')$--contracting for some $\rho''$ then 
$\limsup_{r\to\infty}\frac{\rho''(2r)}{\rho(r)}\geqslant 1$.
\end{theorem}

If $\lambda>0$, a \emph{$C'(\lambda)$-collection of words} $W$ is a subset of $\langle \gens\rangle$ such that the disjoint union of cycle graphs labelled by the elements of $W$ satisfies the graphical $C'(\lambda)$-condition.

\begin{proof}
We may assume $\rho$ is unbounded and integer valued. 
Since $\rho$ is sublinear, there exists an $R$ such that
$5\rho(r)\leqslant 2r$ for all
$r\geqslant R$.
Let $\gens:=\{a,b,c\}$.
Let $I\subset\{z\in\mathbb{Z}\mid z\geqslant R\}$ be an infinite set such that there
exists a $C'(\ifrac{1}{12})$-collection $\{w_i\}_{i\in I}$ of words
  $w_i\in\langle b,c\rangle$ with $|w_i|=4i$.
For $i\in I$, define $R_i:=a^{\rho(i)}w_i$.

Let $\Gamma:=(\Gamma_i)_{i\in I}$ be a disjoint union of
  $\gens$--labelled cycle graphs, with $\Gamma_i$ labelled by
  $R_i$.
Let $G:=G(\Gamma)$ and $X:=\Cay(G,\gens)$.
There are no non-trivial label-preserving automorphisms of any component
$\Gamma_i$ because of the unique $a$--labelled path subgraph.
There are no non-trivial label-preserving automorphisms of $\Gamma$
that exchange components since $\{w_i\}_{i\in I}$ have distinct lengths.

If $p$ is a piece contained in $\Gamma_i$ and labelled by $l$ then $l$
can be written $l=l'+l''+l'''$ where $l'$ is a suffix of $w_i$, $l''$
is a subword of $a^{\rho(i)}$, and $l'''$ is a prefix of $w_i$.
If $l'$ or $l'''$ is non-empty then $l'l'''$ is a piece for
$\{w_i\}_{i\in I}$.
Therefore $|p|<\rho(i)+\ifrac{|w_i|}{12}$.
Since $5\rho(i)\leqslant 2i$ this implies $|p|<\ifrac{|R_i|}{6}$, so
$\Gamma$ is $C'(\ifrac{1}{6})$--labelled.

Let $\alpha$ be the geodesic with all edge labels $a$.
By construction, $\alpha$ is locally $(r,\rho)$--contracting, so
\fullref{theorem:strongcontraction} says there exists
$\rho'\asymp\rho$ such that $\alpha$ is $(r,\rho')$--contracting.

Conversely, if $\alpha$ is $(r,\rho'')$--contracting then it is
locally $(r,\rho'')$--contracting.
By construction, for $i\in I$ there exists a point $x_i$ such that
$d(x_i,\alpha)=2i$ and $\diam \pi(x_i)=\rho(i)$, so we must have
$\rho(i)\leqslant \rho''(2i)$.
\end{proof}

\fullref{thm:weakimpliesstrong} shows that in classical small cancellation groups the geometry of cycle graphs
dictates that only the output contraction function $\rho_2$ plays a role.
In \fullref{thm:nonstability} we construct a graphical example for which the input
contraction function $\rho_1$ also
carries non-trivial information.
\begin{theorem}\label{thm:weakimpliesstrong}
  Let $\Gamma$ be a $Gr'(\ifrac{1}{6})$--labelled graph whose components are
  cycle graphs.
A geodesic $\alpha$ in $X$ that is $(\rho_1,\rho_2)$--contracting is
$(r,\rho'_2)$--contracting for $\rho_2'\asymp\rho_2$.
\end{theorem}
\begin{proof}
Let $\Gamma_0$ be an embedded component of $\Gamma$.
Since $\Gamma_0$ is a cycle graph and $\Gamma_0\cap\alpha$ is
connected, there is a unique point $y$ for which the closest point
projection $\pi(y)$ in $\Gamma_0$ has positive diameter, and
$\diam\pi(y)=\diam\pi(\Gamma_0\cap\alpha)\leqslant \rho_2(d(y,\alpha))$.

Let $x$ be a point of $\Gamma_0$ equidistant from $y$ and $\alpha$.
Then $d(x,y)=d(x,\alpha)$ and:
\[\diam\pi(x)\cup\pi(y)=\diam\pi(y)\leqslant \rho_2(d(y,\alpha))=\rho_2(2d(x,\alpha))=\rho_2'(d(x,\alpha))\]
This is the worst case, since for $x'$ closer to $\alpha$ the ball of
radius $d(x',\alpha)$ about $x'$ does not include $y$.
We conclude that $\alpha$ is locally $(r,\rho_2(2r))$--contracting.
Now apply \fullref{theorem:strongcontraction} to see that $\alpha$ is
$(r,\rho_2')$--contracting for some $\rho_2'(r)\asymp\rho_2(2r)\asymp\rho_2(r)$.
\end{proof}
\begin{corollary}
Let $\Gamma$ be a $Gr'(\ifrac{1}{6})$--labelled graph whose components are
  cycle graphs.
A geodesic $\alpha$ in $X$ that is $(\rho_1,\rho_2)$--contracting with
$\rho_2$ bounded is strongly contracting.
\end{corollary}

Having uniformly bounded intersection with every embedded component of
$\Gamma$ is a sufficient condition for a geodesic $\alpha$ in the
Cayley graph of a $Gr'(\ifrac{1}{6})$--group
to be strongly contracting.
It is not a necessary condition.
Consider the following example:
\begin{example}\label{rem:unboundedproj}
We start with a simple example of an infinite classical $C'(\ifrac{1}{6})$--small cancellation presentation: $\fpres{a,b}{R_n:n\in\N}$ where, for each $n\geqslant 0$ we define:  \[R_n=ab^{20n+1}ab^{20n+2}\cdots ab^{20n+19}ab^{-(20n+20)}\]

The graphs $\Gamma_i$ for $i\geqslant 1$ are defined by taking the disjoint union of oriented cycles $R_n$ 
for $(i-1)(i)\leqslant 2n\leqslant -2+i(i+1)$ and identifying the unique subpath with label $b^{20k}$ in $R_{k-1}$ to the unique subpath in $R_k$ with label $b^{20k}$ which is preceded by $a$ and succeeded by $ba$.
This means we identify the paths labelled by the bold words $ab^{20(k-1)+1}ab^{20(k-1)+2}\dots a\mathbf{b^{-20k}}$ and $a\mathbf{b^{20k}}ba\dots ab^{-(20k+20)}$.

The $\{a,b\}$--labelled graph $\sqcup_{i\in\N}\Gamma_i$ satisfies the $Gr'(\ifrac{1}{6})$--condition and gives rise to the same Cayley graph as $\langle a,b \,|\, R_n:n\in\N\rangle$. 
Consider a path $\alpha$ in the Cayley graph labelled by the powers of $a$. For every $i$, there exists an embedded copy $\Gamma_i'$ of $\Gamma_i$ with $|\alpha\cap\Gamma_i'|=i$. Since paths in $\Gamma_i$ labelled by powers of $a$ are geodesic, this implies that $\alpha$ is a geodesic. It has intersection of length at most 1 with relators $R_n$, so it is strongly contracting, but we have $|\alpha\cap\Gamma_i'|=i$.
\end{example}

We also provide the first examples spaces $X$ and $\tilde X$ and geodesics $\gamma$ and $\tilde \gamma$ such that there exists a quasi-isometry $X\to \tilde X$ mapping $\gamma$ to $\tilde \gamma$ and such that $\gamma$ is not strongly contracting, but $\tilde\gamma$ is strongly contracting.

\begin{thm}[Non-stability of strong contraction]\label{thm:nonstability}
There exists a group $G$ with finite generating sets $\gens\subset
\tilde\gens$ and an infinite geodesic $\gamma$ in $X:=\Cay(G,\gens)$ labelled by the powers of a generator such that $\gamma$ is not strongly contracting, but its image $\tilde\gamma$ in $\tilde X:=\Cay(G,\tilde\gens)$ obtained from the inclusion $\gens\subset\tilde\gens$ is an infinite strongly contracting geodesic.
\end{thm}

The idea is to turn \cite[Example~3.2]{ArzCasGrub} into a $Gr'(\ifrac{1}{6})$--labelled graph $\Gamma$. 
By construction, $\Gamma$ will contain a non-strongly contracting geodesic $\gamma$ that will be labelled by the powers of a generator $y$. 
\fullref{theorem:strongcontraction} then ensures that the image of $\gamma$ in the Cayley graph is not strongly contracting. 
By adding additional edges, corresponding to new generators, to $\Gamma$ and cutting the resulting graph apart into cycle graphs, we obtain a classical $C'(\ifrac{1}{6})$--presentation of the same group in which no relator contains more than one occurrence of the letter $y$. 
Thus, a geodesic labelled by the powers of $y$ will be strongly contracting in the Cayley graph with respect to the new generating set.

\begin{proof}[Proof of \fullref{thm:nonstability}] Assume the sets $\{y\}$, $\{a,b\}$, $\{x_1,x_2\}$, $\gens_1$, $\gens_2$ are pairwise disjoint sets, and $|\gens_1|\geqslant 2$ and $|\gens_2|\geqslant 2$. A \emph{classical piece} with respect to a set of words is the label of a piece in the disjoint union of cycle graphs labelled by the words. It is an exercise to explicitly construct words with the following properties:

Let $\omega:=\{w_1,w_2\}\subset \langle\gens_1\rangle$ be a $C'(\ifrac{1}{6})$-collection of words (as defined for the proof of \fullref{thm:allcontractionrates}) such that $|w_1|=|w_2|\geqslant 24$. 
Moreover, assume that the words $w_1w_2$, $w_1w_2^{-1}$, and $w_1^{-1}w_2$ are freely reduced.

Let $\mu(a,b):=\{\mu_i(a,b)\mid i\in\N\}\subset \langle\{a,b\}\rangle$ be a $C'(\ifrac{1}{6})$-collection of words such that there exists $C\in\N$, $C\geqslant 6$ with $|\mu_{i}(a,b)|=(\ifrac{C}{|w_1|})\cdot 2^i$. 
Note that by our assumptions on $\omega$, the set $\mu(w_1,w_2)$ also satisfies the classical $C'(\ifrac{1}{6})$--condition: 
Since no two $w_1^{\pm1}$ and $w_2^{\pm2}$ start with the same letter, any piece with respect to $\mu(w_1,w_2)$ comes from a piece with respect to $\mu(a,b)$.

Let $\nu:=\{\nu_i:i\in \N\}\subset \langle\gens_2\rangle$ be a $C'(\ifrac{1}{12})$-collection of words such that $|\nu_i|=C\cdot 2^i+1$.

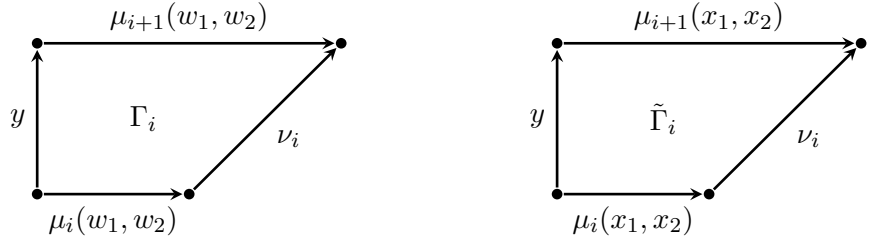
\begin{figure}
\begin{center}
\begin{tikzpicture}[line width=1pt, >=stealth,shorten >=2.5pt, shorten <=2.5pt,x=2cm,y=2cm]
\draw[->] (0,0) to node[above] {\small $\mu_{i+1}(w_1,w_2)$} (2,0);
\draw[->] (0,-1) to node[below] {\small $\mu_i(w_1,w_2)$} (1,-1);
\draw[->] (0,-1) to node[left]{\small $y$} (0,0);
\draw[->] (1,-1) to node[below right]{\small $\nu_i$} (2,0);
\fill (0,0) circle (2pt);
\fill (2,0) circle (2pt);
\fill (0,-1) circle (2pt);
\fill (1,-1) circle (2pt);

\node at (.7,-.5) {\small $\Gamma_i$};
\end{tikzpicture}
\hspace{2cm}
\begin{tikzpicture}[line width=1pt, >=stealth,shorten >=2.5pt, shorten <=2.5pt,x=2cm,y=2cm]
\draw[->] (0,0) to node[above] {\small $\mu_{i+1}(x_1,x_2)$} (2,0);
\draw[->] (0,-1) to node[below] {\small $\mu_i(x_1,x_2)$} (1,-1);
\draw[->] (0,-1) to node[left]{\small $y$} (0,0);
\draw[->] (1,-1) to node[below right]{\small $\nu_i$} (2,0);
\fill (0,0) circle (2pt);
\fill (2,0) circle (2pt);
\fill (0,-1) circle (2pt);
\fill (1,-1) circle (2pt);
\node at (.7,-.5) {\small $\tilde\Gamma_i$};
\end{tikzpicture}
\end{center}
\caption{The graphs $\Gamma_i$ and $\tilde\Gamma_i$.}\label{figure:nonqiexample}
\end{figure}
The graph $\Gamma$ labelled over $\gens:=\{y\}\cup \gens_1\cup \gens_2$ is obtained by taking the disjoint union of the $\Gamma_i$ as in \fullref{figure:nonqiexample} and, for each $i$, identifying the `top' of $\Gamma_i$ with the `bottom' of $\Gamma_{i+1}$. 
These both have the same label. 
Note that any simple closed path $\gamma$ in $\Gamma$ is a path going around a finite union $\Gamma_i\cup\Gamma_{i+1}\cup\dots\cup\Gamma_{j}$ for $i\leqslant j$, i.e., the label of $\gamma$ is, up to inversion and cyclic shift, of the form: \[\mu_i(w_1,w_2)\nu_i\nu_{i+1}\dots\nu_{j}\mu_{j+1}(w_1,w_2)^{-1}y^{i-j-1}\]

A piece that is a simple subpath of a cyclic shift of $\gamma$ has a label that is a subword of one of the following words. 

\begin{itemize}
\item A piece in some $\mu_k(w_1,w_2)$ or a product of two pieces in some $\nu_k$ or $\nu_k\nu_{k+1}$ (since, by gluing together words $\nu_k$ when constructing $\Gamma$, new pieces may have arisen, each labelled by a product of two pieces in $\nu$), or
\item $y^{i-j-1}$, or
\item a product of two words from the first bullet, or a product $py^{i-j-1}q$, where $p$ and $q$ are pieces in $\mu_{j+1}(w_1,w_2)$ and $\mu_{i}(w_1,w_2)$ respectively.
\end{itemize}
Here, a piece in $\mu_i(w_1,w_2)$ means a classical piece with respect to the collection of words $\mu(w_1,w_2)$, a piece in $\nu_i$ means a piece with respect to $\nu$, and so on. 
A piece may also be the empty word.

Note that any piece $p$ in the first bullet has length less than $\ifrac{1}{6}$ of its ambient word $\mu_k(x_1,x_2)$ or $\nu_k$ or $\nu_k\nu_{k+1}$, and that $6|i-j-1|\leqslant C2^{j}<|\nu_j|$. Therefore, $\Gamma$ satisfies the $Gr'(\ifrac{1}{6})$-condition.

Let $\alpha$ be the geodesic ray in $\Gamma$ labelled by positive powers of $y$.
Let $\beta$ be the geodesic ray in $\Gamma$ labelled by $\nu_1\nu_2\cdots$.
The ray $\beta$ leaves every bounded neighborhood of $\alpha$, but has unbounded image under the closest point projection to $\alpha$ in $\Gamma$.
By \fullref{corollary:bgi}, $\alpha$ is not strongly contracting in $\Gamma$.
Thus, by \fullref{theorem:strongcontraction}, the geodesic ray labelled by positive powers of $y$ is not strongly contracting in $\Cay(G(\Gamma),\gens)$.

\medskip

We now define a graph $\tilde \Gamma$ labelled over $\tilde\gens:=\{y\}\cup \{x_1,x_2\} \cup \gens_1 \cup \gens_2$ as follows: 
Let $c_1$ be a cycle graph labelled by $x_1w_1^{-1}$ and $c_2$ a cycle graph labelled by $x_2w_2^{-1}$. 
Set: 
\[\tilde\Gamma:=c_1 \sqcup c_2\sqcup\bigsqcup_{i\in\N}\tilde\Gamma_i\]

where the $\tilde \Gamma_i$ are given in \fullref{figure:nonqiexample}. Note that while, in this new graph, the paths labelled by $\mu_i(x_1,x_2)$ are pieces, no path labelled by $\mu_{i+1}^{-1}(x_1,x_2)y^{-1}$ or $y^{-1}\mu_i(x_1,x_2)$ is a piece.

By construction, $|\mu_i(x_1,x_2)|,|\mu_{i+1}(x_1,x_2)|<\ifrac{|\nu_i|}{12}$, and any piece in $\nu_i$ has length less than $\ifrac{|\nu_i|}{12}$. 
This, together with observations as above and the fact that any simple piece in $c_1$ or $c_2$ has length less than $\ifrac{|c_i|}{6}$ by construction, shows that $\tilde\Gamma$ satisfies the $Gr'(\ifrac{1}{6})$--condition. 
Since its components are cycle graphs, the geodesic ray labelled by the positive powers of $y$ is strongly contracting in $\Cay(G(\tilde\Gamma),\tilde \gens)$, by \fullref{theorem:strongcontraction}. 

The presentation for $G(\tilde\Gamma)$ coming from $\tilde\Gamma$ includes the generators and relations of the presentation of $G(\Gamma)$ coming from $G(\Gamma)$, as well as generators $x_1$ and $x_2$ and relations $x_1w_1^{-1}$ and $x_2w_2^{-1}$.
Rewriting the presentation by Tietze transformations, we see $G(\Gamma)\cong G(\tilde\Gamma)$.
In particular, the inclusion $\Cay(G(\Gamma),\gens)\hookrightarrow\Cay(G(\tilde\Gamma),\tilde \gens)$ is a quasi-isometry.
\end{proof}

We record the following consequence of \fullref{theorem:strongcontraction} and \fullref{lem:strcontbddHdist} for reference:

\begin{cor}
 Let $\alpha$ and $\alpha'$ be infinite geodesic rays in $\Cay(G(\Gamma),\gens)$, where $\Gamma$ is a $Gr'(\ifrac{1}{6})$--labelled graph labelled by $\gens$, such that $d_{\mathrm{Hausdorff}}(\alpha,\alpha')<\infty$. Then $\alpha$ is uniformly locally strongly contracting if and only if $\alpha'$ is.
\end{cor}


\section{Strongly contracting elements}\label{sec:scelements}
In this section, we show the existence of strongly contracting elements in graphical small cancellation groups.

\begin{thm}\label{thm:existenceofstronglycontractingelement}
Let $\Gamma$ be a $Gr'(\ifrac{1}{6})$--labelled graph whose components
are finite, labelled by a finite set $\gens$. Assume that $G(\Gamma)$ is infinite. Then there exists an infinite order element $g\in G(\Gamma)$ such that $\fgen{g}$ is strongly contracting in  $\Cay(G(\Gamma),\gens)$. 
\end{thm}

The element $g$ is the WPD element for the action of $G(\Gamma)$ on the hyperbolic coned-off space of Gruber and Sisto \cite{GruSis14} (see also Section~\ref{sec:conedoff}).

A recent theorem of Arzhantseva, Cashen, and Tao \cite{ArzCasTao15}
says that if $G$ is a group acting cocompactly on a proper metric space $X$, and if $g\in G$ is an infinite order element such that closest point projection to an orbit of $\langle g\rangle$ is strongly contracting, then the action of $G$ on $X$ is \emph{growth tight}. 
This means that the rate of exponential growth of $G$ with respect to the pseudo-metric induced by the metric of $X$ is strictly greater than the growth rate of a quotient of $G$ by any infinite normal subgroup, with respect to the induced pseudo-metric on the quotient group. Thus, a corollary of \fullref{thm:existenceofstronglycontractingelement}, is:
\begin{thm}\label{thm:growthtight}
Let $\Gamma$ be a $Gr'(\ifrac{1}{6})$--labelled graph whose components
are finite, labelled by a finite set $\gens$. 
Then the action of $G(\Gamma)$ on $\Cay(G(\Gamma),\gens)$ is growth tight.
\end{thm}

\subsection{Infinite cyclic subgroups are close to periodic geodesics}
In the previous section we deduced contraction results for geodesics
in a group defined by a $Gr'(\ifrac{1}{6})$-labelled graph. In order to show that cyclic subgroups are strongly contracting, we show that they are actually close to bi-infinite geodesics. As a by-product, we also obtain a result about translation lengths in graphical small cancellation groups. 

We glean from \cite{Gru15SQ} the following:

\begin{lemma}\label{lem:periodicgeodesic} Let $\Gamma$ be a
  $Gr'(\ifrac{1}{6})$-labelled graph whose components
are finite, labelled by a finite set $\gens$.
Every infinite cyclic subgroup of $G:=G(\Gamma)$ is at bounded Hausdorff distance from a periodic bi-infinite geodesic in $\Cay(G,\gens)$.
\end{lemma}
Recall that a bi-infinite path graph $\alpha$ in $\Cay(G,\gens)$ is \emph{periodic} if there is a cyclic subgroup of 
$G$ that stabilizes $\alpha$ and acts cocompactly on it. In the proof we have surjections $\fgen{\gens}\onto H\onto G$.
Let $|\cdot|_H$ denote the word length in $H$ with respect to the 
image of $\gens$. 
Similarly, let $|\cdot|_G$ denote the word length in $G$ with respect to the 
image of $\gens$. 

Recall that the translation length of an element $g\in G$ is defined by:
\[\tau_G(g):=\lim_{n\to\infty}\frac{|g^n|_G}{n}\]
This limit always exists since the map $n\mapsto |g^n|_G$ is subadditive.

\begin{proof}[Proof of \fullref{lem:periodicgeodesic}]
Let $x$ be an infinite order element of $G$.

\medskip

\noindent {\bf Claim.} There exist $w\in\fgen{\gens}$ and $N\in\N$ such
that $w$ represents an element $g\in G$ such that $x$
is conjugate to $g^N$ and the
following property is satisfied: 
There exist a hyperbolic group $H$, also a quotient of $\fgen{\gens}$, and
an epimorphism $\phi\from H\onto G$ induced by the identity on $\gens$ such
that, if $h$ denotes the element of $H$ represented by $w$, then $\phi$
restricts to an isometry $\langle h\rangle\to\langle g\rangle$ with
respect to the subspace metrics in $H$ and $G$, respectively. 

\medskip

We show how to deduce the statement of the lemma from the claim: 
By a theorem of Swenson \cite[Theorem 8]{Swe95}, since $H$ is
hyperbolic, there exist $h_0\in H$ and $M\in\N$ such that $h^M$ is
conjugate to $h_0$ and such that $|h_0^n|_H=n|h_0|_H$ for all $n>0$. 
Therefore, $\tau_H(h^M)=|h_0|_H$, and the bi-infinite path graph in $H$ labelled by the powers of a shortest element $w_0\in \fgen{\gens}$ representing $h_0$ is geodesic. 
Consider $g_0:=\phi(h_0)$, the element of $G$ represented by $w_0$. 
Then $g^M$ is conjugate to $g_0$ by construction, and we have
$\tau_G(g_0)=\tau_G(g^M)=\tau_H(h^M)=|w_0|$ since $\langle h\rangle\to\langle
g\rangle$ is an isometry. 
This implies $|g^{nM}|_G=n|w_0|$ for every $n$, i.e., the
bi-infinite path graph in $\Cay(G,\gens)$ starting at $1$ and labelled by
the powers of $w_0$ is geodesic. 
This proves the lemma, assuming our claim.

\medskip

It remains to show the claim: By \cite[Section 4]{Gru15SQ}, there
exists $g\in G$ such that $x$ is  conjugate to $g^N$ for some $N$ and
such that, if $w$ is a shortest word representing $g$, we have the
following possibilities: 
Any bi-infinite periodic path graph $\gamma$ labelled by the powers of $w$ is
a convex geodesic (Case 1 in \cite[Section 4]{Gru15SQ}), or there
exists some $C_0$ such that for any copy $c$ in $\Cay(G,\gens)$ of a simple cycle in $\Gamma$, the length of the intersection of $c$ with $\gamma$
is at most $C_0$ (Cases 2a and 2b in \cite[Section 4]{Gru15SQ}). 
If the first possibility is true, then the claim holds for  $H=\fgen{\gens}$. 

Assume there exists a $C_0$ as in the second possibility. 
Then we have the following property by \cite[Proof of Theorem 4.2 in
Case 2a]{Gru15SQ} and \cite[Lemma 4.17]{Gru15SQ}: 
Whenever $v$ is a geodesic word representing $g^n$ for some $n$, then
the equation $v=w^n$ already holds in $G(\Gamma^{<6C_0})$, where
$\Gamma^{<6C_0}$ denotes the subgraph of $\Gamma$ that is the union of
all components with girth less than $6C_0$. 
Therefore, the epimorphism $G(\Gamma^{<6C_0})\onto G$ induced by the
identity on $\gens$ restricts to an isometry on the cyclic subgroup
generated by the element of $G(\Gamma^{<6C_0})$ represented by $w$. 
The graph $\Gamma^{<6C_0}$ is (up to identifying isomorphic
components) a finite $Gr'(\ifrac{1}{6})$--labelled graph, so
$G(\Gamma^{<6C_0})$ is a hyperbolic group. 
\end{proof}

\subsection*{A result on translation lengths} 
We record another application of our investigations.
We show:

\begin{theorem}\label{thm:translationlengths}
  Let $\Gamma$ be a $Gr'(\ifrac{1}{6})$--labelled graph whose components
are finite, labelled by a finite set $\gens$. 
Then every infinite order element of $G(\Gamma)$ has rational translation length, and translation lengths are bounded away from zero.
\end{theorem}

Recall that similar theorems are true for hyperbolic groups
\cite{Gro87,Swe95,MR1390660} and for the action of the
mapping class group of a surface on its curve complex \cite{Bow08}.

The rationality statement is a direct consequence of \fullref{lem:periodicgeodesic}. We prove the remaining statement:

\begin{proposition}\label{lem:translationlengthawayfromzero} 
Let $\Gamma$ be a $Gr'(\ifrac{1}{6})$--labelled graph whose components
are finite, labelled by a finite set $\gens$. 
Every infinite order element $x$ of $G:=G(\Gamma)$ has $\tau_G(x)\geqslant \ifrac{1}{3}$.
\end{proposition}

\begin{proof}
Let $x$ be an infinite order element of $G$. 
By \cite[Section 4]{Gru15SQ}, there exists $g\in G$ such that $x$ is
conjugate to $g^N$ for some $N$ and such that, if $w$ is a shortest
word representing $g$, we have the following possibilities: 
Any bi-infinite periodic path graph $\gamma$ labelled by the powers of $w$ is
a convex geodesic (Case 1 in \cite[Section 4]{Gru15SQ}), or, for each
$n$ and any shortest word $g_n$ representing $g^n$, there exists a
diagram $B_n$ over $\Gamma$ whose boundary word is $g_nw^{-n}$, such
that $B_n$ is a combinatorial geodesic bigon 
with respect to the obvious decomposition of $\partial B_n$ (Cases 2a, 2b in \cite[Section 4]{Gru15SQ}). 
In particular, every disk component of $B_n$ is a single face, or has shape $\mathrm{I}_1$.

If the first possibility is true, then $\tau_G(g)=|w|\geqslant 1$. 
Now consider the second possibility. 
In Case 2a, it follows from \cite[Proof of Theorem 4.2 in Case
2a]{Gru15SQ} that any face $\Pi$ of $B_n$, the length of the
intersection of $\Pi$ with the side of $B_n$ corresponding to $g_n$
(the bottom) is more than $\ifrac{1}{3}$ times the length of its
intersection with the side corresponding to $w^n$ (the top). 
Therefore, in this case, $|g_n|>\ifrac{n|w|}{3}\geqslant \ifrac{n}{3}$. 
In Case 2b, if a face $\Pi$  intersects the top in at most
$\ifrac{|\partial\Pi|}{2}$, then it intersects the bottom in more than
$\ifrac{|\partial\Pi|}{6}$ by the small cancellation condition. 
If $\Pi$ intersects the top in more than $\ifrac{|\partial\Pi|}{2}$,
then \cite[Lemma 4.8]{Gru15SQ} implies that the intersection of $\Pi$
with the top has length less than $2|w|$. 
The intersection with the bottom has length at least 1, since every
disk component of $B_n$ has shape $\mathrm{I}_1$. 
Therefore, we have $|g_n|>
n|w|\cdot\min\{\frac{1}{3},\frac{1}{2|w|}\}\geqslant \frac{n}{3}$. 
We conclude $\tau_G(g)\geqslant \ifrac{1}{3}$. 
Hence, $\tau_G(x)=\tau_G(g^N)= |N|\tau_G(g)\geqslant
\ifrac{N}{3}\geqslant \ifrac{1}{3}$.
\end{proof}

\subsection{The coned-off space}\label{sec:conedoff}
In \cite{GruSis14}, Gruber and Sisto prove that non-elementary
groups defined by $Gr(7)$-labelled graphs, which, in particular, includes those defined by $Gr'(\ifrac{1}{6})$-labelled graphs, are
acylindrically hyperbolic. 
They prove this result by studying the action of $G(\Gamma)$ on what we call the \emph{coned-off space} $Y$ defined as follows: given a graph $\Gamma$ labelled by $\gens$, let $\mathcal W$ denote the set of all elements of $G(\Gamma)$ represented by words read on (not necessarily closed!) paths in $\Gamma$. We set $Y:=\Cay(G(\Gamma),\gens\cup\mathcal W)$. Thus, we obtain $Y$ from $X:=\Cay(G(\Gamma),\gens)$ by attaching to every embedded component of $\Gamma$ in $X$ a complete graph. 

The proof in \cite{GruSis14} shows hyperbolicity of the space $Y$ and existence of an element of $G(\Gamma)$ whose action on $Y$ is hyperbolic and weakly properly discontinuous (WPD).
By a theorem of Osin \cite{Osi13}, this yields acylindrical hyperbolicity.

\begin{lem}\label{lem:hyperbolic_iff_bounded_intersection} 
Let $\Gamma$ be a $Gr'(\ifrac{1}{6})$--labelled graph.
Let $g$ be
  an infinite-order element of $G:=G(\Gamma)$.
Let $X:=\Cay(G,\gens)$, and let $Y$ be the coned-off space. 
Let $\gamma$ be a bi-infinite geodesic in $X$ that is at finite Hausdorff distance (in $X$) from $\fgen{g}$. 
Then $g$ is hyperbolic for the
  action of $G$ on $Y$  if and only if 
there exists $C>0$ such that for every embedded component $\Gamma_0$ of $\Gamma$ in $X$ we have $\mathrm{\diam_X}(\gamma\cap\Gamma_0)<C$.
\end{lem}
Combining this with \fullref{lem:strcontbddHdist} and \fullref{theorem:strongcontraction}, we obtain the following.

\begin{corollary}\label{corollary:hyperbolicimpliesstrongcontraction}
  If $\fgen{g}$ is within bounded Hausdorff distance of a bi-infinite geodesic in $X$ and acts hyperbolically on $Y$, then $\fgen{g}$ is strongly contracting in $X$.
\end{corollary}

\begin{remark}\label{newstrongcontractingelements}
The converse of \fullref{corollary:hyperbolicimpliesstrongcontraction} fails.
Consider the group defined by $\sqcup_{i\in \N}\Gamma_i$ constructed in \fullref{rem:unboundedproj}. 
The subgroup $\fgen{a}$ is strongly contracting, but in the coned-off
space $Y$ its orbit has diameter at most $2$, so $a$ does not act hyperbolically on $Y$. 
Thus, the methods of \cite{GruSis14} do not detect that $\fgen{a}$ is
strongly contracting.
\end{remark}

\begin{proof}[{Proof of \fullref{lem:hyperbolic_iff_bounded_intersection} }]
 Let $\gamma$ be a bi-infinite geodesic in $X$ whose
 Hausdorff distance from $\fgen{g}$ is equal to $\epsilon<\infty$.

 Suppose for every $n\in\mathbb{N}$ there exists an embedded component
 $\Gamma_n$ of $\Gamma$ in $X$ such that
 $\mathrm{diam}_X(\Gamma_n\cap\gamma)\geqslant n$. Denote by
 $\gamma_n$ the path graph $\Gamma_n\cap\gamma$. (This is a connected set since each
 $\Gamma_n$ is convex by \cite[Lemma~2.15]{GruSis14}.) Let
 $\iota\gamma_n$ and $\tau\gamma_n$ denote, respectively, the initial and terminal
 vertices of $\gamma_n$.
Then, by assumption, for each $n$, there exist $m_n$ and $l_n$ such
that $d_X(\iota\gamma_n,g^{m_n})<\epsilon$ and
$d_X(\tau\gamma_n,g^{l_n})<\epsilon$ and, since $|\gamma_n|\to\infty$,
we have $|m_n-l_n|\to\infty$. We have
$d_Y(g^{m_n},g^{l_n})<2\epsilon+1$, since the vertex set of $\Gamma_n$
has diameter at most 1 in the metric of $Y$. Therefore, the map
$\mathbb{Z}\to Y\from z\mapsto g^z$ is not a quasi-isometric embedding.
 
 On the other hand, suppose there exists $C$ such that, for every
 embedded component $\Gamma_0$ of $\Gamma$, we have
 $\diam_X(\gamma\cap\Gamma_0)<C$. Then, by
 \cite[Proposition~3.6]{GruSis14}, we have, for any $k$ and $l$, that
 $d_Y(g^k,g^l)>\frac{1}{C}(d_X(g^k,g^l)-2\epsilon)-2\epsilon$. Since
 $\langle g\rangle$ is undistorted in $G$ by
 \cite[Theorem~4.2]{Gru15SQ}, this gives the lower quasi-isometry
 bound for the map $\mathbb{Z}\to Y\from z\mapsto g^z$. Since this map is obviously Lipschitz, it is, in fact, a quasi-isometric embedding, whence $g$ is hyperbolic. 
\end{proof}

We will use the following result of \cite{GruSis14}:

\begin{proposition}[{\cite[Section~4]{GruSis14}}]\label{prop:existenceofhyperbolicelement} Let $\Gamma$ be a $Gr'(\ifrac{1}{6})$-labelled graph. Suppose that $\Gamma$ has at least two non-isomorphic components that each contain an embedded cycle of length at least 2. Then $G(\Gamma)$ contains a hyperbolic element $g$ for the action of $G(\Gamma)$ on the coned-off space $Y$.
\end{proposition}

\subsection{Proof of \fullref{thm:existenceofstronglycontractingelement}}
If $\Gamma$ has only finitely many pairwise non-isomorphic components with non-trivial fundamental groups, then $G(\Gamma)$ is Gromov hyperbolic \cite{Oll06,Gru15} and, hence, the result holds.

If $\Gamma$ has infinitely many pairwise non-isomorphic components with non-trivial fundamental groups, then, since $\gens$ is finite, there exist at least two such components each containing an embedded cycle of length at least 2. Therefore, by \fullref{prop:existenceofhyperbolicelement}, there exists a hyperbolic element $g$ for the action of $G(\Gamma)$ on $Y$. Since the components of $\Gamma$ are finite, \fullref{lem:periodicgeodesic} yields that $\langle g\rangle$ is within bounded Hausdorff distance of a bi-infinite geodesic  in $X$. Therefore, \fullref{corollary:hyperbolicimpliesstrongcontraction} implies the result. \hfill $\square$


\section{Hyperbolically embedded subgroups}\label{sec:hypembsubgroups}
\subsection{Contracting subgroups and hyperbolically embedded subgroups}

We recall a definition of hyperbolically embedded subgroups from \cite{DGO11}.
\begin{defi}[Hyperbolically embedded subgroup]\label{defi:hyperbolicallyembedded} Let $G$ be a group and $H$ a subgroup. Then $H$ is \emph{hyperbolically embedded} in $G$ if there exists a subset $\gens$ of $G$ with the  properties below. We denote by $X^H$ the Cayley graph of $H$ with respect to $H$, considered as subgraph of $X:=\Cay(G,H\sqcup \gens)$, and by $\hat d$ the metric on $H$ obtained as follows: we define $\hat d(x,y)$ to be the length of a shortest path from $x$ to $y$ in $X$ that does not use any edges of $X^H$. If no such path exists, set $\hat d(x,y)=\infty$. The required properties are:
\begin{itemize}
\item $H\sqcup \gens$ generates $G$ (i.e. $X$ is connected).
\item $X$ is Gromov hyperbolic.
\item The metric $\hat d$ is proper on $H$, i.e. $\hat d$--balls are finite.
\end{itemize}
\end{defi}
Note that if $s\in H\cap \gens$, then it is our convention that $\Cay(G,H\sqcup \gens)$ will have a double-edge corresponding to $s$ (once considered as element of $H$ and once considered as element of $\gens$); hence the symbol $\sqcup$ for disjoint union.

One approach to finding hyperbolically embedded subgroups is provided by \cite[Theorem~H]{BesBroFuj15}.
A special case of this theorem states that if $G$ is a finitely generated group and a hyperbolic element $g\in G$ has a strongly contracting orbit in a Cayley graph of $G$, then the elementary closure $E(g)$ of $g$ is an infinite, virtually cyclic, hyperbolically embedded subgroup of $G$.

Combining this with \fullref{lem:strcontbddHdist} and \fullref{theorem:strongcontraction}, we obtain the following.

\begin{thm}\label{thm:hypembsubgps} 
Let $G:=G(\Gamma)$ be the group defined by a $Gr'(\ifrac{1}{6})$--labelled graph $\Gamma$.
Let $g\in G$ be of infinite order, and let $\gamma$ be a bi-infinite geodesic at finite Hausdorff distance from $\fgen{g}$.
If $\gamma$ is uniformly locally strongly contracting, then $\langle g\rangle$ is strongly contracting and, in particular, the elementary closure $E(g)$ of $g$ is a virtually cyclic hyperbolically embedded subgroup.
\end{thm}

\begin{remark}\label{remark:acylindricalhyperbolicity} By \cite[Theorem~H]{BesBroFuj15}, we can also view \fullref{thm:existenceofstronglycontractingelement} as an alternative proof of acylindrical hyperbolicity for the groups considered in that theorem. The initial proof of acylindrical hyperbolicity of these groups in \cite{GruSis14} relies on the hyperbolicity of the coned-off space and on showing that the element $g$ satisfies a certain weak proper discontinuity condition (and, in fact, applies to a larger class of groups). Our \fullref{thm:hypembsubgps}  gives an alternative proof that the element $g$ from \cite{GruSis14} gives rise to a hyperbolically embedded virtually cyclic subgroup.
\end{remark}

Every hyperbolically embedded subgroup of a finitely generated group is Morse \cite{SistoMorse}, and every element with a strongly contracting orbit has a virtually cyclic hyperbolically embedded elementary closure. 
In graphical small cancellation groups, Morse and strongly contracting elements have been classified in terms of the defining graph, by our results in \fullref{sec:gsccontraction}. 
It is natural to ask whether the collection of cyclic hyperbolically embedded subgroups can be classified in a similar way.
We give one negative result in this direction.

\begin{thm}\label{theorem:unbddrhocontnotHE} 
Let $\rho_2$ be an unbounded sublinear function. 
There exists a $Gr'(\ifrac{1}{6})$--labelled graph $\Gamma$ with set of labels $\gens:=\{a,b\}$ whose components are all cycles such that the group $G:=G(\Gamma)$ has the following properties: Any virtually cyclic subgroup $E$ of $G$ containing $\langle a\rangle$ is $(r,\rho'_2)$--contracting in the Cayley graph $X:=\Cay(G,\gens)$ for some $\rho'_2\asymp \rho_2$, but $E$ is not hyperbolically embedded in $G$.
\end{thm}

\begin{proof} 
For every $r>0$, choose $N_r\geqslant 6$ such that $\rho_2\bigl(\frac{1}{2}(1+r)N_r\bigr)\geqslant r$. 
This is possible because $\rho_2$ is unbounded. For every $r>0$, let $R_r:=(a^rba^rb^{-1})^{N_r}$, and consider the graph $\Gamma$ that is the disjoint union of cycles $\gamma_r$ labelled by the $R_r$. Any reduced path that is a piece in $\gamma_r$ has length at most $2r+1$, which is less than $\ifrac{|\gamma_r|}{6}$. 
Therefore, $\Gamma$ satisfies the $Gr'(\ifrac{1}{6})$-condition. 
Denote by $\alpha$ a bi-infinite path in $X$ labelled by the powers of $a$. 
For every copy $\gamma$ of a $\gamma_r$ embedded in $X$, we have $\diam\alpha\cap\gamma<\ifrac{|\gamma|}{6}$. 
Therefore, it is readily seen from considering diagrams of shape $\mathrm{I}_1$ that $\alpha$ is a bi-infinite geodesic. 
Also note that $\gamma$ intersects $\alpha$ in a path of length at most $r$.

Suppose $\gamma$ is a relator intersecting $\alpha$, and $x$ is a point in $\gamma$ with $\delta:=d(x,\alpha)$. 
If $|\gamma|>4\delta$, then, for any point $y$ in $\gamma$ with $d(x,y)\leqslant \delta$, we have $\pi(y)=\pi(x)$, and $\pi(x)$ is a singleton, where $\pi$ denotes closest point projection to $\alpha$. 
Now assume $|\gamma|\leqslant 4\delta$, and $y\in \gamma$. 
Then $\diam\pi(x)\cup\pi(y)\leqslant \diam\gamma\cap\alpha\leqslant \rho_2(\delta)$, by construction, since $\rho_2$ is non-decreasing.
Therefore, $\alpha$ is locally $(r,\rho_2)$-contracting and, hence, by \fullref{theorem:strongcontraction} and \fullref{lem:strcontbddHdist} there exists $\rho_2'$ with $\rho_2'\asymp \rho_2$ such that any virtually cyclic subgroup $E$ of $G$ containing $\langle a\rangle$ is $(r,\rho_2')$-contracting.

If $E$ is a hyperbolically embedded subgroup then, by \cite[Theorem 5.3]{DGO11}, there exists $r>0$ such that the normal closure $N$ of $a^r$ is a free group. By construction, the element of $G$ represented by $a^rba^rb^{-1}$ is non-trivial and has finite order, whence $N$ is not torsion-free. 
Therefore, $E$ is not hyperbolically embedded.
\end{proof}

Another natural question is the following.

\begin{question} Let $G$ be a group generated by a finite set $\gens$, and suppose $g\in G$ has a $(\rho_1,\rho_2)$--contracting orbit in $X:=\Cay(G,\gens)$ where $\rho_2$ is bounded. 
Is the elementary closure $E(g)$ a hyperbolically embedded virtually cyclic subgroup of $G$?

Is the statement true if $G:=G(\Gamma)$ for a $Gr'(\ifrac{1}{6})$--labelled graph $\Gamma$ whose components are finite, with finite set of labels $\gens$? 
\end{question}

In the case of classical $C'(\ifrac{1}{6})$-groups, an affirmative answer follows from \fullref{thm:weakimpliesstrong}.


\subsection{Hyperbolically embedded cycles}\label{sec:hypembcycle}
Recall that a group has the \emph{hyperbolically embedded cycles property (HEC property)} if
the elementary closure $E(g)$ of every infinite order element $g$ is virtually cyclic and hyperbolically embedded.
Hyperbolic groups have this property: $E(g)$ is the stabilizer of a $g$-axis.
The HEC property in fact passes to any subgroup of a hyperbolic
group $G$ using the action of the subgroup on a Cayley graph of $G$.

It is therefore very natural to ask whether this classifies subgroups of hyperbolic groups. 
Torsion presents one complication:
A free product of infinite torsion groups, or, more generally, a group hyperbolic relative to infinite torsion subgroups, has the HEC property but cannot be a subgroup of a hyperbolic group.
We show that even among torsion-free groups there are many examples of groups with the HEC property that are not subgroups of any hyperbolic group.

\begin{thm}\label{thm:allmaxsubgrpsHE} There exist $2^{\aleph_0}$
  pairwise non-quasi-isometric finitely generated torsion-free groups
  in which every non-trivial cyclic subgroup is strongly contracting and which, therefore, have the HEC property.
\end{thm}

Since there are only countably many finitely generated subgroups of
finitely presented groups, most of the groups of \fullref{thm:allmaxsubgrpsHE} do not occur as
subgroups of hyperbolic groups.

We show in \fullref{corollary:HECmonster}
that there are even exotic examples of groups with the HEC
property such as the Gromov monster groups.

The theorem is proven by building small cancellation groups
in which no power of any element has long intersection with an embedded
component of the defining graph. 
In \fullref{thm:scandnonrep} we give a condition  on the labelling
that guarantees this property. 
In \fullref{thm:osajdaalon} we show that this condition can be
satisfied.
Then we construct specific examples satisfying our condition and apply
a version
of a construction of Thomas and Velickovic \cite{ThoVel00}, proven in
\fullref{prop:lotsofgroups}, to show that we get uncountably many
quasi-isometry classes of groups.

A labelling of the edges of an undirected graph is said to be \emph{non-repetitive}
if there does not exist a non-trivial embedded path graph that is labelled by a word
of the form $ww$. (Here, the label of a path is just the concatenation of the labels of the edges in the free monoid on the labelling set.)
The \emph{Thue number} of a graph is the minimal cardinality of a
labelling set for which the graph admits a non-repetitive labelling. 

We define a labelling of a directed graph to be \emph{non-repetitive}
if there does not exist a non-trivial embedded path graph
$\gamma=e_1,\dots,e_{2n}$ such that for all $1\leqslant i\leqslant n$
the label of the directed edge $e_i$ is equal to the label of the
directed edge $e_{i+n}$. 
 Note that, given an undirected graph with a non-repetitive labelling, any choice of orientation gives rise to a non-repetitive labelling of the resulting oriented graph.

\begin{theorem}\label{thm:scandnonrep}
  Let $\gens$ be a finite set and let $\Gamma$ be a graph with finite
  components and a labelling by  $\gens$ that is both $Gr'(\ifrac{1}{6})$ and
  non-repetitive.
Let $G:=G(\Gamma)$.
Every infinite cyclic subgroup $H$ of $G$ is strongly
contracting in $X:=\Cay(G,\gens)$.
Thus, $G$ has the HEC property.
\end{theorem}
\begin{proof} 
By \fullref{lem:periodicgeodesic}, every infinite cyclic subgroup
is bounded Hausdorff distance from a periodic geodesic $\alpha$.
By \fullref{lem:strcontbddHdist}, it is enough to show $\alpha$ is
strongly contracting, so we may assume that $H=\fgen{h}$ acts
cocompactly on $\alpha$, and that $h$ is represented by a cyclically reduced word $v$ whose powers label $\alpha$.

Suppose $\alpha$ intersects some embedded component $\Gamma_0$ in more
than a single vertex.
Since $\alpha$ is geodesic, $\Gamma_0\cap\alpha$ is an embedded
path graph.
Since the labelling of $\Gamma$ is non-repetitive, the label of this
path does not contain a subword of the form $ww$.
However, the labelling of $\alpha$ is repetitive --- it is $\cdots vv\cdots$.
Thus, $|\Gamma_0\cap\alpha|<2|v|$.
We conclude
 $\alpha$ is strongly contracting by \fullref{theorem:strongcontraction}.

By \cite[Theorem~H]{BesBroFuj15}, the elementary closure of $H$ is a virtually cyclic hyperbolically embedded subgroup.
\end{proof}

Suppose that $\Gamma$ is a directed graph with two labellings $L_1\from
\mathcal{E}\Gamma\to\gens_1$ and $L_2\from \mathcal{E}\Gamma\to\gens_2$.
We define the \emph{push-out labelling} $L\from
\mathcal{E}\Gamma\to\gens$, where $\gens=\gens_1\times\gens_2$, by
$L(e):=(L_1(e),L_2(e))$.
We will write $\gens_1$ alphabetically and $\gens_2$ as numerical index, i.e.\
$\gens^{\pm}=\{a^\epsilon_n\mid \epsilon\in\pm 1,\, a\in\gens_1,\,n\in\gens_2\}$.

If $L_2\from \mathcal{E}\Gamma\to\gens_2$ is a non-repetitive
labelling of $\Gamma$, then for any labelling $L_1\from
\mathcal{E}\Gamma\to\gens_1$ the push-out labelling is a
non-repetitive labelling of $\Gamma$.
Similarly, if $L_1\from
\mathcal{E}\Gamma\to\gens_1$ satisfies  
the $C'(\ifrac{1}{6})$-condition, then so does the
push-out labelling. 

\begin{theorem}\label{thm:osajdaalon}
  Let $\Gamma=(\Gamma_i)_{i\in\N}$ be a sequence of finite, connected graphs satisfying
  the following conditions:
  \begin{itemize}
  \item $\Gamma$ has bounded valence.
\item $(\mathrm{girth}(\Gamma_i))_{i\in\N}$ is an unbounded sequence.
\item The ratios $\frac{\mathrm{girth}(\Gamma_i)}{\diam(\Gamma_i)}$ are
  bounded, uniformly over $i$, away from 0. 
  \end{itemize}
Then there exist an infinite subsequence $(\Gamma_{i_j})_{j\in\N}$ of
graphs and a finite set $\gens$ such that $(\Gamma_{i_j})_{j\in\N}$
admits an labelling by $\gens$ that is both $C'(\ifrac{1}{6})$ and non-repetitive.
\end{theorem}
\begin{proof}
  A theorem of Osajda \cite{Osa} says that, given the hypotheses on $\Gamma$, there exists an infinite subsequence $(\Gamma_{i_j})_{j\in\N}$ that admits a choice of orientation and a labelling by a finite set $\gens_1$ satisfying the $C'(\ifrac{1}{6})$-condition. 
Alon, et al.\ \cite{AloGryHau02} show that the Thue number of a bounded valence undirected graph is bounded by a polynomial function of the valence bound.
Since $\Gamma$ has bounded valence, there exists a finite set of labels $\gens_2$ such that $\Gamma$ admits a non-repetitive $\gens_2$--labelling (as undirected graph and, hence, also as directed graph).
The push-out of these two labellings satisfies the theorem.
\end{proof}

Combining this with \fullref{thm:scandnonrep}, we have:
\begin{corollary}
  If $\Gamma$ is as in \fullref{thm:osajdaalon} then $G(\Gamma)$ has
  the HEC property.
\end{corollary}
\begin{corollary}\label{corollary:HECmonster}
  There exist Gromov monster groups with the HEC property.
\end{corollary}

More generally, if $\Gamma$ is a bounded valence graph with a
labelling satisfying some property $\mathcal{P}$ such that
$\mathcal{P}$ is preserved upon passing to a refinement of the
labelling, then $\Gamma$ admits a labelling that is both $\mathcal{P}$
and non-repetitive. 
For instance, Arzhantseva and
Osajda \cite{ArzOsa14}
introduced a `lacunary walling condition' to produce first examples of
non-coarsely amenable groups\footnote{These are groups $G$ whose
  reduced $C^*$--algebra $C^*_{red}(G)$ is not exact.} with
the Haagerup property.
This condition is preserved upon passing to refinements of the
labelling, so the same argument as in \fullref{thm:osajdaalon}
yields:
\begin{corollary}
  There exist finitely generated, non-coarsely amenable groups with the Haagerup property
  and the HEC property.
\end{corollary}

Not every interesting property of labellings is preserved upon passing
to refinements. 
For example, the $Gr'(\lambda)$ condition is not preserved, because
in the original labelling there may be a long labelled path $p$ with distinct
label-preserving maps $\phi_1,\,\phi_2\from p\to\Gamma$ and a
label-preserving automorphism $\psi$ of $\Gamma$ such that
$\phi_2=\psi\circ\phi_1$. 
If $\psi$ fails to be a label-preserving automorphism of $\Gamma$ with
the refined labelling then $p$ may be too long a piece.

For another example, passing to a refinement can yield a group with
non-trivial free factors.
In particular, no property of a labelling
that implies the resulting group has Kazhdan's Property (T) is
preserved by passing to refinements.

Both of the factor labellings in the proof of \fullref{thm:osajdaalon}
are produced probabilistically using the Lov\'asz Local Lemma.
For the purpose of proving \fullref{thm:allmaxsubgrpsHE} we 
construct \emph{explicit} examples of labelled graphs satisfying
\fullref{thm:scandnonrep}, so we get concrete examples of groups
satisfying \fullref{thm:allmaxsubgrpsHE}.
To construct such examples we use the fact, first observed by Alon, et
al.\ \cite{AloGryHau02}, that a cycle graph has Thue number at most 4,
as follows:

\begin{defi}[Thue-Morse sequence]\label{defi:Thue-Morse} Define
  $\sigma\from \{0,1\}^*\to \{0,1\}^*$ by $\sigma(0):=01$, $\sigma(1):=10$.
The sequence $(x_i)_{i\in\N}:=\sigma^\infty(0)$ is called the \emph{Thue-Morse
  sequence}\footnote{Sequence \href{http://oeis.org/A010060}{A010060} of \cite{OEIS}.}. 
\end{defi}

The Thue-Morse sequence famously does not contain any subword of the
form $www$. 
The `first difference' sequence\footnote{Sequence \href{http://oeis.org/A029883}{A029883} of \cite{OEIS}.}
$(y_i)_{i\in\N}\in\{-1,0,1\}^\N$ defined by $y_i:=x_{i+1}-x_i$, where
$(x_i)_{i\in\N}$ is the Thue-Morse sequence, does not contain any
subword of the form $ww$, which means that it gives a non-repetitive
labelling of the ray graph.
Both of these facts are due to Thue \cite{zbMATH02624830}.

A cycle graph of length $n$ with edges $e_0,e_1,\dots,e_{n-1}$ admits a non-repetitive labelling by
$\{-1,0,1,\infty\}$ by labelling edge $e_0$ with $\infty$ and
labelling edge $e_i$ for $i>0$ with term $y_i$ of the
first difference of the Thue-Morse sequence.

Define $\Gamma:=(\Gamma_n)_{n\in\N}$ to be a disjoint union of cycle
graphs such that $|\Gamma_n|=11(44n-19)$.
Give each of these cycles the non-repetitive
$\{-1,0,1,\infty\}$--labelling defined above.
Also give $\Gamma_n$ the $\{a,b\}$--labelling $R_n:=\prod_{i=22n-21}^{22n}ab^i$.
This is a $C'(\ifrac{1}{6})$--labelling.
There are no label-preserving automorphisms of $\Gamma$ since the
components are cycles of different lengths and are labelled by
positive words that are not proper powers.
If $p$ is a piece
contained in $\Gamma_n$ then, since the gaps between $a$'s in the
$R_i$ are all
different, $p$ contains at most one edge labelled $a$.
The longest subword of $R_n$ containing at most one $a$ is $b^{22n-1}ab^{22n}$ of length $44n$.
For $n\in\mathbb{N}$, the ratio $\frac{44n}{11(44n-19)}$ takes maximum
value $\ifrac{4}{25}<\ifrac{1}{6}$ at $n=1$.

The push-out of these two labellings is an $\gens$--labelling of
$\Gamma$ that is both non-repetitive and $C'(\ifrac{1}{6})$, with $\gens:=\{a_{-1},a_0,a_1,a_\infty,b_{-1},b_0,b_1,b_\infty\}$.

For each subset $I\subset\N$ define $\Gamma_I:=(\Gamma_n)_{n\in I}$.
All of these graphs have non-repetitive, $C'(\ifrac{1}{6})$--labellings inherited from $\Gamma$.
By \fullref{thm:scandnonrep}, every non-trivial cyclic subgroup of the group $G_I$ defined by $\Gamma_I$
is strongly contracting.
The proof of \fullref{thm:allmaxsubgrpsHE} is completed by the
following proposition and the fact that groups defined by $C'(\ifrac{1}{6})$-labelled graphs are torsion-free \cite{Gru15}.

\begin{prop}[{\cite{ThoVel00}}]\label{prop:lotsofgroups}
There is a subset $\mathcal{I}\subseteq 2^{\N}$ of cardinality
$2^{\aleph_0}$ such that given $I,J\in\mathcal{I}$, the groups $G_I$
and $G_J$ are quasi-isometric if and only if $I=J$.
\end{prop}
\begin{proof} Choose $\mathcal{I}$ to be a collection of infinite subsets of
  $\{2^{2^n}\mid n\in\N\}$ with infinite pairwise symmetric difference.

Let $I,J\in\mathcal{I}$ be distinct.
Without loss of generality, assume $I\setminus J$ is infinite, and let $\mu$ be a non-atomic ultrafilter on $\N$ with $\mu(I)=1$ and $\mu(J)=0$.

The asymptotic cone of $G_J$ over $\mu$ with scaling sequence
$(|R_n|)_{n\in I}$ is an $\R$-tree, while the asymptotic cone of $G_I$ over $\mu$ with the same scaling contains a loop of length $1$.

If $G_I$ and $G_J$ are quasi-isometric then their asymptotic cones are
bi-Lipschitz equivalent. 
Thus these groups are not quasi-isometric.
\end{proof}




\bibliographystyle{hypershort}
\bibliography{scsc}

\providecommand{\bysame}{\leavevmode\hbox to3em{\hrulefill}\thinspace}
\providecommand{\MR}{\relax\ifhmode\unskip\space\fi MR }
\providecommand{\MRhref}[2]{%
  \href{http://www.ams.org/mathscinet-getitem?mr=#1}{#2}
}
\providecommand{\href}[2]{#2}
\providecommand{\doi}[1]{doi: #1}
\begin{thebibliography}{10}

\bibitem{OEIS}
\emph{\href{http://oeis.org}{The on-line encyclopedia of integer sequences}},
  \href{http://oeis.org}{{\texttt{\detokenize{http://oeis.org}}}}.

\bibitem{Alg11}
Yael Algom-Kfir,
  \emph{\href{http://dx.doi.org/10.2140/gt.2011.15.2181}{Strongly contracting
  geodesics in outer space}}, Geom. Topol. \textbf{15} (2011), no.~4,
  2181--2233.

\bibitem{AloGryHau02}
Noga Alon, Jaros{\l}aw Grytczuk, Mariusz Ha{\l}uszczak, and Oliver Riordan,
  \emph{\href{http://dx.doi.org/10.1002/rsa.10057}{Nonrepetitive colorings of
  graphs}}, Random Structures \& Algorithms \textbf{21} (2002), no.~3-4,
  336--346.

\bibitem{ArzOsa14}
Goulnara Arzhantseva and Damian Osajda,
  \emph{\href{http://arxiv.org/abs/1404.6807}{Graphical small cancellation
  groups with the {H}aagerup property}}, preprint (2014),
  \href{http://arXiv.org/abs/1404.6807v2}{{\texttt{arXiv:1404.6807v2}}}.

\bibitem{ArzSte14}
Goulnara Arzhantseva and Markus Steenbock,
  \emph{\href{http://arxiv.org/pdf/1407.2441v1.pdf}{Rips construction without
  unique product}}, preprint (2014),
  \href{http://arXiv.org/abs/1407.2441}{{\texttt{arXiv:1407.2441}}}.

\bibitem{ArzTes16}
Goulnara Arzhantseva and Romain Tessera,
  \emph{\href{http://arxiv.org/pdf/1605.01192v2.pdf}{Admitting a coarse
  embedding is not preserved under group extensions}}, Int. Math. Res. Not.
  IMRN (in press).

\bibitem{ArzCasGrub}
Goulnara~N. Arzhantseva, Christopher~H. Cashen, Dominik Gruber, and David Hume,
  \emph{\href{https://www.math.uni-bielefeld.de/documenta/vol-22/36.html}{Characterizations
  of {M}orse geodesics via superlinear divergence and sublinear contraction}},
  Doc. Math. \textbf{22} (2017), 1193--1224.

\bibitem{ArzCasTao15}
Goulnara~N. Arzhantseva, Christopher~H. Cashen, and Jing Tao,
  \emph{\href{http://dx.doi.org/10.2140/pjm.2015.278.1}{Growth tight actions}},
  Pacific J. Math. \textbf{278} (2015), no.~1, 1--49.

\bibitem{ArzDel08}
Goulnara~N. Arzhantseva and Thomas Delzant, \emph{Examples of random groups},
  preprint, available on the authors' webpages (2008).

\bibitem{BehDruMos09}
Jason Behrstock, Cornelia Dru\c{t}u, and Lee Mosher, \emph{Thick metric spaces,
  relative hyperbolicity, and quasi-isometric rigidity}, Math. Ann.
  \textbf{344} (2009), no.~3, 543--595.

\bibitem{BehDru14}
Jason Behrstock and Cornelia Dru{\c{t}}u,
  \emph{\href{http://projecteuclid.org/euclid.ijm/1446819294}{Divergence, thick
  groups, and short conjugators}}, Illinois J. Math. \textbf{58} (2014), no.~4,
  939--980.

\bibitem{Beh06}
Jason~A. Behrstock,
  \emph{\href{http://dx.doi.org/10.2140/gt.2006.10.1523}{Asymptotic geometry of
  the mapping class group and {T}eichm\"uller space}}, Geom. Topol. \textbf{10}
  (2006), 1523--1578.

\bibitem{BesBroFuj15}
Mladen Bestvina, Kenneth Bromberg, and Koji Fujiwara,
  \emph{\href{http://dx.doi.org/10.1007/s10240-014-0067-4}{Constructing group
  actions on quasi-trees and applications to mapping class groups}}, Publ.
  Math. Inst. Hautes \'Etudes Sci. \textbf{122} (2015), no.~1, 1--64.

\bibitem{BesFuj09}
Mladen Bestvina and Koji Fujiwara,
  \emph{\href{http://dx.doi.org/10.1007/s00039-009-0717-8}{A characterization
  of higher rank symmetric spaces via bounded cohomology}}, Geom. Funct. Anal.
  \textbf{19} (2009), no.~1, 11--40.

\bibitem{Bow08}
Brian~H. Bowditch,
  \emph{\href{http://dx.doi.org/10.1007/s00222-007-0081-y}{Tight geodesics in
  the curve complex}}, Invent. Math. \textbf{171} (2008), no.~2, 281--300.

\bibitem{Con97}
Gregory~R. Conner, \emph{\href{http://dx.doi.org/10.1007/s000130050120}{A class
  of finitely generated groups with irrational translation numbers}}, Arch.
  Math. (Basel) \textbf{69} (1997), no.~4, 265--274.

\bibitem{Con00}
Gregory~R. Conner,
  \emph{\href{http://dx.doi.org/10.1017/CBO9780511600609.003}{Translation
  numbers of groups acting on quasiconvex spaces}}, Computational and geometric
  aspects of modern algebra ({E}dinburgh, 1998), London Math. Soc. Lecture Note
  Ser., vol. 275, Cambridge Univ. Press, Cambridge, 2000, pp.~28--38.

\bibitem{CuSh12}
R{\'e}mi Cun{\'e}o and Hamish Short, \emph{Graph small cancellation theory
  applied to alternating link groups}, Journal of Knot Theory and Its
  Ramifications \textbf{21} (2012), no.~11, 1250113, 15 pages.

\bibitem{DGO11}
Fran\c{c}ois Dahmani, Vincent Guirardel, and Denis Osin,
  \emph{\href{http://dx.doi.org/10.1090/memo/1156}{Hyperbolically embedded
  subgroups and rotating families in groups acting on hyperbolic spaces}}, Mem.
  Amer. Math. Soc. \textbf{245} (2017), no.~1156, v+152.

\bibitem{MR1390660}
Thomas Delzant,
  \emph{\href{http://dx.doi.org/10.1215/S0012-7094-96-08321-0}{Sous-groupes
  distingu\'es et quotients des groupes hyperboliques}}, Duke Math. J.
  \textbf{83} (1996), no.~3, 661--682.

\bibitem{DruMozSap10}
Cornelia Dru{\c{t}}u, Shahar Mozes, and Mark Sapir,
  \emph{\href{http://dx.doi.org/10.1090/S0002-9947-09-04882-X}{Divergence in
  lattices in semisimple {L}ie groups and graphs of groups}}, Trans. Amer.
  Math. Soc. \textbf{362} (2010), no.~5, 2451--2505.

\bibitem{DruSap05}
Cornelia Dru{\c{t}}u and Mark Sapir, \emph{Tree-graded spaces and asymptotic
  cones of groups}, Topology \textbf{44} (2005), no.~5, 959--1058, With an
  appendix by D. Osin and Sapir.

\bibitem{DucRaf09}
Moon Duchin and Kasra Rafi,
  \emph{\href{http://dx.doi.org/10.1007/s00039-009-0017-3}{Divergence of
  geodesics in {T}eichm{\"u}ller space and the mapping class group}}, Geom.
  Funct. Anal. \textbf{19} (2009), no.~3, 722--742.

\bibitem{GriHar97}
Rostislav Grigorchuk and Pierre de~la Harpe,
  \emph{\href{http://dx.doi.org/10.1007/BF02471762}{On problems related to
  growth, entropy, and spectrum in group theory}}, J. Dynam. Control Systems
  \textbf{3} (1997), no.~1, 51--89.

\bibitem{Gro87}
Mikhael Gromov, \emph{Hyperbolic groups}, Essays in group theory, Math. Sci.
  Res. Inst. Publ., vol.~8, Springer, New York, 1987, pp.~75--263.

\bibitem{Gro03}
Mikhael Gromov, \emph{\href{http://dx.doi.org/10.1007/s000390300002}{Random
  walk in random groups}}, Geom. Funct. Anal. \textbf{13} (2003), no.~1,
  73--146.

\bibitem{Gru15}
Dominik Gruber,
  \emph{\href{http://dx.doi.org/10.1090/S0002-9947-2014-06198-9}{Groups with
  graphical {$C(6)$} and {$C(7)$} small cancellation presentations}}, Trans.
  Amer. Math. Soc. \textbf{367} (2015), no.~3, 2051--2078.

\bibitem{Gru15SQ}
Dominik Gruber, \emph{\href{http://dx.doi.org/10.1112/jlms/jdv022}{Infinitely
  presented {C}(6)-groups are {SQ}-universal}}, J. Lond. Math. Soc. \textbf{92}
  (2015), no.~1, 178--201.

\bibitem{GruPhD}
Dominik Gruber, \emph{Infinitely presented graphical small cancellation groups:
  Coarse embeddings, acylindrical hyperbolicity, and subgroup constructions},
  Ph.D. thesis, University of Vienna, 2015.

\bibitem{GruMarSte15}
Dominik Gruber, Alexandre Martin, and Markus Steenbock, \emph{Finite index
  subgroups without unique product in graphical small cancellation groups},
  Bull. London Math. Soc. \textbf{47} (2015), no.~4, 631--638.

\bibitem{GruSis14}
Dominik Gruber and Alessandro Sisto,
  \emph{\href{http://arxiv.org/abs/1408.4488}{Infinitely presented graphical
  small cancellation groups are acylindrically hyperbolic}}, Ann. Inst. Fourier
  (in press).

\bibitem{HigLafSka02}
Nigel Higson, Vincent Lafforgue, and George Skandalis,
  \emph{\href{http://dx.doi.org/10.1007/s00039-002-8249-5}{Counterexamples to
  the {B}aum-{C}onnes conjecture}}, Geom. Funct. Anal. \textbf{12} (2002),
  no.~2, 330--354.

\bibitem{Hume14}
David Hume, \emph{\href{http://dx.doi.org/10.4064/fm101-11-2016}{A continuum of
  expanders}}, Fund. Math. \textbf{238} (2017), no.~2, 143--152.

\bibitem{Kap97}
Ilya Kapovich,
  \emph{\href{http://dx.doi.org/10.1090/S0002-9947-97-01653-X}{Small
  cancellation groups and translation numbers}}, Trans. Amer. Math. Soc.
  \textbf{349} (1997), no.~5, 1851--1875.

\bibitem{MasMin99}
Howard~A. Masur and Yair~N. Minsky,
  \emph{\href{http://dx.doi.org/10.1007/s002220050343}{Geometry of the complex
  of curves. {I}. {H}yperbolicity}}, Invent. Math. \textbf{138} (1999), no.~1,
  103--149.

\bibitem{MasMin00}
Howard~A. Masur and Yair~N. Minsky,
  \emph{\href{http://dx.doi.org/10.1007/PL00001643}{Geometry of the complex of
  curves. {II}. {H}ierarchical structure}}, Geom. Funct. Anal. \textbf{10}
  (2000), no.~4, 902--974.

\bibitem{Min96}
Yair~N. Minsky,
  \emph{\href{http://dx.doi.org/10.1515/crll.1995.473.121}{Quasi-projections in
  {T}eichm{\"u}ller space}}, J. Reine Angew. Math. \textbf{473} (1996),
  121--136.

\bibitem{Oll06}
Yann Ollivier, \emph{\href{http://projecteuclid.org/euclid.bbms/1148059334}{On
  a small cancellation theorem of {G}romov}}, Bull. Belg. Math. Soc. Simon
  Stevin \textbf{13} (2006), no.~1, 75--89.

\bibitem{OllWis07}
Yann Ollivier and Daniel~T. Wise, \emph{Kazhdan groups with infinite outer
  automorphism group}, Trans. Amer. Math. Soc. \textbf{359} (2007), no.~5,
  1959--1976.

\bibitem{Osa}
Damian Osajda, \emph{\href{http://arxiv.org/abs/1406.5015}{Small cancellation
  labellings of some infinite graphs and applications}}, preprint (2014),
  \href{http://arXiv.org/abs/1406.5015v1}{{\texttt{arXiv:1406.5015v1}}}.

\bibitem{Osi13}
Denis~V. Osin, \emph{\href{http://dx.doi.org/10.1090/tran/6343}{Acylindrically
  hyperbolic groups}}, Trans. Amer. Math. Soc. \textbf{368} (2016), no.~2,
  851--888.

\bibitem{Pri89}
Stephen Pride, \emph{\href{http://dx.doi.org/10.1007/BFb0086251}{Some problems
  in combinatorial group theory}}, Groups---{K}orea 1988 ({P}usan, 1988),
  Lecture Notes in Math., vol. 1398, Springer, Berlin, 1989, pp.~146--155,
  \href{http://dx.doi.org/10.1007/BFb0086251}{{\texttt{\detokenize{http://dx.doi.org/10.1007/BFb0086251}}}}.

\bibitem{RipSeg87}
Eliyahu Rips and Yoav Segev, \emph{Torsion-free group without unique product
  property}, J. Algebra \textbf{108} (1987), no.~1, 116--126.

\bibitem{Sil03}
Lior Silberman, \emph{Addendum to: ``{R}andom walk in random groups'' [{G}eom.\
  {F}unct.\ {A}nal.\ {\bf 13} (2003), no.\ 1, 73--146; mr1978492] by {M}.
  {G}romov}, Geom. Funct. Anal. \textbf{13} (2003), no.~1, 147--177.

\bibitem{Sis13projection}
Alessandro Sisto, \emph{\href{http://dx.doi.org/10.4171/LEM/59-1-6}{Projections
  and relative hyperbolicity}}, Enseign. Math. (2) \textbf{59} (2013), no.~1-2,
  165--181.

\bibitem{SistoMorse}
Alessandro Sisto,
  \emph{\href{http://dx.doi.org/10.1007/s00209-016-1615-z}{Quasi-convexity of
  hyperbolically embedded subgroups}}, Math. Z. \textbf{283} (2016), no.~3-4,
  649--658.

\bibitem{Ste15}
Markus Steenbock, \emph{Rips-{S}egev torsion-free groups without the unique
  product property}, J. Algebra \textbf{438} (2015), 337--378.

\bibitem{Str90}
Ralph Strebel, \emph{Appendix. {S}mall cancellation groups}, Sur les groupes
  hyperboliques d'apr{\`e}s {M}ikhael {G}romov ({B}ern, 1988), Progr. Math.,
  vol.~83, Birkh{\"a}user Boston, Boston, MA, 1990, pp.~227--273.

\bibitem{Sul14}
Harold Sultan,
  \emph{\href{http://dx.doi.org/10.1007/s10711-013-9851-4}{Hyperbolic
  quasi-geodesics in {CAT(0)} spaces}}, Geom. Dedicata \textbf{169} (2014),
  no.~1, 209--224.

\bibitem{Swe95}
Eric~L. Swenson, \emph{\href{http://dx.doi.org/10.1007/BF01264930}{Hyperbolic
  elements in negatively curved groups}}, Geom. Dedicata \textbf{55} (1995),
  no.~2, 199--210.

\bibitem{ThoVel00}
Simon Thomas and Boban Velickovic,
  \emph{\href{http://dx.doi.org/10.1112/S0024609399006621}{Asymptotic cones of
  finitely generated groups}}, Bull. London Math. Soc. \textbf{32} (2000),
  no.~2, 203--208.

\bibitem{zbMATH02624830}
Axel Thue, \emph{{\"Uber die gegenseitige Lage gleicher Teile gewisser
  Zeichenreihen.}}, {Kristiania: J. Dybwad. 67 S. Lex. $8^\circ$}, 1912.

\end{thebibliography}

\end{document}